\documentclass[12pt,a4paper,utf8x]{article}
\usepackage{color}
\usepackage{fancyhdr}
\usepackage{tikz}
\usepackage[vmargin=70pt, hmargin=35pt]{geometry}
\usetikzlibrary{decorations.pathmorphing}
\usepackage[T1]{fontenc}
\usepackage{ucs}
\usepackage[utf8x]{inputenc}
\usepackage{amsfonts}
\usepackage{listings}
\usepackage{amsmath}
\usepackage[mathscr]{euscript}
\usepackage{amssymb}
\usepackage{amsfonts}
\usepackage{dsfont}
\usepackage{amsthm}
\usepackage{fancybox}
\usepackage{enumitem}
\usepackage{dsfont}
\usepackage{algpseudocode}
\usepackage{algorithm}
\usepackage{graphicx}
\usepackage{caption}
\usepackage{subcaption}
\usepackage{float}
\usepackage{authblk}
\usepackage{relsize}
\usepackage{setspace}
\usepackage{parskip}

\DeclareMathOperator{\LL}{L}
\DeclareMathOperator{\Tr}{Tr}

\newcommand*{\DomA}{\mathcal{D}(A)}

\everydisplay={\aftergroup\specindent}
\def\specindent{\global\hangindent=2em \global\hangafter=-1 \global\prevgraf=0 }

\makeatletter
\newcommand*{\inlineequation}[2][]{%
  \begingroup
    \refstepcounter{equation}%
    \ifx\\#1\\%
    \else
      \label{#1}%
    \fi
    \relpenalty=100 %
    \binoppenalty= 100 %
    \ensuremath{%
      #2%
    }%
    ~\@eqnnum
  \endgroup
}
\makeatother

\usepackage{hyperref}

\newenvironment{alphafootnotes}
  {\par\edef\savedfootnotenumber{\number\value{footnote}}
   
   \setcounter{footnote}{0}}
  {\par\setcounter{footnote}{\savedfootnotenumber}}

\newtheorem{mytheo}{Theorem}[section]
\newtheorem{myprop}{Proposition}[section]
\newtheorem{lemme}{Lemma}[section]

\newtheorem{remark}{Remark}[section]

\date{}

\begin{document}

\title{Recursive computation of invariant distributions of Feller processes}
\author[1]{Gilles Pag\`es}
\author[1]{Cl\'ement Rey}
\affil[1]{Universit\'e Pierre et Marie Curie, LPMA, 4 Place Jussieu, 75005 Paris, France}
\maketitle

\begin{alphafootnotes}
\footnote{e-mails : gilles.pages@upmc.fr, clement.rey@upmc.fr This research benefited from the support of the "Chaire Risques Financiers''.}
\end{alphafootnotes}

\abstract{ \noindent This paper provides a general and abstract approach to compute invariant distributions for Feller processes. More precisely, we show that the recursive algorithm presented in \cite{Lamberton_Pages_2002} and based on simulation algorithms of stochastic schemes with decreasing steps can be used to build invariant measures for general Feller processes. We also propose various applications: Approximation of Markov Brownian diffusion stationary regimes with Milstein or Euler scheme and approximation of Markov switching Brownian diffusion stationary regimes using Euler scheme.
}



\section{Introduction}

In this paper, we propose a method for the recursive computation of the invariant distribution (denoted $\nu$) of a Feller processes $(X_t)_{t \geqslant 0}$. The starting idea is to consider a non-homogeneous discrete Markov process which can be simulated using a family of transitions kernels $(Q_{\gamma})_{\gamma>0}$ and approximating $(X_t)_{t \geqslant 0}$ in a sense made precise further on. \\
As suggested by the pointwise Birkhoff ergodic theorem, we then show that some sequence $( \nu_n)_{n \in \mathbb{N}^{\ast}}$ of random empirical measures $a.s.$ weakly converges toward $\nu$ under some appropriate mean-reverting and moment assumptions. An abstract framework is developed which, among others, enables to extend this convergence to the $\mbox{L}^p$-Wasserstein distance. For a given $f$, $\nu_n(f)$ can be recursively defined making its computation straightforward.\\

\noindent Invariant distributions are crucial in the study of the long term behavior of stochastic differential systems. We invite the reader to refer to \cite{Hasminskii_1980} and \cite{Ethier_Kurtz_1986} for an overview of the subject. The computation of invariant distributions for stochastic systems has already been widely explored in the literature. In \cite{Soize_1994}, explicit exact expressions of the invariant density distribution for some solutions of Stochastic Differential Equations are given.\\

However, in many cases there is no explicit formula for $\nu$. A first approach consists in studying the convergence, as $t$ tends to infinity, of the semigroup $(P_t)_{t \geqslant 0}$ of the Markov process $(X_t)_ {t\geqslant 0}$ with infinitesimal generator $A$ towards the invariant measure $\nu$. This is done $e.g.$ in \cite{Ganidis_Roynette_Simonot_1999} for the total variation topology which is thus adapted when the simulation of $P_T$ is possible for $T$ large enough. \\

Whenever $(X_t)_{t \geqslant 0}$ can be simulated, we can use a Monte Carlo method to estimate $(P_t)_{t \geqslant 0}$, $i.e.$ $\mathbb{E}[f(X_t)]$, $t\geqslant 0$, producing a second term in the error analysis. When $(X_t)_{t \geqslant 0}$ cannot be simulated at a reasonable cost, a solution consists in simulating an approximation of $(X_t)_{t \geqslant 0}$, using a numerical scheme $(\overline{X}^{\gamma}_{\Gamma_n})_{n\in \mathbb{N}}$ built with transition functions $( \mathscr{Q}_{\gamma_n})_{n \in\mathbb{N}^{\ast}}$ (given a step sequence $(\gamma_n)_{n \in \mathbb{N}}$, $\Gamma_0=0$ and $\Gamma_n=\gamma_1+..+\gamma_n$). If the process $(\overline{X}^{\gamma}_{\Gamma_n})_{n\in \mathbb{N}}$ weakly converges towards $(X_t)_{t \geqslant 0}$, a natural construction relies on numerical homogeneous schemes ($(\gamma_n)_{n \in \mathbb{N}}$ is constant, $\gamma_n=\gamma_1>0$, for every $n \in \mathbb{N}^{\ast}$). This approach induces two more terms to control in the approximation of $\nu$ in addition to the error between $P_T$ and $\nu$ for a large enough fixed $T>0$, such that there exists $n(T) \in\mathbb{N}^{\ast}$,with $T=n(T) \gamma_1$: The first one is due to the weak approximation of $\mathbb{E}[f(X_T)]$ by $\mathbb{E}[f(\overline{X}^{\gamma_1}_T)]$ and the second one is due to the Monte Carlo error resulting from the computation of $\mathbb{E}[f(\overline{X}^{\gamma_1}_T]$.\\
%
%
%

\noindent Such an approach does not benefit from the ergodic feature of $(X_t)_{t \geqslant 0}$. In fact, as investigated in \cite{Talay_1990} for Brownian diffusions, the ergodic (or positive recurrence) property of $(X_t)_{t \geqslant 0}$ is also satisfied by its approximation $(\overline{X}^{\gamma}_{\Gamma_n})_{n\in \mathbb{N}}$ at least for small enough time step $\gamma_n= \gamma_1, n \in \mathbb{N}^{\ast}$. Then $(\overline{X}^{\gamma_1}_{\Gamma_n})_{n\in \mathbb{N}}$ has an invariant distribution $\nu^{\gamma_1}$ (supposed to be unique for simplicity) and the sequence of empirical measures 
 \begin{align*}
 \nu^{\gamma_1}_n(dx)=\frac{1}{ \Gamma_n} \sum_{k=1}^n \gamma_1 \delta_{\overline{X}^{\gamma_1}_{\Gamma_{k-1}}}(dx), \qquad \Gamma_n = n \gamma_1.
 \end{align*}
almost surely weakly converges to $\nu^{\gamma_1}$. With this last result makes it is possible to compute by simulation, arbitrarily accurate approximations of $\nu^{\gamma_1}(f)$ using only one simulated path of $(\overline{X}^{\gamma}_{\Gamma_n})_{n\in \mathbb{N}}$. It is an ergodic - or Langevin - simulation of $\nu^{\gamma_1}(f)$. However, it remains to establish at least that $\nu^{\gamma_1}(f)$ converges to $\nu(f)$ when $\gamma_1$ converges to zero and, if possible, at which rate.
%
%
 Another approach was proposed in \cite{Basak_Hu_Wei_1997}, still for Brownian diffusions, which avoids the asymptotic analysis between $\nu^{\gamma_1}$ and $\nu$. The authors directly prove that the discrete time Markov process $(\overline{X}^{\gamma}_{\Gamma_n})_{n\in \mathbb{N}}$, with step sequence $\gamma=(\gamma_n)_{n \in \mathbb{N}}$ vanishing to 0, weakly converges toward $\nu$. Therefore, the resulting error is made of two terms. The first one is due to this weak convergence and the second one to the Monte Carlo error involved in the computation of the law of $\overline{X}^{\gamma}_{\Gamma_n}$, for $n$ large enough. The reader may notice that in mentioned approaches, strong ergodicity assumptions are required for the process with infinitesimal generator $A$. \\
\noindent In \cite{Lamberton_Pages_2002}, these two ideas are combined to design a Langevin Euler Monte Carlo recursive algorithm with decreasing steps which $a.s.$ weakly converges to the right target $\nu$. This paper treats the case where $ (\overline{X}^{\gamma}_{\Gamma_n})_{n\in \mathbb{N}}$ is a (inhomogeneous) Euler scheme with decreasing steps associated to a strongly mean reverting Brownian diffusion process. The sequence $(\nu^{\gamma}_n)_{n \in \mathbb{N}^{\ast}}$ is defined as the weighted empirical measures of the path of $ (\overline{X}^{\gamma}_{\Gamma_n})_{n\in \mathbb{N}}$ (which is the procedure that is used in every work we mention from now on and which is also the one we use in this paper). In particular, the $a.s.$ weak convergence of
 \begin{align}
 \label{eq:def_weight_emp_meas_intro}
 \nu^{\gamma}_n(dx)=\frac{1}{\Gamma_n} \sum_{k=1}^n \gamma_k \delta_{\overline{X}^{\gamma}_{\Gamma_{k-1}}}(dx), \qquad \Gamma_n=\sum\limits_{k=1}^n \gamma_k,
 \end{align}
toward the (non-empty) set $\mathcal{V}$ of the invariant distributions of the underlying Brownian diffusion is established.
 Moreover, when the invariant measure $\nu$ is unique, it is proved that $\lim\limits_{n \to + \infty} \nu^{\gamma}_n f=\nu f \; a.s.$ for a larger class of test functions than $\mathcal{C}^0$ which contains $\nu-a.s.$ continuous functions with polynomial growth $i.e.$ convergence for the Wasserstein distance. In the spirit of \cite{Bhattacharya_1982} for the empirical measure of the underlying diffusion, they also obtained rates and limit gaussian laws for the convergence of $(\nu^{\gamma}_n(f))_{n \in \mathbb{N}^{\ast}}$ for test functions $f$ which can be written $f= A \varphi$. Note that, this approach does not require that the invariant measure $\nu$ is unique by contrast with the results obtained in \cite{Talay_1990} and \cite{Basak_Hu_Wei_1997} or in \cite{Durmus_Moulines_2015} where the authors study of the total variation convergence for the Euler scheme with decreasing steps of the over-damped Langevin diffusion under strong ergodicity assumptions. for instance. In this case, it is established that $a.s.$, every weak limiting distribution of $(\nu^{\gamma}_n)_{n \in \mathbb{N}^{\ast}}$ is an invariant distribution for the Brownian diffusion. \\
 This first paper gave rise to many generalizations and extensions. In \cite{Lamberton_Pages_2003}, the initial result is extended to the case of Euler scheme of Brownian diffusions with weakly mean reverting properties. Thereafter, in \cite{Lemaire_thesis_2005}, the class of test functions for which we have $\lim\limits_{n \to + \infty} \nu^{\gamma}_n f =\nu f \; a.s.$ (when the invariant distribution is unique) is extended to include functions with exponential growth. Finally, in \cite{Panloup_2008}, the results concerning the polynomial case are shown to hold for the computation of invariant measures for weakly mean reverting Levy driven diffusion processes, still using the algorithm from \cite{Lamberton_Pages_2002}. This extension encourages relevant perspectives concerning not only the approximation of mean reverting Brownian diffusion stationary regimes but also to treat a larger class of processes. For a more complete overview of the studies concerning (\ref{eq:def_weight_emp_meas_intro}) for the Euler scheme, the reader can also refer to \cite{Pages_2001_ergo}, \cite{Lemaire_2007}, \cite{Panloup_2008_rate}, \cite{Pages_Panloup_2009}, \cite{Pages_Panloup_2012} or \cite{Mei_Yin_2015}.
 %

\medskip
\noindent In this paper, we extend those existing results and show that the Langevin Euler Monte Carlo algorithm presented in \cite{Lamberton_Pages_2002} and generalized to the case where $( \mathscr{Q}_{\gamma})_{\gamma >0}$ is not specified explicitly, enables to approximate invariant, not necessarily unique, measures for Feller processes. \\

In a first step, we present an abstract framework adapted to the computation of invariant distributions for Feller processes under general mean reverting assumptions (including weakly mean reverting assumptions). Then, we establish $a.s$ weak convergence of $(\nu^{\gamma}_n)_{n \in \mathbb{N}^{\ast}}$. Moreover, when the invariant distribution $\nu$ is unique we obtain $\lim\limits_{n \to + \infty} \nu^{\gamma}_n f =\nu f \; a.s.$ for a generic class of continuous test functions $f$ (adapted among other to polynomial and exponential test functions $f$). \\
Then in a second step, we apply this abstract results to concrete cases and obtain original results. Notice that the existing results mentioned above can be recovered from our abstract framework. We begin by providing Wasserstein convergence results concerning Euler and Milstein schemes of Brownian diffusion processes in a weakly mean reverting setting. 
Then, we propose a detailed application concerning the Euler scheme of a Markov Switching diffusion for test functions $f$ with polynomial growth (Wasserstein convergence) or exponential growth. Here, we extend the convergence results from \cite{Mei_Yin_2015} where the authors adapted the algorithm from \cite{Lamberton_Pages_2002} under strong ergodicity assumptions for the Wasserstein convergence. 

\section{Convergence to invariant distributions - A general approach}
\label{section:convergence_inv_distrib_gnl}
%
In this section, we show that the empirical measures defined in the same way as in (\ref{eq:def_weight_emp_meas_intro}) and built from an approximation $(\overline{X}^{\gamma}_{\Gamma_n})_{n\in \mathbb{N}}$ of a Feller process $(X_t)_{t \geqslant 0}$ (which are not specified explicitly), where the step sequence $(\gamma_n)_{n \in \mathbb{N}^{\ast}} \underset{n \to + \infty}{\to}0$, $a.s.$ weakly converges the set $\mathcal{V}$, of the invariant distributions of $(X_t)_{t \geqslant 0}$. To this end, we will provide as weak as possible mean reverting assumptions on the pseudo-generator of $(\overline{X}^{\gamma}_{\Gamma_n})_{n\in \mathbb{N}}$ on the one hand and appropriate rate conditions on the step sequence $(\gamma_n)_{n \in \mathbb{N}^{\ast}}$ on the other hand.
\subsection{Presentation of the abstract framework}
%
\subsubsection{Notations}
Let $(E,\vert . \vert)$ be a locally compact separable metric space, we denote $\mathcal{C}(E)$ the set of continuous functions on $E$ and $\mathcal{C}_0(E)$ the set of continuous functions that vanish a infinity. We equip this space with the sup norm $\Vert f \Vert_{\infty}=\sup_{x \in E} \vert f(x) \vert$ so that $(\mathcal{C}_0(E),\Vert . \Vert_{\infty})$ is a Banach space. We will denote $\mathcal{B}(E)$ the $\sigma$-algebra of Borel subsets of $E$ and $\mathcal{P}(E)$ the family of Borel probability measures on $E$. We will denote by $\mathcal{K}_E$ the set of compact subsets of $E$.\\
Finally, for every Borel function $f:E \to \mathbb{R}$, and every $l_{\infty} \in \mathbb{R} \cup \{-\infty,+\infty\}$, $\lim\limits_{x\to \infty}f(x)= l_{\infty}$ if and only if for every $\epsilon >0$, there exists a compact $K_{\epsilon} \subset \mathcal{K}_E$ such that $\sup_{x \in K_{\epsilon}^c} \vert f(x)- l_{\infty} \vert < \epsilon$ if $l_{\infty} \in \mathbb{R} $, $\inf_{x \in K_{\epsilon}^c}  f(x)  > 1/\epsilon$ if $l_{\infty} =+\infty$, and $\sup\limits_{x \in K_{\epsilon}^c}  f(x)  < -1/\epsilon$ if $l_{\infty} =-\infty$ with $K_{\epsilon}^c=E \setminus K_{\epsilon}.$  \\
%
%
%
%
\subsubsection{Construction of the random measures}
Let $(\Omega,\mathcal{G}, \mathbb{P})$ be a probability space. We consider a Feller process $(X_t)_{t \geqslant 0}$ (see \cite{Feller_1952} for details) on $(\Omega,\mathcal{G}, \mathbb{P})$ taking values in a locally compact and separable metric space $E$. We denote by $(P_t)_{t \geqslant 0}$ the Feller semigroup (see \cite{Pazy_1992}) of this process. We recall that $(P_t)_{t \geqslant 0}$  is a family of linear operators from $\mathcal{C}_0(E)$ to itself such that $P_0 f=f$, $P_{t+s}f=P_tP_sf$, $t,s \geqslant 0$ (semigroup property) and $\lim\limits_{t \to 0} \Vert P_tf-f \Vert_{\infty}=0$ (Feller property). Using this semigroup, we can introduce the infinitesimal generator of $(X_t)_{t \geqslant 0}$ as a linear operator $A$ defined on a subspace $\DomA$ of $\mathcal{C}_0(E)$, satisfying: For every $f \in \DomA$,
\begin{align*}
Af= \lim\limits_{t \to 0} \frac{P_tf-f}{t}
\end{align*}
exists for the $\Vert . \Vert_{\infty}$-norm. The operator $A: \DomA \to \mathcal{C}_0(E)$ is thus well defined and $\DomA$ is called the domain of $A$. From the Echeverria Weiss theorem (see Theorem \ref{th:eche_weiss_feller}), the set of invariant distributions for $(X_t)_{t \geqslant 0}$ can be characterized in the following way: 
\begin{align*}
\mathcal{V}=\{ \nu \in \mathcal{P}(E), \forall t \geqslant 0, P_t \nu=  \nu \}=\{ \nu \in \mathcal{P}(E), \forall f \in \DomA, \nu(Af)=0 \}.
\end{align*}
The starting point of our reasoning is thus to consider an approximation of $A$. First, we introduce the family of transition kernels $(\mathscr{Q}_{\gamma})_{\gamma >0}$ from $\mathcal{C}_0(E)$ to itself. Now, let us define the family of linear operators $\widetilde{A} : = (\widetilde{A}_{\gamma} )_{ \gamma >0}$ from $\mathcal{C}_0(E)$ into itself, as follows
 \begin{equation*}
 \forall f \in \mathcal{C}_0(E), \quad \gamma>0,  \qquad \widetilde{A}_{\gamma}f=\frac{\mathscr{Q}_{\gamma}f -f}{\gamma}.
 \end{equation*}
The family $\widetilde{A}$ is usually called the pseudo-generator of the transition kernels $(\mathscr{Q}_{\gamma})_{\gamma>0}$ and is an approximation of $A$ as $\gamma$ tends to zero. From a practical viewpoint, the main interest of our approach is that we can consider that there exists $\overline{\gamma}>0$ such that for every $x \in E$ and every $\gamma \in [0, \overline{\gamma}]$, $\mathscr{Q}_{\gamma} (x,dy)$ is simulable at a reasonable computational cost. We use the family $(\mathscr{Q}_{\gamma})_{\gamma>0}$, to build $(\overline{X}_{\Gamma_n})_{n \in \mathbb{N}}$ (this notation replaces $(\overline{X}^{\gamma}_{\Gamma_n})_{n \in \mathbb{N}}$ from now for clarity in the writing) as the non-homogeneous Markov approximation of the Feller process $(X_t)_{t \geqslant 0}$. It is defined on the time grid $\{ \Gamma_n=\sum\limits_{k=1}^n \gamma_k, n \in \mathbb{N} \} $ with the sequence $\gamma:=(\gamma_n)_{n\in \mathbb{N}^{\ast} }$ of time step satisfying
\begin{align*}
\forall n \in \mathbb{N}^{\ast}, \quad 0 < \gamma_n  \leqslant \overline{\gamma}:= \sup_{n \in \mathbb{N}^{\ast}} \gamma_n< + \infty, \quad \lim\limits_{n \to + \infty} \gamma_n = 0 \quad \mbox{ and } \quad \lim\limits_{n \to + \infty}\Gamma_n=+ \infty.
\end{align*}
Its transition probability distributions are given by $\mathscr{Q}_{\gamma_n} (x,dy),n\in \mathbb{N}^{\ast}$, $x\in E$, $i.e. :$
\begin{align*}
 \mathbb{P}(\overline{X}_{\Gamma_{n+1}} \in dy \vert \overline{X}_{\Gamma_n})= \mathscr{Q}_{\gamma_{n+1}}(\overline{X}_{\Gamma_n},dy), \quad n \in \mathbb{N}.
\end{align*}
%
%
We can canonically extend $(\overline{X}_{\Gamma_n})_{n \in \mathbb{N}}$ into a \textit{càdlàg} process by setting $\overline{X}(t,\omega) =\overline{X}_{\Gamma_{n(t)}}(\omega)$ with $n(t)= \inf \{n \in \mathbb{N}, \Gamma_{n+1}>t \}$. Then $(\overline{X}_{\Gamma_n})_{n \in \mathbb{N}}$ is a simulable (as soon as $\overline{X}_0$ is) non-homogeneous Markov chain with transitions 
\begin{align*}
\forall m  \leqslant n, \qquad \overline{P}_{\Gamma_m,\Gamma_n}(x,dy)= \mathscr{Q}_{\gamma_{m+1}} \circ \cdots \circ \mathscr{Q}_{\gamma_n}(x,dy),
\end{align*}
and law
\begin{align*}
\mathcal{L}(\overline{X}_{\Gamma_n} \vert \overline{X}_{0}=x)=\overline{P}_{\Gamma_n}(x,dy)= \mathscr{Q}_{\gamma_1} \circ \cdots \circ \mathscr{Q}_{\gamma_n}(x,dy).
\end{align*}
We use $(\overline{X}_{\Gamma_n})_{n \in \mathbb{N}}$ to design a Langevin Monte Carlo algorithm. Notice that this approach is generic since the approximation transition kernels $(\mathscr{Q}_{\gamma})_{\gamma>0}$ are not explicitly specified and then, it can be used in many different configurations including among others, weak numerical schemes or exact simulation $i.e.$ $(\overline{X}_{\Gamma_n})_{n \in \mathbb{N}}=(X_{\Gamma_n})_{n \in \mathbb{N}}$. In particular, using high weak order schemes for $(X_t)_{t \geqslant 0}$ may lead to higher rates of convergence for the empirical measures. The approach we use to build the empirical measures is quite more general than in (\ref{eq:def_weight_emp_meas_intro}) as we consider some general weights which are not necessarily equal to the time steps. We define this weight sequence. Let $\eta:=(\eta_n)_{n \in \mathbb{N}^{\ast}}$ be such that
\begin{equation*}
\forall n \in \mathbb{N}^{\ast}, \quad \eta_n \geqslant 0, \quad \lim\limits_{n \to + \infty} H_n=+ \infty, \qquad \mbox{with} \qquad H_n= \sum\limits_{k=1}^n \eta_k.
\end{equation*}
Now we present our algorithm adapted from the one introduced in \cite{Lamberton_Pages_2002} designed with a Euler scheme with decreasing steps $(\overline{X}_{\Gamma_n})_{n \in \mathbb{N}}$ of a Brownian diffusion process $(X_t)_{t \geqslant 0}$. For $x \in E$, let $\delta_x$ denote the Dirac mass at point $x$. For every $n \in \mathbb{N}^{\ast}$, we define the random weighted empirical random measures as follows
 \begin{equation}
 \label{eq:def_weight_emp_meas}
 \nu^{\eta}_n(dx)=\frac{1}{H_n} \sum_{k=1}^n \eta_k \delta_{\overline{X}_{\Gamma_{k-1}}}(dx).
 \end{equation}
%
%
%
%
%
This paper is dedicated to show that $a.s.$ every weak limiting distribution of $(\nu^{\eta}_n)_{n \in \mathbb{N}^{\ast}}$ belongs to $\mathcal{V}$. In particular when the invariant measure of $(X_t)_{t \geqslant 0}$ is unique, $i.e. \; \mathcal{V}=\{\nu\}$, we show that $\mathbb{P}-a.s. \;\lim\limits_{n \to + \infty} \nu^{\eta}_n f =\nu f  $, for a generic class of continuous test functions $f$. The approach we develop consists in two steps. First, we establish a tightness property to obtain existence of at least one weak limiting distribution for $(\nu^{\eta}_n )_{n \in \mathbb{N}^{\ast}}$. Then, in a second step, we identify everyone of these limiting distributions with an invariant distributions of the Feller process $(X_t)_{t \geqslant 0}$ exploiting the Echeverria Weiss theorem (see Theorem \ref{th:eche_weiss_feller}).
\subsubsection{Assumptions on the random measures}
%
In this part, we present the necessary assumptions on the pseudo-generator $\widetilde{A}  = (\widetilde{A}_{\gamma} )_{ \gamma >0}$ in order to prove the convergence of the empirical measures $(\nu^{\eta}_n)_{n \in \mathbb{N}^{\ast}}$.
\paragraph{Mean reverting recursive control\\}
 In our framework, we introduce a well suited assumption, referred to as the mean reverting recursive control of the pseudo-generator $\widetilde{A}$, that leads to a tightness property on $(\nu^{\eta}_n)_{n \in \mathbb{N}^{\ast}}$ from which follows the existence (in weak sense) of a limiting distribution for $(\nu^{\eta}_n)_{n \in \mathbb{N}^{\ast}}$. A supplementary interest of our approach is that it is designed to obtain the $a.s.$ convergence of $(\nu^{\eta}_n(f))_{n \in \mathbb{N}^{\ast}}$ for a generic class of continuous test functions $f$ which is larger then $\mathcal{C}_b(E)$. To do so, we introduce a Lyapunov function $V$ related to $(\overline{X}_{\Gamma_n})_{n \in \mathbb{N}}$. Assume that $V$ a Borel function such that
\begin{equation}
\label{hyp:Lyapunov}
\mbox{L}_{V}  \quad \equiv \qquad   V :(E \to [v_{\ast},+\infty), v_{\ast}> 0 \quad  \mbox{ and } \quad \lim\limits_{ x  \to \infty} V(x)=+ \infty. \\
\end{equation}
We now relate $V$ to $(\overline{X}_{\Gamma_n})_{n \in \mathbb{N}}$ introducing its mean reversion Lyapunov property. Let $\psi, \phi : [v_{\ast},\infty) \to (0,+\infty) $ some Borel functions such that $\widetilde{A}_{\gamma}\psi \circ V$ exists for every $\gamma \in (0, \overline{\gamma}]$. Let $\alpha>0$ and $\beta \in \mathbb{R}$. We assume  
%
%
 \begin{eqnarray}
\label{hyp:incr_sg_Lyapunov}
&\mathcal{RC}_{Q,V} (\psi,\phi,\alpha,\beta) \quad \equiv  \nonumber\\
& \quad  \left\{
    \begin{array}{l}
  (i) \; \quad \exists n_0 \in \mathbb{N}^{\ast},   \forall n   \geqslant n_0, x \in E,  \quad\widetilde{A}_{\gamma_n}\psi \circ V(x)\leqslant  \frac{\psi \circ V(x)}{V(x)}(\beta - \alpha \phi \circ V(x)). \\
  (ii) \quad     \liminf\limits_{y \to + \infty} \phi(y)> \beta / \alpha .
    \end{array}
\right.
\end{eqnarray}
$\mathcal{RC}_{Q,V} (\psi,\phi,\alpha,\beta)$ is called the weakly mean reverting recursive control assumption of the pseudo generator for Lyapunov function $V$. \\

Lyapunov functions are usually used to show the existence and sometimes the uniqueness of the invariant measure of Feller processes. In particular, when $p=1$, the condition $\mathcal{RC}_{Q,V}(I_d,I_d,\alpha,\beta) (i)$ appears as the discrete version of $AV \leqslant \beta-\alpha V$, which is used in that interest for instance in \cite{Hasminskii_1980}, \cite{Ethier_Kurtz_1986}, \cite{Basak_Hu_Wei_1997} or\cite{Pages_2001_ergo}. \\

 The condition $\mathcal{RC}_{Q,V}(V^p,I_d,\alpha,\beta) (i)$, $p \geqslant 1$, is studied in the seminal paper \cite{Lamberton_Pages_2002} (and then in \cite{Lamberton_Pages_2003} with $\phi(y)=y^a,a\in (0,1]$,$y \in [v_{\ast},\infty)$) concerning the Wasserstein convergence of the weighted empirical measures of the Euler scheme with decreasing steps of a Brownian diffusions. When $\phi=I_d$, the Euler scheme is also studied for markov switching Brownian diffusions in \cite{Mei_Yin_2015}. Notice also that $\mathcal{RC}_{Q,V}(I_d,\phi,\alpha,\beta) (i)$ with $\phi$ concave appears in \cite{DFMS_2004} to prove sub-geometrical ergodicity of Markov chains. In \cite{Lemaire_thesis_2005}, a similar hypothesis to $\mathcal{RC}_{Q,V}(I_d,\phi,\alpha,\beta) (i)$ (with $\phi$ not necessarily concave and $\widetilde{A}_{\gamma_n}$ replaced by $A$), is also used  to study the Wasserstein but also exponential convergence of the weighted empirical measures (\ref{eq:def_weight_emp_meas}) for the Euler scheme of a Brownian diffusions. Finally in \cite{Panloup_2008} similar properties as $\mathcal{RC}_{Q,V}(V^p,V^a,\alpha,\beta) (i)$, $a\in (0,1]$, $p >0$, are developped in the study of the Euler scheme for Levy processes.\\
 
On the one hand, the function $\phi$ controls the mean reverting property. In particular, we call strongly mean reverting property when $\phi=I_d$ and weakly mean reverting property when $\lim\limits_{y \to +\infty} \phi(y)/y=0$, for instance $\phi(y)=y^a$, $a \in (0,1)$ for every $y \in [v_{\ast},\infty)$. On the other hand, the function $\psi$ is closely related to the identification of the set of test functions $f$ for which we have $\lim\limits_{n \to +\infty} \nu^{\eta}_n(f)=\nu(f) \; a.s.$, when $\nu$ is the unique invariant distribution of the underlying Feller process.\\
 
  To this end, for $s \geqslant 1$, which is related to step weight assumption, we introduce the sets of test functions for which we will show the $a.s.$ convergence of the weighted empirical measures (\ref{eq:def_weight_emp_meas}):
\begin{align}
\label{def:espace_test_function_cv}
\mathcal{C}_{\tilde{V}_{\psi,\phi,s}}(E)=& \big\{ f \in \mathcal{C}(E), \vert f(x) \vert=\underset{  x \to \infty}{o}( \tilde{V}_{\psi,\phi,s} (x) ) \big\}, \\
&\mbox{with} \quad \tilde{V}_{\psi,\phi,s}:E \to \mathbb{R}_+, x \mapsto\tilde{V}_{\psi,\phi,s}(x): =\frac{\phi\circ V(x)\psi \circ V(x)^{1/s}}{V(x)}. \nonumber
\end{align}
Notice that our approach benefits from providing generic results because we consider general Feller processes and approximations but also because the functions $\phi$ and $\psi$ are not specified explicitly.
\paragraph{Infinitesimal generator approximation \\}
This section presents the assumption that enables to characterize the limiting distributions of the $a.s.$ tight  sequence $(\nu^{\eta}_n(dx, \omega))_{n \in \mathbb{N}^{\ast}}$.
 We aim to estimate the distance between $\mathcal{V}$ and $\nu^{\eta}_n$ (see (\ref{eq:def_weight_emp_meas})) for $n$ large enough. We thus introduce an hypothesis concerning the distance between $(\widetilde{A}_{\gamma} )_{ \gamma>0}$, the pseudo-generator of $(\mathscr{Q}_{\gamma} )_{ \gamma>0}$, and $A$, the infinitesimal generator of $(P_t)_{t \geqslant 0}$. We assume that there exists $\DomA_0 \subset \DomA$ with $\DomA_0 $ dense in $\mathcal{C}_0(E)$ such that:
%
%
 \begin{align}
\mathcal{E}(\widetilde{A},A,\DomA_0) \quad \equiv \qquad   \forall \gamma \in (0, \overline{\gamma}], & \forall f \in \DomA_0, \forall x \in E,   \nonumber \\
&  \vert \widetilde{A}_{\gamma} f(x) -Af(x)\vert \leqslant  \Lambda_f(x,\gamma),
 \label{hyp:erreur_tems_cours_fonction_test_reg}
\end{align}
where $\Lambda_{f}:E \times \mathbb{R}_+ \to \mathbb{R}_+$ can be represented in the following way: Let $(\tilde{\Omega},\tilde{\mathcal{G}},\tilde{\mathbb{P}})$ be a probability space. Let $g :E\to \mathbb{R}_+^{q}$, $q \in \mathbb{N}$, be a locally bounded Borel measurable function and let $\tilde{\Lambda}_{f}:(E\times \mathbb{R}_+ \times \tilde{\Omega}, \mathcal{B}(E) \otimes \mathcal{B}(\mathbb{R}_+) \otimes \tilde{\mathcal{G}}) \to \mathbb{R}_+^{q}$ be a measurable function such that  $\sup_{i \in \{1,\ldots,q\} } \tilde{\mathbb{E}}[ \sup_{x \in E} \sup_{\gamma \in (0,\overline{\gamma}] } \tilde{\Lambda}_{f,i}(x,\gamma, \tilde{\omega}) ]< + \infty$ and
%
%
%
\begin{align*}
\forall x \in E , \forall \gamma \in (0,\overline{\gamma}], \qquad \Lambda_f(x,\gamma)= \langle g (  x ) ,\tilde{\mathbb{E}} [\tilde{\Lambda}_{f}(x,\gamma, \tilde{\omega})]  \rangle_{\mathbb{R}^q}
\end{align*}
%
%
%
%
%
%
%
Moreover, we assume that for every $i \in \{1,\ldots,q\}$, $\sup_{n \in \mathbb{N}^{\ast}} \nu_n^{\eta}( g_i ,\omega )< + \infty, \; \mathbb{P}(d\omega)-a.s.$, and that $\tilde{\Lambda}_{f,i}$ satisfies one of the following two properties:\\
There exists a measurable function $\underline{\gamma}:(\tilde{\Omega}, \tilde{\mathcal{G}}) \to((0, \overline{\gamma}],\mathcal{B}((0, \overline{\gamma}]) )$ such that:
\begin{enumerate}[label=\textbf{\Roman*)}]
\item \label{hyp:erreur_tems_cours_fonction_test_reg_Lambda_representation_1} 
\inlineequation[hyp:erreur_temps_cours_fonction_test_reg_Lambda_representation_2_1]{
 \tilde{\mathbb{P}}(d\tilde{\omega})-a.s \qquad \left\{
    \begin{array}{l}
  (i)   \quad \; \;  \forall K \in \mathcal{K}_E ,   \quad  \lim\limits_{\gamma \to 0} \sup\limits_{x \in K} \tilde{\Lambda}_{f,i}(x, \gamma,\tilde{\omega})=0, \\
  (ii) \quad     \lim\limits_{x \to \infty}  \sup\limits_{\gamma \in (0,\underline{\gamma}(\tilde{\omega}) ]} \tilde{\Lambda}_{f,i}(x, \gamma,\tilde{\omega})=0,  \qquad \qquad \quad 
    \end{array}
\right.
}
\item \label{hyp:erreur_temps_cours_fonction_test_reg_Lambda_representation_2}\inlineequation[hyp:erreur_temps_cours_fonction_test_reg_Lambda_representation_2_2]{
 \tilde{\mathbb{P}}(d\tilde{\omega})-a.s \qquad  \lim\limits_{\gamma \to 0} \sup\limits_{x \in E} \tilde{\Lambda}_{f,i}(x, \gamma,\tilde{\omega}) g_i(x) =0  . \qquad \qquad \qquad \qquad \qquad \; \;}
\end{enumerate}
\begin{remark}
\label{rmk:representation_mesure_infinie}
Let $(F,\mathcal{F},\lambda)$ be a measurable space. Using the exact same approach, the results we obtain hold when we replace the probability space $(\tilde{\Omega},\tilde{\mathcal{G}},\tilde{\mathbb{P}})$ by the product measurable space $(\tilde{\Omega} \times F,\tilde{\mathcal{G}} \otimes \mathcal{F},\tilde{\mathbb{P}}\otimes \lambda)$ in the representation of $\Lambda_f$ and in (\ref{hyp:erreur_temps_cours_fonction_test_reg_Lambda_representation_2_1}) and (\ref{hyp:erreur_temps_cours_fonction_test_reg_Lambda_representation_2_2}) but we restrict to that case for sake of clarity in the writing. This observation can be useful when we study jump process where $\lambda$ can stand for the jump intensity.
\end{remark}
This representation assumption benefits from the fact that the transition functions $(\mathscr{Q}_{\gamma} (x,dy))_{\gamma \in (0, \overline{\gamma}]}$, $x \in E$, can be represented using distributions of random variables which are involved in the computation of $(\overline{X}_{\Gamma_n})_{n \in \mathbb{N}^{\ast}}$. In particular, this approach is well adapted to stochastic approximations associated to a time grid such as numerical schemes for stochastic differential equations with a Brownian part or/and a jump part. 
\paragraph{Growth control and Step Weight assumptions \\}
We conclude with hypothesis concerning the control of the martingale part of one step of our approximation. Let $\rho \in [1,2]$ and let $\epsilon_{\mathcal{I}} : \mathbb{R}_+ \to \mathbb{R}_+$ an increasing function. For $F \subset \{f,f:(E, \mathcal{B}(E)) \to (\mathbb{R}, \mathcal{B}(\mathbb{R}) ) \}$ and $g:E \to \mathbb{R}_+$ a Borel function, we assume that, for every $n \in \mathbb{N}$,
 \begin{align}
\label{hyp:incr_X_Lyapunov}
\mathcal{GC}_{Q} & (F,g,\rho,\epsilon_{\mathcal{I}}) \;  \equiv \quad  \mathbb{P}-a.s.\quad  \forall f \in F, \nonumber \\
& \mathbb{E}[ \vert  f  ( \overline{X}_{\Gamma_{n+1}})- \mathscr{Q}_{\gamma_{n+1}}f(\overline{X}_{\Gamma_n}) \vert^{\rho}\vert \overline{X}_{\Gamma_n} ]  \leqslant   C_f \epsilon_{\mathcal{I}}(\gamma_{n+1})  g (\overline{X}_{\Gamma_n}) ,
\end{align}
with $C_f>0$ a finite constant which may depend on $f$. 
\begin{remark}\label{rmrk:Accroiss_mes} The reader may notice that $\mathcal{GC}_{Q}(F,g,\rho,\epsilon_{\mathcal{I}}) $ holds as soon as (\ref{hyp:incr_X_Lyapunov}) is satisfied with $\mathscr{Q}_{\gamma_{n+1}}f(\overline{X}_{\Gamma_n})$, $n  \in \mathbb{N}^{\ast} $, replaced by a $\mathcal{F}^{\overline{X}}_n:=\sigma(\overline{X}_{\Gamma_k},k \leqslant n)$- progressively  measurable process $(\mathfrak{X}_n)_{n \in \mathbb{N}^{\ast}}$ since we have $\mathscr{Q}_{\gamma_{n+1}}f(\overline{X}_{\Gamma_n}) =\mathbb{E}[f(\overline{X}_{\Gamma_{n+1}}) \vert \overline{X}_{\Gamma_n}]$ and $\mathbb{E}[ \vert  f  ( \overline{X}_{\Gamma_{n+1}})- \mathscr{Q}_{\gamma_{n+1}}f(\overline{X}_{\Gamma_n}) \vert^{\rho}\vert \overline{X}_{\Gamma_n} ]  \leqslant 2^{\rho} \mathbb{E}[ \vert  f  ( \overline{X}_{\Gamma_{n+1}})- \mathfrak{X}_n \vert^{\rho}\vert \overline{X}_{\Gamma_n} ]$ for every $\mathfrak{X}_n \in \LL^2(\mathcal{F}^{\overline{X}}_n)$.
\end{remark}
We will combine this assumption with the following step weight related ones:
\begin{equation}
 \label{hyp:step_weight_I_gen_chow}
\mathcal{S}\mathcal{W}_{\mathcal{I}, \gamma,\eta}(g, \rho , \epsilon_{\mathcal{I}}) \quad \equiv \qquad   \mathbb{P}-a.s. \quad   \sum_{n=1}^{\infty} \Big \vert \frac{\eta_n }{H_n \gamma_n } \Big \vert^{\rho} \epsilon_{\mathcal{I}}(\gamma_n)  g(\overline{X}_{\Gamma_n})  < + \infty,
 \end{equation}
and
 \begin{align}
 \label{hyp:step_weight_I_gen_tens}
\mathcal{S}\mathcal{W}_{\mathcal{II},\gamma,\eta}(F) \; \equiv \quad \mathbb{P}-a.s. & \quad \forall f \in F, \nonumber \\
&   \sum_{n=0}^{\infty} \frac{(\eta_{n+1} /\gamma_{n+1}-\eta_n /\gamma_n)_+ }{H_{n+1} }  \vert f(\overline{X}_{\Gamma_n}) \vert < + \infty,
 \end{align}
with the convention $\eta_0/\gamma_0=1$. Notice that this last assumption holds as soon as the sequence $(\eta_n / \gamma_n)_{n \in \mathbb{N}^{\ast} }$ is non-increasing. \\

\noindent At this point we can focus now on the main results concerning this general approach.
\subsection{Convergence}
%
%
%
\subsection{Preliminary results }
%
In this section, we recall standard general results we employ to study the convergence. Our approach will rely on a specific version of the Martingale problem characterizing the existence of a Feller Markov process which directly provides the existence of a steady regime $i.e.$ an invariant distribution. This is the object of the Echeverria Weiss theorem.
\begin{mytheo}
\label{th:eche_weiss_feller}
\begin{enumerate}[label=\textbf{\Alph*.}]
\item {\textbf{(Echeverria Weiss} (see \cite{Ethier_Kurtz_1986} Theorem 9.17)\textbf{).}}
\label{th:eche_weiss}
Let $E$ be a locally compact and separable metric space and let $A:\DomA \subset \mathcal{C}_0(E)  \to \mathcal{C}_0(E) $ be a linear operator satisfying the positive maximum principle\footnote{$\forall f \in \DomA , f(x_0)= \sup \{f(x), x \in E \} \geqslant 0 , x_0 \in E \Rightarrow Af(x_0) \leqslant 0.$}, such that $\DomA$ is dense in $\mathcal{C}_0(E)$ and that there exists a sequence of functions $\varphi_n \in \DomA$ such that $\lim\limits_{n \to + \infty} \varphi_n=1$ and $\lim\limits_{n \to + \infty}A \varphi_n=0$ with $\sup_{n \in \mathbb{N}} \{ \Vert A \varphi_n \Vert_{\infty} \}< + \infty$. If $\nu \in \mathcal{P}(E)$ satisfies
\begin{align}
\label{eq:hyp_eche_weiss}
\forall f \in \DomA, \quad \int_E Af d\nu =0,
\end{align}
then there exists a stationary solution to the martingale problem $(A,\nu)$.
\item {\textbf{(Hille Yoshida} (see \cite{Revuz_Yor_1999} (Chapter VII,  Proposition 1.3 and Proposition 1.5) or \cite{Ethier_Kurtz_1986} (Chapter IV,  Theorem 2.2)) \textbf{).}}
 The infinitesimal generator of a Feller process satisfies the hypothesis from point \ref{th:eche_weiss} except for (\ref{eq:hyp_eche_weiss}).
 \end{enumerate}
\end{mytheo}
%
%
%
%
This paper is devoted to the proof of the existence of a measure $\nu$ which satisfies (\ref{eq:hyp_eche_weiss}). Using this result, property (\ref{eq:hyp_eche_weiss}) is sufficient to prove that $\nu$ is an invariant measure for the process with infinitesimal generator $A$. To be more specific, the measure $\nu$ is built as the limit of a sequence of random empirical measures $(\nu^{\eta}_n)_{n \in \mathbb{N}^{\ast}}$. When (\ref{eq:hyp_eche_weiss}) holds for this limit, we say that the sequence $(\nu^{\eta}_n)_{n \in \mathbb{N}^{\ast}}$ converges towards an invariant distribution of the Feller process with generator $A$. We begin with some preliminary results.
\begin{lemme}{\textbf{(Kronecker).}}
\label{lemme:Kronecker}
 Let $(a_n)_{n \in \mathbb{N}^{\ast}}$ and $(b_n)_{n \in \mathbb{N}^{\ast}}$ be two sequences of real numbers. If $(b_n)_{n \in \mathbb{N}^{\ast}}$ is non-decreasing, strictly positive, with $\lim\limits_{n \to + \infty}b_n=+\infty$ and $\sum\limits_{n \geqslant 1} a_n /b_n$ converges in $\mathbb{R}$, then
 \begin{equation*}
 \lim\limits_{n \to +\infty} \frac{1}{b_n} \sum_{k=1}^n a_k=0.
 \end{equation*}
\end{lemme}
\begin{mytheo}{\textbf{(Chow} (see \cite{Hall_Heyde_1980}, Theorem 2.17)\textbf{).}}
\label{th:chow}
Let $(M_n)_{n \in \mathbb{N}^{\ast}}$ be a real valued martingale with respect to some filtration $\mathcal{F}=(\mathcal{F}_n)_{n \in \mathbb{N}}$. Then
\begin{align*}
\quad \lim\limits_{n \to +\infty} M_n =M_{\infty} \in \mathbb{R} \;  a.s. & \quad   \mbox{on the event} \\
&\bigcup_{r \in [0,1]} \Big \{ \sum_{n=1}^{\infty} \mathbb{E} [ \vert M_n -M_{n-1} \vert^{1+r} \vert \mathcal{F}_{n-1} ] < + \infty \Big \}.
\end{align*}
\end{mytheo}
\subsubsection{Almost sure tightness}
From the recursive control assumption, the following Theorem establish the $a.s.$ tightness of the sequence $(\nu^{\eta}_n)_{n \in \mathbb{N}^{\ast}}$ and also provides a uniform control of $(\nu^{\eta}_n)_{n \in \mathbb{N}^{\ast}}$ on a generic class of test functions. 
%
%
%
\begin{mytheo}
\label{th:tightness}
Let $s \geqslant 1$, $\rho \in[1,2]$, $v_{\ast}>0$, and let us consider the Borel functions $V :E \to [v_{\ast},\infty)$, $g:E \to \mathbb{R}_+$, $\psi : [v_{\ast},\infty) \to \mathbb{R}_+ $ and $\epsilon_{\mathcal{I}} : \mathbb{R}_+ \to \mathbb{R}_+$ an increasing function. We have the following properties:
\begin{enumerate}[label=\textbf{\Alph*.}]
\item\label{th:tightness_point_A}  Assume that $\widetilde{A}_{\gamma_n}(\psi \circ V)^{1/s}$ exists for every $n \in \mathbb{N}^{\ast}$, and that $\mathcal{GC}_{Q}((\psi \circ V)^{1/s},g ,\rho ,\epsilon_{\mathcal{I}}) $ (see (\ref{hyp:incr_X_Lyapunov})), $\mathcal{S}\mathcal{W}_{\mathcal{I}, \gamma,\eta}( g,\rho,\epsilon_{\mathcal{I}}) $ (see (\ref{hyp:step_weight_I_gen_chow})) and $\mathcal{S}\mathcal{W}_{\mathcal{II},\gamma,\eta}((\psi \circ V)^{1/s}) $ (see (\ref{hyp:step_weight_I_gen_tens}) hold. Then
\begin{equation}
\label{eq:invariance_mes_emp_Lyap_gen}
\mathbb{P} \mbox{-a.s.} \quad  \sup_{n\in \mathbb{N}^{\ast} } - \frac{1}{H_n} \sum_{k=1}^n \eta_k \widetilde{A}_{\gamma_k} (\psi \circ V)^{1/s}  (\overline{X}_{\Gamma_{k-1}})< + \infty.
\end{equation}
\item\label{th:tightness_point_B}
Let $\alpha>0$ and $\beta \in \mathbb{R}$. Let $\phi:[v_{\ast},\infty )\to \mathbb{R}_+^{\ast}$ be a continuous function such that $C_{\phi}:= \sup_{y \in [v_{\ast},\infty )}\phi(y)/y< \infty$. Assume that (\ref{eq:invariance_mes_emp_Lyap_gen}) holds and
\begin{enumerate}[label=\textbf{\roman*.}]
\item \label{th:tightness_point_B_i} $\mathcal{RC}_{Q,V}(\psi,\phi,\alpha,\beta)$ (see (\ref{hyp:incr_sg_Lyapunov})) holds.
\item \label{th:tightness_point_B_ii} $\mbox{L}_{V}$ (see (\ref{hyp:Lyapunov})) holds and $\lim\limits_{y \to +\infty}  \frac{\phi(y) \psi (y)^{1/s}}{y}=+\infty$.
\end{enumerate}
Then,
 \begin{equation}
 \label{eq:tightness_gen}
\mathbb{P} \mbox{-a.s.} \quad \sup_{n \in \mathbb{N}^{\ast}} \nu_n^{\eta}( \tilde{V}_{\psi,\phi,s} ) < + \infty .
\end{equation}
with $\tilde{V}_{\psi,\phi,s}$ defined in (\ref{def:espace_test_function_cv}). Therefore, the sequence $(\nu^{\eta}_n)_{n \in \mathbb{N}^{\ast}}$ is $\mathbb{P}-a.s.$ tight. 
 \end{enumerate}
\end{mytheo}
\begin{proof}
We first prove point \ref{th:tightness_point_A} For $n \in \mathbb{N}^{\ast}$, we write 
\begin{align*}
 -\sum_{k=1}^n \eta_k \widetilde{A}_{\gamma_k}(\psi \circ V)^{1/s}(\overline{X}_{\Gamma_{k-1}}) = & -\sum_{k=1}^n  \frac{\eta_k}{\gamma_k}((\psi \circ V)^{1/s}(\overline{X}_{\Gamma_k})-(\psi \circ V)^{1/s}(\overline{X}_{\Gamma_{k-1}})) \\
 &+  \sum_{k=1}^n \frac{\eta_k}{\gamma_k} ((\psi \circ V)^{1/s}(\overline{X}_{\Gamma_k})-\mathscr{Q}_{\gamma_k}(\psi \circ V)^{1/s}(\overline{X}_{\Gamma_{k-1}}) )
\end{align*}
We study the first term of the $r.h.s.$ First, an Abel transform yields
\begin{align*}
 - \frac{1}{H_n}\sum_{k=1}^n  \frac{\eta_k}{\gamma_k}  ((\psi \circ V)^{1/s}(\overline{X}_{\Gamma_k})- & (\psi \circ V)^{1/s}(\overline{X}_{\Gamma_{k-1}}))  \\
 = & \frac{\eta_1}{H_n \gamma_1}(\psi \circ V)^{1/s}(\overline{X}_{0})-\frac{\eta_n}{H_n \gamma_n}(\psi \circ V)^{1/s}(\overline{X}_{\Gamma_n}) \\
 &+ \frac{1}{H_n}\sum_{k=2}^n \Big( \frac{\eta_k}{\gamma_k}-\frac{\eta_{k-1}}{\gamma_{k-1}} \Big)(\psi \circ V)^{1/s}(\overline{X}_{\Gamma_{k-1}}).\\
\end{align*}
%
%
%
%
%
We recall that $(\psi \circ V )^{1/s}$ is non negative. From $\mathcal{S}\mathcal{W}_{\mathcal{II},\gamma,\eta}((\psi \circ V )^{1/s})$ (see (\ref{hyp:step_weight_I_gen_tens})), we have
\begin{align*}
\mathbb{E} \Big[ \sup_{n \in \mathbb{N}^{\ast}} \sum_{k=1}^n  \frac{1}{H_k }\Big( \frac{\eta_k}{\gamma_k}-\frac{\eta_{k-1}}{\gamma_{k-1}} \Big)_+(\psi \circ V)^{1/s}(\overline{X}_{\Gamma_{k-1}} ) \Big]< +\infty,
\end{align*}
so that
\begin{align*}
\mathbb{P}-a.s. \quad \sup_{n \in \mathbb{N}^{\ast}} \sum_{k=1}^n  \frac{1}{H_k }\Big( \frac{\eta_k}{\gamma_k}-\frac{\eta_{k-1}}{\gamma_{k-1}} \Big)_+ (\psi \circ V)^{1/s}(\overline{X}_{\Gamma_{k-1}} ) < +\infty.
\end{align*}
By Kronecker's lemma, we deduce that
\begin{align*}
\mathbb{P}-a.s. \quad  \lim\limits_{n \to +\infty}  \frac{1}{H_n}\sum_{k=2}^n \Big( \frac{\eta_k}{\gamma_k}-\frac{\eta_{k-1}}{\gamma_{k-1}} \Big)_+ (\psi \circ V )^{1/s}(\overline{X}_{\Gamma_{k-1}})=0.
\end{align*}
%
This concludes the study of the first term and now we focus on the second one. From Kronecker lemma, it remains to prove the almost sure convergence of the martingale $(M_n)_{n \in \mathbb{N}^{\ast}}$ defined by $M_0:=0$ and 
\begin{equation*}
M_n:= \sum_{k=1}^{n} \frac{\eta_k}{\gamma_k H_k} \big((\psi \circ V )^{1/s}(\overline{X}_{\Gamma_k})- \mathscr{Q}_{\gamma_k}(\psi \circ V )^{1/s}(\overline{X}_{\Gamma_{k-1}}) \big), \quad n \in \mathbb{N}^{\ast}.
\end{equation*}
Using the Chow's theorem (see Theorem \ref{th:chow}), this $a.s.$ convergence is a direct consequence of the $a.s.$ finiteness of the series
\begin{equation*}
\sum_{n=1}^{\infty} \Big( \frac{\eta_n}{\gamma_n H_n} \Big)^{\rho}  \mathbb{E}[\vert  (\psi \circ V )^{1/s}(\overline{X}_{\Gamma_n})-\mathscr{Q}_{\gamma_n} (\psi \circ V )^{1/s}(\overline{X}_{\Gamma_{n-1}})\vert^{\rho} \vert \overline{X}_{\Gamma_{n-1}}  ] ,
\end{equation*}
which follows from $\mathcal{GC}_{Q}((\psi \circ V )^{1/s},g ,\rho ,\epsilon_{\mathcal{I}})$ (see (\ref{hyp:incr_X_Lyapunov})) and $\mathcal{S}\mathcal{W}_{\mathcal{I}, \gamma,\eta}( g,\rho,\epsilon_{\mathcal{I}}) $ (see (\ref{hyp:step_weight_I_gen_chow})).\\
Now, we focus on the proof of point \ref{th:tightness_point_B} Using $\mathcal{RC}_{Q,V}(\psi,\phi,\alpha,\beta) (i)$ (see (\ref{hyp:incr_sg_Lyapunov})), there exists $n_0 \in \mathbb{N}^{\ast}$, such that for every $n \geqslant n_0$, we have
\begin{equation*}
\mathbb{E}\Big[ \frac{\psi \circ V (\overline{X}_{\Gamma_{n+1}})}{\psi \circ V (\overline{X}_{\Gamma_n})} \Big\vert \overline{X}_{\Gamma_n}\Big] \leqslant 1 + \gamma_{n+1} \frac{\beta- \alpha \phi \circ V (\overline{X}_{\Gamma_n}) }{V(\overline{X}_{\Gamma_n})}.
\end{equation*}
Since the function defined on $\mathbb{R}_+^{\ast}$ by $y \mapsto y^{1/s}$ is concave and $C_{\phi}:= \sup_{y \in [v_{\ast},\infty )}\phi(y)/y< +\infty$, for $n$ large enough we use the Jensen's inequality and we derive
\begin{align*}
\mathbb{E}\Big[ \Big( \frac{\psi \circ V (\overline{X}_{\Gamma_{n+1}})}{\psi \circ V (\overline{X}_{\Gamma_n})} \Big)^{1/s} \Big\vert \overline{X}_{\Gamma_n} \Big] \leqslant & \Big( 1 + \gamma_{n+1} \frac{\beta- \alpha \phi \circ V(\overline{X}_{\Gamma_n}) }{V(\overline{X}_{\Gamma_n})} \Big)^{1/s} \\
\leqslant & 1+\frac{\gamma_{n+1}(\beta-\alpha \phi \circ V (\overline{X}_{\Gamma_n}))}{sV(\overline{X}_{\Gamma_n})}.
\end{align*}
%
%
%
%
Now when $\beta \geqslant 0$, by $\mathcal{RC}_{Q,V}(\psi,\phi,\alpha,\beta) (ii)$ (see (\ref{hyp:incr_sg_Lyapunov})), there exists $\lambda \in (0,1)$ and $y_{\lambda} \in (0,+\infty)$ such that for every $y >y_{\lambda}$, then $\phi(y)\geqslant \beta /(\lambda \alpha)$. It follows that the Borel function $C_{\lambda,s}:[v_{\ast},+\infty)\to \mathbb{R}$, $y \mapsto C_{\lambda,s}(y):=y^{-1} \psi (y)^{1/s} (\beta-\lambda \alpha  \phi (y))$ is locally bounded on $[v_{\ast},+\infty)$ and non positive on $[y_{\lambda},+\infty)$, hence $\overline{C}_{\lambda,s}:=\sup_{y \in [v_{\ast}, +\infty)}  C_{\lambda,s}(y) <+\infty$. When $\beta<0$, since $\phi$ and $\psi$ are positive functions, then the function $C_{\lambda,s}$ is non positive and it follows that
\begin{align*}
\mathscr{Q}_{\gamma_{n+1}}(\psi \circ V )^{1/s} (\overline{X}_{\Gamma_n})\leqslant & (\psi \circ V )^{1/s} (\overline{X}_{\Gamma_n}) \\
&+\frac{\gamma_{n+1}}{s}(C_{\lambda,s} \circ V(\overline{X}_{\Gamma_n})-(1- \lambda) \alpha \tilde{V}_{\psi,\phi,s}(\overline{X}_{\Gamma_n})) ,
\end{align*}
which yields,
\begin{equation*}
\tilde{V}_{\psi,\phi,s}(\overline{X}_{\Gamma_n})\leqslant -\frac{s}{\alpha (1-\lambda)} \widetilde{A}_{\gamma_{n+1}}(\psi \circ V )^{1/s}(\overline{X}_{\Gamma_n})+\frac{\overline{C}_{\lambda,s} \vee 0}{\alpha(1-\lambda)}.
\end{equation*}
Consequently (\ref{eq:tightness_gen}) follows from (\ref{eq:invariance_mes_emp_Lyap_gen}). The tightness of $(\nu_n^{\eta})_{n \in \mathbb{N}^{\ast}}$ is a immediate consequence of (\ref{eq:tightness_gen}) since $\lim\limits_{ x \to \infty}\tilde{V}_{\psi,\phi,s}(x)=+\infty$.
%
%
\end{proof}
\subsubsection{Identification of the limit}
In Theorem \ref{th:tightness}, we obtained the tightness of $(\nu_n^{\eta})_{n \in \mathbb{N}^{\ast}}$. It remains to show that every limiting point of this sequence is an invariant distribution of the Feller process with infinitesimal generator $A$. This is the interest of the following Theorem which relies on the infinitesimal generator approximation.
\begin{mytheo}
\label{th:identification_limit}
Let $\rho \in [1,2]$. We have the following properties:
\begin{enumerate}[label=\textbf{\Alph*.}]
\item\label{th:identification_limit_A} 
Let $\DomA_0 \subset \DomA$, with  $\DomA_0$ dense in $\mathcal{C}_0(E)$. We assume that $\widetilde{A}_{\gamma_n}f$ exists for every $f \in \DomA_0$ and every $n \in \mathbb{N}^{\ast}$. Also assume that there exists $g:E \to \mathbb{R}_+$ a Borel function and $\epsilon_{\mathcal{I}} : \mathbb{R}_+ \to \mathbb{R}_+$ an increasing function such that $\mathcal{GC}_{Q}(\DomA_0,g,\rho ,\epsilon_{\mathcal{I}}) $ (see (\ref{hyp:incr_X_Lyapunov})) and $\mathcal{S}\mathcal{W}_{\mathcal{I}, \gamma,\eta}( g,\rho,\epsilon_{\mathcal{I}}) $ (see (\ref{hyp:step_weight_I_gen_chow})) hold and that
 \begin{equation}
 \label{hyp:accroiss_sw_series_2}  \lim\limits_{n \to + \infty} \frac{1}{H_n} \sum_{k =1}^{n} \vert \eta_{k+1}/\gamma_{k+1}-\eta_k /\gamma_k \vert = 0.
\end{equation} Then
\begin{equation}
\label{hyp:identification_limit}
\mathbb{P} \mbox{-a.s.} \quad \forall f \in \DomA_0, \qquad   \lim\limits_{n \to + \infty}  \frac{1}{H_n} \sum_{k=1}^n \eta_k \widetilde{A}_{\gamma_k}f (\overline{X}_{\Gamma_{k-1}})=0.
\end{equation}
\item \label{th:identification_limit_B} 
We assume that (\ref{hyp:identification_limit}) and $\mathcal{E}(\widetilde{A},A,\DomA_0) $ (see (\ref{hyp:erreur_tems_cours_fonction_test_reg})) hold. Then
\begin{align*}
\mathbb{P} \mbox{-a.s.} \quad \forall f \in \DomA_0, \qquad     \lim\limits_{n \to + \infty} \nu_n^{\eta}( Af )=0.
\end{align*}
It follows that, $\mathbb{P}-a.s.$, every weak limiting distribution $\nu^{\eta}_{\infty}$ of the sequence $(\nu_n^{\eta})_{n \in \mathbb{N}^{\ast}}$ belongs to $\mathcal{V}$, the set of the invariant distributions of $(X_t)_{t \geqslant 0}$. Finally, if the hypothesis from Theorem \ref{th:tightness} point \ref{th:tightness_point_B} hold and $(X_t)_{t \geqslant 0}$ has a unique invariant distribution, $i.e.$ $\mathcal{V}=\{\nu\}$, then 
\begin{align}
\label{eq:test_function_gen_cv}
\mathbb{P} \mbox{-a.s.} \quad \forall f \in \mathcal{C}_{\tilde{V}_{\psi,\phi,s}}(E), \quad \lim\limits_{n \to + \infty} \nu_n^{\eta}(f)=\nu(f),
\end{align}
 with $\mathcal{C}_{\tilde{V}_{\psi,\phi,s}}(E)$ defined in (\ref{def:espace_test_function_cv}).
\end{enumerate}
\end{mytheo}
In the particular case where the function $\psi$ is polynomial, (\ref{eq:test_function_gen_cv}) also reads as the $a.s.$ convergence of the empirical measures for some $\mbox{L}^p$-Wasserstein distances, $p>0$, that we will study further in this paper for some numerical schemes of some diffusion processes. From the liberty granted by the choice of $\psi$ in this abstract framework, where only a recursive control with mean reverting is required, we will also propose an application for functions $\psi$ with exponential growth.
\begin{proof}
We prove point \ref{th:identification_limit_A} We write 
\begin{align*}
 -\sum_{k=1}^n \eta_k \widetilde{A}_{\gamma_k}f(\overline{X}_{\Gamma_{k-1}}) = & -\sum_{k=1}^n  \frac{\eta_k}{\gamma_k}(f(\overline{X}_{\Gamma_k})-f(\overline{X}_{\Gamma_{k-1}})) \\
 &+  \sum_{k=1}^n \frac{\eta_k}{\gamma_k} (f(\overline{X}_{\Gamma_k})-\mathscr{Q}_{\gamma_k}f(\overline{X}_{\Gamma_{k-1}}) )
\end{align*}
We study the first term of the $r.h.s.$ We derive by an Abel transform that
\begin{align*}
 - \frac{1}{H_n}\sum_{k=1}^n  \frac{\eta_k}{\gamma_k}  (f(\overline{X}_{\Gamma_k})- f(\overline{X}_{\Gamma_{k-1}}))= & \frac{\eta_1}{H_n \gamma_1}f(\overline{X}_{0})-\frac{\eta_n}{H_n \gamma_n}f(\overline{X}_{\Gamma_n}) \\
 &+ \frac{1}{H_n}\sum_{k=2}^n \Big( \frac{\eta_k}{\gamma_k}-\frac{\eta_{k-1}}{\gamma_{k-1}} \Big)f(\overline{X}_{\Gamma_{k-1}}).\\
\end{align*}
Since $f$ is bounded and $\lim\limits_{n \to + \infty} \eta_n/(H_n \gamma_n) =0 $, we deduce that $\lim\limits_{n \to + \infty} \eta_n f(\overline{X}_{\Gamma_n})/(H_n \gamma_n) \overset{a.s.}{=}0 $ and, on the other hand, we deduce from (\ref{hyp:accroiss_sw_series_2}) that  
\begin{align*}
\lim\limits_{n \to + \infty} \frac{1}{H_n}\sum_{k=1}^n  \frac{\eta_k}{\gamma_k}(f(\overline{X}_{\Gamma_k})-f(\overline{X}_{\Gamma_{k-1}})) =0 .
\end{align*}
This completes the study of the first term. To treat the second term, the approach is quite similar to the one in the proof of Theorem \ref{th:tightness} point \ref{th:tightness_point_A} using $\mathcal{GC}_{Q}(\DomA,g,\rho ,\epsilon_{\mathcal{I}})$ (see (\ref{hyp:incr_X_Lyapunov})) with $\mathcal{S}\mathcal{W}_{\mathcal{I}, \gamma,\eta}( g,\rho,\epsilon_{\mathcal{I}}) $ (see (\ref{hyp:step_weight_I_gen_chow})). Details are left to the reader.
Now, we focus on the proof of point \ref{th:identification_limit_B} First we write
\begin{equation*}
 \frac{1}{H_n} \sum_{k=1}^n \eta_k \widetilde{A}_{\gamma_k}f (\overline{X}_{\Gamma_{k-1}} )-\nu_n^{\eta} (Af) =\frac{1}{H_n}\sum_{k=1}^n  \eta_k \big( \widetilde{A}_{\gamma_k}f(\overline{X}_{\Gamma_{k-1}})-A f(\overline{X}_{\Gamma_{k-1}}) \big).
\end{equation*}
 Now we use the short time approximation $\mathcal{E}(\widetilde{A},A,\DomA_0) $ (see (\ref{hyp:erreur_tems_cours_fonction_test_reg})) and it follows that,
\begin{align*}
\Big \vert \frac{1}{H_n}\sum_{k=1}^n  \eta_k (\widetilde{A}_{\gamma_k}f(\overline{X}_{\Gamma_{k-1}})-A f(\overline{X}_{\Gamma_{k-1}})) \Big \vert  \leqslant \frac{1}{H_n}\sum_{k=1}^n \eta_k  \Lambda_{f}(\overline{X}_{\Gamma_{k-1}},\gamma_k) .
\end{align*}
Moreover, we have the following decomposition:
\begin{align*}
\forall f \in \DomA_0, \forall x \in E , \forall \gamma \in [0,\overline{\gamma}], \qquad \Lambda_f(x,\gamma)= \langle g (  x ) ,\tilde{\mathbb{E}} [\tilde{\Lambda}_{f}(x,\gamma)]  \rangle_{\mathbb{R}^q}
\end{align*}
with $g : (E, \mathcal{B}(E))\to \mathbb{R}_+^{q}$, $q \in \mathbb{N}$, a locally bounded Borel measurable function and $\tilde{\Lambda}_{f}:(E\times \mathbb{R}_+ \times \tilde{\Omega}, \mathcal{B}(E) \otimes \mathcal{B}(\mathbb{R}_+) \otimes \tilde{\mathcal{G}}) \to \mathbb{R}_+^{q}$ a measurable function such that  $\sup_{i \in \{1,\ldots,q\} }  \tilde{\mathbb{E}}[\sup_{x \in E} \sup_{\gamma \in (0,\overline{\gamma}] }\tilde{\Lambda}_{f,i}(x,\gamma) ]< + \infty$. Since for every $i \in \{1,\ldots,q\}$, $\sup_{n \in \mathbb{N}^{\ast}} \nu_n^{\eta}( g_i ,\omega )< + \infty$, $\mathbb{P}(d\omega)-a.s.$, the $\mathbb{P}(d\omega)-a.s.$ convergence of $\frac{1}{H_n}\sum_{k=1}^n \eta_k  \Lambda_{f}(\overline{X}_{\Gamma_{k-1}},\gamma_k) $ towards zero for every $f \in \DomA_0$, will follow from the following result: Let $(\overline{x}_{n})_{n \in \mathbb{N}} \in E^{\otimes \mathbb{N}} $. If 
\begin{align*}
\sup_{i \in \{1,\ldots,q \}}\sup_{n \in \mathbb{N}^{\ast}} \frac{1}{H_n}\sum_{k=1}^n \eta_k  g_i(\overline{x}_{{k-1}})<+ \infty,
\end{align*} then, for every $f \in \DomA_0$, $\lim\limits_{n\to +\infty} \frac{1}{H_n}\sum_{k=1}^n \eta_k  \Lambda_{f}(\overline{x}_{{k-1}},\gamma_k) =0$. In order to obtain this result, we first show that, for every $f \in \DomA_0$, every $i \in \{1,\ldots,q \}$, and every $(\overline{x}_{n})_{n \in \mathbb{N}} \in E^{\otimes \mathbb{N}} $, then
\begin{equation*}
\tilde{\mathbb{P}}(d\tilde{\omega})-a.s. \quad\quad \lim\limits_{n \to + \infty} \frac{1}{H_n}\sum_{k=1}^n \eta_k  \tilde{\Lambda}_{f,i}(\overline{x}_{{k-1}},\gamma_k,\tilde{\omega}) g_i (\overline{x}_{{k-1}}) =0  ,
\end{equation*}
and the result will follow from the Dominated Convergence theorem since, for every $n \in \mathbb{N}^{\ast}$, 
\begin{align*}
 \frac{1}{H_n}\sum_{k=1}^n \eta_k  \tilde{\Lambda}_{f,i} & (\overline{x}_{{k-1}},\gamma_k,\tilde{\omega}) g_i (\overline{x}_{{k-1}}) \\
 &\leqslant \sup_{x \in E} \sup_{\gamma \in (0,\overline{\gamma}]}\tilde{\Lambda}_{f,i} (x,\gamma,\tilde{\omega})    \sup_{n \in \mathbb{N}^{\ast}} \frac{1}{H_n}\sum_{k=1}^n \eta_k  g_i(\overline{x}_{{k-1}})< +\infty .
\end{align*}
with $ \tilde{\mathbb{E}}  [  \sup_{x \in E} \sup_{\gamma \in (0,\overline{\gamma}]}\tilde{\Lambda}_{f,i} (x,\gamma,\tilde{\omega})   ]< +\infty$ and $\sup_{n \in \mathbb{N}^{\ast}} \frac{1}{H_n}\sum_{k=1}^n \eta_k  g_i(\overline{x}_{{k-1}})<+ \infty$.
%
%
%
We fix $f \in \DomA_0$, $i \in \{1,\ldots,q \}$ and $(\overline{x}_{n})_{n \in \mathbb{N}} \in E^{\otimes \mathbb{N}} $ and we assume that $\mathcal{E}(\widetilde{A},A,\DomA_0)$ \ref{hyp:erreur_tems_cours_fonction_test_reg_Lambda_representation_1} (see (\ref{hyp:erreur_temps_cours_fonction_test_reg_Lambda_representation_2_1})) holds for $\tilde{\Lambda}_{f,i}$ and $g_i$. If instead $\mathcal{E}(\widetilde{A},A,\DomA_0)$ \ref{hyp:erreur_temps_cours_fonction_test_reg_Lambda_representation_2} (see (\ref{hyp:erreur_temps_cours_fonction_test_reg_Lambda_representation_2_2})) is satisfied, the proof is similar but simpler so we leave it to the reader. 
By assumption $\mathcal{E}(\widetilde{A},A,\DomA_0) $ \ref{hyp:erreur_tems_cours_fonction_test_reg_Lambda_representation_1} (ii)(see (\ref{hyp:erreur_temps_cours_fonction_test_reg_Lambda_representation_2_2})), $\tilde{\mathbb{P}}(d\tilde{\omega})-a.s$, for every $R>0$, there exists $K_R(\tilde{\omega}) \in \mathcal{K}_E$ such that $ \sup_{x \in K_{R}^c(\tilde{\omega}) }\sup_{\gamma \in (0, \underline{\gamma}(\tilde{\omega}) ]} \tilde{\Lambda}_{f,i}(x, \gamma,\tilde{\omega})<1/R$.
Then from $\mathcal{E}(\widetilde{A},A,\DomA_0) $ \ref{hyp:erreur_tems_cours_fonction_test_reg_Lambda_representation_1} (i)(see (\ref{hyp:erreur_temps_cours_fonction_test_reg_Lambda_representation_2_1})), we derive that, $\tilde{\mathbb{P}}(d\tilde{\omega})-a.s$, for every $R>0$, $\lim\limits_{n \to + \infty} \tilde{\Lambda}_{f,i}(\overline{x}_{{n-1}},\gamma_n ,\tilde{\omega} ) \mathds{1}_{  K_ R(\tilde{\omega}) }(\overline{x}_{{k-1}})=0, \;  $ Then, since $g_i $ is a locally bounded function, as an immediate consequence of the Cesaro's lemma, we obtain
\begin{align*}
\tilde{\mathbb{P}}(d \tilde{\omega})-a.s. & \quad \forall R>0, \\
&\lim\limits_{n \to + \infty} \frac{1}{H_n}\sum_{k=1}^n \eta_k  \tilde{\Lambda}_{f,i}(\overline{x}_{{k-1}},\gamma_k,\tilde{\omega}) g_i (\overline{x}_{{k-1}}) \mathds{1}_{   K_ R(\tilde{\omega}) }(\overline{x}_{{k-1}}) =0  
\end{align*}
Let $\underline{n}(\tilde{\omega}) :=\inf\{n \in \mathbb{N}^{\ast},\sup_{k \geqslant n} \gamma_{k} \leqslant \underline{\gamma}(\tilde{\omega}) \}$. By the assumption $\mathcal{E}(\widetilde{A},A,\DomA_0) $ \ref{hyp:erreur_tems_cours_fonction_test_reg_Lambda_representation_1} (ii) (see (\ref{hyp:erreur_temps_cours_fonction_test_reg_Lambda_representation_2_1})), we have, $\tilde{\mathbb{P}}(d\tilde{\omega})-a.s$, $ \lim\limits_{\vert x \vert  \to +\infty} \sup_{n \geqslant \underline{n}(\tilde{\omega}) }\tilde{\Lambda}_{f,i}(x,\gamma_n,\tilde{\omega} )=0, $ Moreover,
\begin{align*}
 \sup_{n \geqslant \underline{n}(\tilde{\omega}) } \frac{1}{H_n}\sum_{k=\underline{n}(\tilde{\omega}) }^n \eta_k   \tilde{\Lambda}_{f,i}& (\overline{x}_{{k-1}},\gamma_k,\tilde{\omega})  g (\overline{x}_{{k-1}})  \mathds{1}_{K_{R}^c(\tilde{\omega})}(\overline{x}_{{k-1}}) \\
 \leqslant &    \sup_{x \in K_{R}^c(\tilde{\omega})} \sup_{\gamma \in (0,\underline{\gamma}(\tilde{\omega})]}  \tilde{\Lambda}_{f,i} (x,\gamma,\tilde{\omega})  \sup_{n \in \mathbb{N}^{\ast}} \frac{1}{H_n}\sum_{k=1}^n \eta_k  g_i(\overline{x}_{{k-1}}).
\end{align*}
We let $R$ tends to infinity and since $\sup_{n \in \mathbb{N}^{\ast}} \frac{1}{H_n}\sum_{k=1}^n \eta_k  g_i(\overline{x}_{{k-1}})< + \infty$, the $l.h.s.$ of the above equation converges $\tilde{\mathbb{P}}(d\tilde{\omega})-a.s.$ to 0. Finally, since $\underline{n}(\tilde{\omega}) $ is $\tilde{\mathbb{P}}(d\tilde{\omega})-a.s.$ finite, we also have
\begin{align*}
 \tilde{\mathbb{P}}(d \tilde{\omega})-a.s.& \quad \forall R>0, \\
& \lim_{n \to + \infty}\frac{1}{H_n}\sum_{k=1}^{\underline{n}(\tilde{\omega}) -1} \eta_k   \tilde{\Lambda}_{f,i}(\overline{x}_{{k-1}},\gamma_k,\tilde{\omega}) g (\overline{x}_{{k-1}})  \mathds{1}_{K_{R}^c(\tilde{\omega})}(\overline{x}_{{k-1}}) =0.
\end{align*}
Applying the same approach for every $i \in \{1,\ldots,q \}$, the Dominated Convergence Theorem yields:
\begin{align*}
\forall (\overline{x}_{n})_{n \in \mathbb{N}} \in E^{\otimes \mathbb{N}} , \forall f \in \DomA_0, \qquad \lim\limits_{n \to + \infty} \frac{1}{H_n}\sum_{k=1}^n  \eta_k \Lambda_{f}(\overline{x}_{{k-1}},\gamma_k) = 0.
\end{align*}
and since for every $i \in \{1,\ldots,q\}$, $\sup_{n \in \mathbb{N}^{\ast}} \nu_n^{\eta}( g_i ,\omega )< + \infty, \; \mathbb{P}(d\omega)-a.s.$, then
\begin{align*}
\mathbb{P}(d \omega)-a.s. \qquad \forall f \in \DomA_0, \quad \frac{1}{H_n}\sum_{k=1}^n \eta_k  (\widetilde{A}_{\gamma_k}f(\overline{X}_{\Gamma_{k-1}}) - Af(\overline{X}_{\Gamma_{k-1}}) ) =0.
\end{align*}
It follows that $\mathbb{P}(d \omega)-a.s.$, for every $f \in \DomA_0$, $\lim\limits_{n \to + \infty} \nu_n^{\eta}(Af)=0$. The conclusion follows from the Echeverria Weiss theorem (see Theorem \ref{th:eche_weiss_feller}). Simply notice that  we maintain the assumptions of this Theorem when $\DomA$ is replaced by $\DomA_0$, since $\DomA_0 \subset \DomA$ and $\DomA_0$ is dense in $\mathcal{C}_0(E)$.
\end{proof}
\subsection{About Growth control and Step Weight assumptions}
The following Lemma presents a $\mbox{L}_1$-finiteness property that we can obtain under recursive control hypothesis and strongly mean reverting assumptions ($\phi=I_d$). This result is thus useful to prove $\mathcal{S}\mathcal{W}_{\mathcal{I}, \gamma,\eta}(g, \rho , \epsilon_{\mathcal{I}})$ (see (\ref{hyp:step_weight_I_gen_chow})) or $\mathcal{S}\mathcal{W}_{\mathcal{II},\gamma,\eta}(F)  $ (see (\ref{hyp:step_weight_I_gen_tens})) for well chosen $F$ and $g$ in this specific situation.
\begin{lemme}
\label{lemme:mom_psi_V}
Let $v_{\ast}>0$, $V:E \to [v_{\ast},\infty) $, $\psi:[v_{\ast},\infty) \to \mathbb{R}_+, $ such that $\widetilde{A}_{\gamma_n}\psi \circ V$ exists for every $n \in \mathbb{N}^{\ast}$. Let $\alpha>0$ and $\beta \in \mathbb{R}$. We assume that $\mathcal{RC}_{Q,V}(\psi,I_d,\alpha,\beta)$ (see (\ref{hyp:incr_sg_Lyapunov})) holds and that $\mathbb{E}[\psi\circ V (\overline{X}_{\Gamma_{n_0}})]< + \infty$ for every $n_0 \in \mathbb{N}^{\ast}$. Then
\begin{align}
\label{eq:mom_psi_V}
\sup_{n \in \mathbb{N}} \mathbb{E}[\psi \circ V(\overline{X}_{\Gamma_n})] < + \infty
\end{align}
In particular, let $\rho \in [1,2]$ and $\epsilon_{\mathcal{I}} : \mathbb{R}_+ \to \mathbb{R}_+$, an increasing function. It follows that if $\sum_{n=1}^{\infty} \Big \vert \frac{\eta_n }{H_n \gamma_n } \Big \vert^{\rho} \epsilon_{\mathcal{I}}(\gamma_n)   < + \infty$, then $\mathcal{S}\mathcal{W}_{\mathcal{I}, \gamma,\eta}(\psi \circ V, \rho , \epsilon_{\mathcal{I}}) $ holds and if $ \sum_{n=0}^{\infty} \frac{(\eta_{n+1} /\gamma_{n+1}-\eta_n /\gamma_n)_+ }{H_{n+1} } < + \infty$, then $\mathcal{S}\mathcal{W}_{\mathcal{II},\gamma,\eta}(\psi \circ V) $ is satisfied
\end{lemme}
 \begin{proof} 
First, we deduce from $\mathcal{RC}_{Q,V}(\psi,I_d,\alpha,\beta) (i)$ that there exists $n_0 \in \mathbb{N}$ such that for $n \geqslant n_0$, $\mathcal{RC}_{Q,V}(\psi,I_d,\alpha,\beta)$ can be rewritten
 \begin{align*}
 \mathbb{E}[\psi \circ V(\overline{X}_{\Gamma_{n+1}}) \vert \overline{X}_{\Gamma_n} ]   \leqslant &\psi \circ V(\overline{X}_{\Gamma_n} )+  \gamma_{n+1}\frac{ \psi \circ V(\overline{X}_{\Gamma_n} )}{V(\overline{X}_{\Gamma_n} )} (\beta-\alpha  V(\overline{X}_{\Gamma_n}) )
 \end{align*}
Now, let $\lambda \in (0,1)$ and $y_{\lambda} = \beta /(\lambda \alpha)$. It follows that the Borel function $C_{\lambda}:[v_{\ast},+\infty)\to \mathbb{R}$, $y \mapsto C_{\lambda}(y):=y^{-1} \psi (y)(\beta-\lambda \alpha  y)$ is locally bounded on $[v_{\ast},+\infty)$ and non positive on $[y_{\lambda},+\infty)$, hence $\overline{C}_{\lambda}:=\sup_{y \in [v_{\ast}, y_{\lambda})}  C_{\lambda}(y) <+\infty$ and
\begin{align*}
 \mathbb{E}[\psi \circ V(X_{\Gamma_{n+1}}) \vert X_{\Gamma_n} ]   \leqslant & \psi \circ V  (\overline{X}_{\Gamma_n})+\gamma_{n+1}(C_{\lambda} \circ V(\overline{X}_{\Gamma_n})-(1- \lambda)\alpha \psi \circ V(\overline{X}_{\Gamma_n})) ,\\
  \leqslant &  \psi \circ V(\overline{X}_{\Gamma_n})(1- \gamma_{n+1}(1- \lambda) \alpha) + \gamma_{n+1} \overline{C}_{\lambda}.
\end{align*}
Applying a simple induction we deduce that $\mathbb{E}[ \psi \circ V (X_{\Gamma_n}) ] \leqslant \mathbb{E}[ \psi \circ V (X_{n_0}) ]  \vee \frac{ \overline{C}_{\lambda}}{(1- \lambda) \alpha}$.
\end{proof}
Now, we provide a general way to obtain $\mathcal{S}\mathcal{W}_{\mathcal{I}, \gamma,\eta}(g,\rho, \epsilon_{\mathcal{I}})$ and $\mathcal{S}\mathcal{W}_{\mathcal{I}\mathcal{I},\gamma,\eta}(F)$  for some specific $g$ and $F$ as soon as a recursive control with weakly mean reversion assumption holds.
\begin{lemme}
\label{lemme:mom_V}
Let $v_{\ast}>0$, $V:E \to [v_{\ast},\infty) $, $\psi, \phi :[v_{\ast},\infty) \to \mathbb{R}_+, $ such that $\widetilde{A}_{\gamma_n}\psi \circ V$ exists for every $n \in \mathbb{N}^{\ast}$. Let $\alpha>0$ and $\beta \in \mathbb{R}$. We also introduce the non-increasing sequence $(\theta_n)_{n \in \mathbb{N}^{\ast}}$ such that $\sum_{n \geqslant 1} \theta_n \gamma_n < + \infty$. We assume that $\mathcal{RC}_{Q,V}(\psi,\phi,\alpha,\beta)$ (see (\ref{hyp:incr_sg_Lyapunov})) holds and that $\mathbb{E}[\psi\circ V (\overline{X}_{\Gamma_{n_0}})]< + \infty$ for every $n_0 \in \mathbb{N}^{\ast}$. Then
\begin{equation*}
\sum_{n=1}^{\infty} \theta_n \gamma_n \mathbb{E} [\tilde{V}_{\psi,\phi,1}(\overline{X}_{\Gamma_{n-1}}) ] < + \infty
\end{equation*}
 with $ \tilde{V}_{\psi,\phi,1}$ defined in (\ref{def:espace_test_function_cv}). In particular, let $\rho \in [1,2]$ and $\epsilon_{\mathcal{I}} : \mathbb{R}_+ \to \mathbb{R}_+$, an increasing function. If we also assume 
 \begin{align}
 \label{hyp:step_weight_I}
\mathcal{S}\mathcal{W}_{\mathcal{I}, \gamma,\eta}(\rho, \epsilon_{\mathcal{I}})  \quad \equiv \qquad & \Big( \gamma_n^{-1} \epsilon_{\mathcal{I}}(\gamma_n) \big( \frac{\eta_n }{H_n \gamma_n } \big)^{\rho} \Big)_{n \in \mathbb{N}^{\ast}} \mbox{ is non-increasing and } \nonumber \\
& \sum_{n=1}^{\infty} \Big( \frac{\eta_n }{H_n \gamma_n } \Big)^{\rho} \epsilon_{\mathcal{I}}(\gamma_n) < + \infty,
 \end{align}
 then we have $\mathcal{S}\mathcal{W}_{\mathcal{I}, \gamma,\eta}(\tilde{V}_{\psi,\phi,1},\rho,\epsilon_{\mathcal{I}}) $ (see (\ref{hyp:step_weight_I_gen_chow})). Finally,if
  \begin{align}
 \label{hyp:step_weight_II}
\mathcal{S}\mathcal{W}_{\mathcal{II},\gamma,\eta}  \quad \equiv \qquad &\Big( \frac{ \frac{\eta_{n+1} }{(\gamma_{n+1}}-\frac{\eta_n}{\gamma_n})_+ }{ \gamma_n H_n }  \Big)_{n \in \mathbb{N}^{\ast}} \mbox{ is non-increasing and } \nonumber \\
&\sum_{n=1}^{\infty}  \frac{(\eta_{n+1} /\gamma_{n+1}-\eta_n /\gamma_n)_+ }{H_n }  < + \infty,
 \end{align}
  then we have $\mathcal{S}\mathcal{W}_{\mathcal{II},\gamma,\eta}( \tilde{V}_{\psi,\phi,1}) $ (see (\ref{hyp:step_weight_I_gen_tens})).
\end{lemme}
\begin{proof}
 Now when $\beta \geqslant 0$, by $\mathcal{RC}_{Q,V}(\psi,\phi,\alpha,\beta) (ii)$ (see (\ref{hyp:incr_sg_Lyapunov})), there exists $\lambda \in (0,1)$ and $y_{\lambda} \in (0,+\infty)$ such that for every $y >y_{\lambda}$, then $\phi(y)\geqslant \beta /(\lambda \alpha)$. It follows that the Borel function $C_{\lambda,s}:[v_{\ast},+\infty)\to \mathbb{R}$, $y \mapsto C_{\lambda,s}(y):=y^{-1} \psi (y) (\beta-\lambda \alpha  \phi (y))$ is locally bounded on $[v_{\ast},+\infty)$ and non positive on $[y_{\lambda},+\infty)$, hence $\overline{C}_{\lambda}:=\sup_{y \in [v_{\ast}, +\infty)}  C_{\lambda}(y) <+\infty$. When $\beta<0$, since $\phi$ and $\psi$ are positive functions, then the function $C_{\lambda}$ is non positive. Using the same approach as in the proof of Theorem \ref{th:tightness} point \ref{th:tightness_point_B}, we deduce that there exists $n_0 \in \mathbb{N}$  such that we have the following telescopic decomposition:
\begin{align*}
\forall n \geqslant n_0, \quad \theta_{n+1}\gamma_{n+1}\tilde{V}_{\psi,\phi,1}(\overline{X}_{\Gamma_n}) \leqslant &\theta_{n+1} \frac{  \psi \circ V(X_{\Gamma_n}) - \mathbb{E}[\psi \circ V(\overline{X}_{\Gamma_{n+1}}) \vert \overline{X}_{\Gamma_n} ] }{ \alpha(1- \lambda)} \\
&+\gamma_{n+1}\theta_{n+1}\frac{ \overline{C}_{\lambda}} {\alpha(1- \lambda)} \\
\leqslant &  \frac{ \theta_{n}  \psi \circ V(\overline{X}_{\Gamma_n}) - \theta_{n+1}\mathbb{E}[\psi \circ V(\overline{X}_{\Gamma_{n+1}}) \vert \overline{X}_{\Gamma_n} ] }{ \alpha(1- \lambda)} \\
&+\gamma_{n+1}\theta_{n+1} \frac{ \overline{C}_{\lambda}} {\alpha(1- \lambda)}.
 \end{align*}
where the last inequality follows from the fact that the sequence $(\theta_n)_{n \in \mathbb{N}^{\ast}}$ is non-increasing. Taking expectancy and summing over $n$ yields the result as $\psi$ takes positive values and $\mathbb{E}[\psi \circ V(X_{n_0})]< + \infty$ for every $n_0 \in \mathbb{N}^{\ast}$. 
\end{proof}
This result concludes the general approach in a generic framework to prove convergence. The next part of this paper is dedicated to various applications.
\subsection{Example - The Euler scheme}
Using this abstract approach, we recover the results obtained in \cite{Lamberton_Pages_2002} and \cite{Lamberton_Pages_2003} for the Euler scheme of a $d$-dimensional Brownian diffusion. We consider a $N$-dimensional Brownian motion $(W_t)_{t \geqslant 0}$. We are interested in the strong solution - assumed to exist and to be unique - of the $d$-dimensional stochastic equation
\begin{align}
\label{eq:EDS_Euler_dim_d}
X_t= x+\int_0^t b(X_{s})ds + \int_0^t  \sigma(X_{s})  dW_s
\end{align}
where $b:\mathbb{R}^d \to \mathbb{R}^d$, $\sigma:\mathbb{R}^d \to \mathbb{R}^{d \times N}$. Let $V:\mathbb{R} \to [1,+\infty)$, the Lyapunov function of this system such that $\mbox{L}_V$ (see (\ref{hyp:Lyapunov})) holds with $E=\mathbb{R}^d $, and
\begin{align*}
\vert \nabla V \vert^2 \leqslant C_V V, \qquad \Vert D^2 V \Vert_{\infty} < + \infty.
\end{align*}
Moreover, we assume that for every $x \in \mathbb{R},  \quad  \vert b(x) \vert^2 + \mbox{Tr}[\sigma \sigma^{\ast}(x)]   \leqslant V^a(x)$ for some $a \in (0,1]$. Finally, for $p \geqslant 1$, we introduce the following $\mbox{L}_p$-mean reverting property of $V$,
\begin{align*}
\exists \alpha > 0, \beta \in \mathbb{R}, &\forall x \in \mathbb{R},  \nonumber \\   
&\langle \nabla V(x), b(x) \rangle+ \frac{1}{2}   \Vert \lambda_{p} \Vert_{\infty} 2^{(2p-3)_+} \mbox{Tr}[\sigma \sigma^{\ast}(x)]   \leqslant \beta - \alpha V^a (x )
\end{align*}
with for every $x \in \mathbb{R}^d$, $\lambda_p(x):= \sup\{\lambda_{p,1}(x),\ldots,\lambda_{p,d}(x),0 \}$, with $\lambda_{p,i}(x)$ the $i$-th eigenvalue of the matrix $D^2V(x)+2(p-1) \nabla V(x)^{\otimes 2} / V(x) $. We now introduce the Euler scheme of $(X_t)_{t \geqslant 0}$. Let $\rho \in [1,2]$ and $\epsilon_{\mathcal{I}}(\gamma)=\gamma^{\rho/2}$ and assume that (\ref{hyp:accroiss_sw_series_2}), $\mathcal{S}\mathcal{W}_{\mathcal{I}, \gamma,\eta}(\rho, \epsilon_{\mathcal{I}}) $ (see (\ref{hyp:step_weight_I})) and $\mathcal{S}\mathcal{W}_{\mathcal{II},\gamma,\eta}$ (see (\ref{hyp:step_weight_II})) hold. Let $(U_n)_{n}$ be a sequence of $\mathbb{R}^N$-valued centered independent and identically distributed random variables with covariance identity and bounded moments of order $2p$. We define the Euler scheme with decreasing steps $(\gamma_n)_{n \in \mathbb{N}^{\ast}}$, $(\overline{X}_{\Gamma_n} )_{n \in \mathbb{N}}$ of $(X_t)_{t \geqslant 0}$ (\ref{eq:EDS_Euler_dim_d}) on the time grid $\{\Gamma_n=\sum_{k=1}^n \gamma_k, n \in \mathbb{N} \}$ by
\begin{align*}
\forall n \in \mathbb{N} , \qquad  \overline{X}_{\Gamma_{n+1}}  =& \overline{X}_{\Gamma_n} + \gamma_{n+1} b(\overline{X}_{\Gamma_{n}}) + \sqrt{\gamma_{n+1}} \sigma(\overline{X}_{\Gamma_{n}})U_{n+1},\quad \overline{X}_{0} =x.
\end{align*}
We consider $ (\nu^{\eta}_n(dx,\omega))_{n \in \mathbb{N}^{\ast}}$ defined as in (\ref{eq:def_weight_emp_meas}) with $(\overline{X}_{\Gamma_n} )_{n \in \mathbb{N}}$ defined above. Now,we specify the measurable functions $\psi,\phi:[1,+\infty) \to [1,+\infty)$ as $\psi(y)=y^{p}$ and $\phi(y)=y^a$. Moreover, let $s \geqslant 1$ such that $a\,p\rho /s \leqslant p+a-1$, $p/s+a-1>0$ and $\Tr[\sigma \sigma^{\ast}] \leqslant CV^{p/s+a-1}$. Then, it follows from Theorem \ref{th:identification_limit} that there exists an invariant distribution $\nu$ for $(X_t)_{t \geqslant 0}$. Moreover, $ (\nu^{\eta}_n(dx,\omega))_{n \in \mathbb{N}^{\ast}}$ $a.s.$ weakly converges toward $\mathcal{V}$, the set of invariant distributions of $(X_t)_{t \geqslant 0}$ and when it is unique $i.e.$ $\mathcal{V}=\{\nu\}$, we have
\begin{align*}
\lim_{n \to + \infty}\nu^{\eta}_n(f)= \nu(f),
\end{align*}
for every $\nu-a.s.$ continuous function $f\in \mathcal{C}_{\tilde{V}_{\psi,\phi,s}}(\mathbb{R}^d)$ defined in (\ref{def:espace_test_function_cv}). Notice that this result was initially obtained in \cite{Lamberton_Pages_2002} when $a=1$ and in \cite{Lamberton_Pages_2003} when $a \in (0,1]$ and in both cases $s=\rho=2$. Afterwards, the study was extended in the case function $\psi$ with polynomial growth in \cite{Lemaire_2007}. We do not recall this result. However, in the sequel we prove the convergence of the empirical measures for both polynomial growth and exponential growth of $\psi$ for the Euler scheme of a Brownian Markov switching diffusions and those mentioned results can be recovered from a simplified version of our approach.
\section{Applications}
In this section, we propose some concrete applications which follow from the results presented in Section \ref{section:convergence_inv_distrib_gnl}. We first give Wasserstein convergence results concerning the Milstein scheme of a weakly mean reverting Brownian diffusion. Then, we propose a detailed application for the Euler scheme of a Markov Switching diffusion for test functions with polynomial or exponential growth. As a preliminary, we give some standard notations and properties that will be used extensively in the sequel. \\

\noindent First, for $\alpha \in (0,1]$ and $f$ an $\alpha$-H\"older function we denote $[f]_{\alpha}=\sup_{x \neq y}\vert f(y)-f(x) \vert / \vert y -x \vert^{\alpha}$.\\
Now, let $d \in \mathbb{N}$. For any $ \mathbb{R}^{d \times d}$-valued symmetric matrix $S$, we define $\lambda_S:= \sup\{\lambda_{S,1},\ldots, \lambda_{S,d},0 \}$, with $\lambda_{S,i}$ the $i$-th eigenvalue of $S$.\\
\subsection{Wasserstein convergence for the Milstein scheme}
In this section, we establish Wasserstein convergence results for the empirical measures (\ref{eq:def_weight_emp_meas}) built with the Milstein approximation scheme of a one-dimensional weakly mean reverting Brownian diffusion. The framework presented in Section \ref{section:convergence_inv_distrib_gnl} is well suited this scheme and we present the result that we obtain in this case.\\
 The Milstein scheme has not been investigated until now but the convergence results are similar to the Euler case that is why, even if the proofs are more technical, we simply state them. Moreover, looking at $\mathcal{E}(\widetilde{A},A,\DomA_0)$ (see (\ref{hyp:erreur_tems_cours_fonction_test_reg})), the approximation of $A$ seems to rely on the weak order of the scheme. As a consequence, even from a rate of convergence viewpoint, intuitively, it does not possible to achieve a better rate of convergence of $(\nu^{\eta}_n)_{n \in \mathbb{N}^{\ast}}$ with Milstein scheme than with Euler scheme. We will give the proof of this result in a further paper. \\
 
 We consider a one dimensional Brownian motion $(W_t)_{t \geqslant 0}$. We are interested in the strong solution - assumed to exist and to be unique - of the one dimensional stochastic equation
\begin{align}
\label{eq:EDS_1d}
X_t= x+\int_0^t b(X_{s})ds + \int_0^t  \sigma(X_{s})  dW_s
\end{align}
where $b,\sigma, \partial_{x}\sigma:\mathbb{R} \to \mathbb{R}$. Moreover, we assume that for every $x \in \mathbb{R},  \quad  \vert b(x) \vert^2 + \vert \sigma (x) \vert^2  + \vert \sigma \sigma' ( x )   \vert^2 \leqslant C (1+\vert x \vert^{2a})$ for some $a \in (0,1]$. Finally, for $p \geqslant 1$, we introduce the following $\mbox{L}_p$-mean reverting property:
\begin{align*}
\exists \alpha > 0, \beta \in \mathbb{R}, \qquad \forall x \in \mathbb{R},  \quad  2 x b(x)+   (4p-3 )2^{(2p-3)_+} \sigma^2(x)   \leqslant \beta - \alpha \vert x \vert^{2a}
\end{align*}
We now introduce the Milstein scheme for $(X_t)_{t \geqslant 0}$. Let $\rho \in [1,2]$ and $\epsilon_{\mathcal{I}}(\gamma)=\gamma^{\rho/2}$ and assume that (\ref{hyp:accroiss_sw_series_2}), $\mathcal{S}\mathcal{W}_{\mathcal{I}, \gamma,\eta}(\rho, \epsilon_{\mathcal{I}}) $ (see (\ref{hyp:step_weight_I})) and $\mathcal{S}\mathcal{W}_{\mathcal{II},\gamma,\eta}$ (see (\ref{hyp:step_weight_II})) hold. Let $(U_n)_{n}$ be a sequence of centered independent and identically distributed random variables with variance one and bounded moments of order $2p$. We define the Milstein scheme with decreasing steps $(\gamma_n)_{n \in \mathbb{N}^{\ast}}$, $(\overline{X}_{\Gamma_n} )_{n \in \mathbb{N}}$ of $(X_t)_{t \geqslant 0}$ (\ref{eq:EDS_1d}) by: $\overline{X}_{0} =x$, $\forall n \in \mathbb{N} $,
\begin{align*}
 \overline{X}_{\Gamma_{n+1}}  =& \overline{X}_{\Gamma_n} + \gamma_{n+1} b(\overline{X}_{\Gamma_{n}}) + \sqrt{\gamma_{n+1}} \sigma(\overline{X}_{\Gamma_{n}})U_{n+1}+  \gamma_{n+1}\sigma \sigma'( \overline{X}_{\Gamma_n} )  (\vert U_{n+1} \vert^2 -1),
\end{align*}
 Then $V:\mathbb{R} \to [1,+\infty)$, $x \mapsto 1+x^2$ is a Lyapunov function for this scheme. We consider $ (\nu^{\eta}_n(dx,\omega))_{n \in \mathbb{N}^{\ast}}$ defined as in (\ref{eq:def_weight_emp_meas}) with $(\overline{X}_{\Gamma_n} )_{n \in \mathbb{N}}$ defined above. Now,we specify the measurable functions $\psi,\phi:[1,+\infty) \to [1,+\infty)$ as $\psi(y)=y^{p}$ and $\phi(y)=y^a$. Moreover, let $s \geqslant 1$ such that $ap\rho /s \leqslant p+a-1$ and $p/s+a-1>0$. Then, it follows from Theorem \ref{th:identification_limit} that there exists an invariant distribution $\nu$ for $(X_t)_{t \geqslant 0}$. Moreover, $ (\nu^{\eta}_n(dx,\omega))_{n \in \mathbb{N}^{\ast}}$ $a.s.$ weakly converges toward $\mathcal{V}$, the set of invariant distributions of $(X_t)_{t \geqslant 0}$ and when it is unique $i.e.$ $\mathcal{V}=\{\nu\}$, we have
\begin{align*}
\lim_{n \to + \infty}\nu^{\eta}_n(f)= \nu(f),
\end{align*}
for every $\nu-a.s.$ continuous function $f:\mathbb{R} \to \mathbb{R}$ such that, for every $x \in \mathbb{R}$, $\vert f(x) \vert \leqslant C(1+ \vert x \vert^{\overline{p}})$, with $\overline{p} <p/s+a-1$. In other words $(\nu^{\eta}_n)_{n \in \mathbb{N}^{\ast}}$ converges towards $\nu$ (as $n$ tends to infinity) for the $\mbox{L}_{\overline{p}}$ Wasserstein distances.
\subsection{The Euler scheme for a Markov Switching diffusion}
In this part of the paper, we study invariant distributions for Markov switching Brownian diffusions. The framework presented in Section \ref{section:convergence_inv_distrib_gnl} is well suited to this case. Our results extend the convergence results obtained in \cite{Mei_Yin_2015} and inspired by \cite{Lamberton_Pages_2002}. More particularly, in \cite{Mei_Yin_2015}, the convergence of $(\nu^{\eta}_n)_{n \in\mathbb{N}^{\ast}}$ is established under a strongly mean reverting assumption that is $\phi=I_d$. In this paper, we do not restrict to that case and consider a weakly mean-reverting setting, namely $\phi(y)=y^a$, $a \in (0,1]$ for every $y \in [v_{\ast},\infty)$. As a first step, we consider polynomial test functions that is $\psi(y)=y^p$, $ p \geqslant 1$ for every $y \in [v_{\ast},\infty)$ like in \cite{Mei_Yin_2015} (where $p\geqslant 4$ is required). As a second step, still under a weakly mean-reverting setting (but where $\phi$ is not explicitly specified), we extend those results to functions $\psi$ with exponential growth which enables to obtain convergence of the empirical measures for much wider class of test functions.\\

Now, we present the Markov switching model, its decreasing step Euler approximation and the hypothesis necessary to obtain the convergence of $(\nu^{\eta}_n)_{n \in\mathbb{N}^{\ast}}$. We consider a $d$-dimensional Brownian motion $(W_t)_{t \geqslant 0}$ and $(\zeta_t)_{t \geqslant 0} $ a continuous time Markov chain taking values in the finite state space $\{1,\ldots,M_0 \}$, $M_0 \in \mathbb{N}^{\ast}$ with generator $\mathfrak{Q}=(q_{z,w})_{z,w\in \{1,\ldots,M_0 \} }$ and independent from $W$. We are interested in the strong solution - assumed to exist and to be unique - of the d-dimensional stochastic equation
\begin{align*}
X_t= x+\int_0^t b(X_{s}, \zeta_s)ds + \int_0^t & \sigma(X_{s},\zeta_s)  dW_s
\end{align*}
where for every $z \in \{1,\ldots,M_0\}$, $b(.,z): \mathbb{R}^d \to \mathbb{R}^d$ and $\sigma(.,z)\to \mathbb{R}^{d \times d}$ are locally bounded functions. We recall that $q_{z,w} \geqslant 0$ for $z \neq w$, $z,w \in \{1,\ldots,M_0\}$ and $\sum\limits_{w =1}^{M_0} q_{z,w}=0$ for every $z\in \{1,\ldots,M_0\}$. The infinitesimal generator of this process reads
\begin{align*}
Af(x,z) = & \langle b(x,z) , \nabla_x f(x,z) \rangle + \frac{1}{2}\sum_{i,j=1}^d (\sigma \sigma^{\ast})_{i,j}(x, z) \frac{\partial^2f}{\partial x_i \partial x_j}(x, z)  \\
&+ \sum_{w=1}^{M_0} q_{z,w} f(x,w),
\end{align*} 
for every $(x, z) \in E:= \mathbb{R}^d \times \{1,\ldots,M_0 \}$. Moreover, the domain $\DomA$ of $A$ contains $\DomA_0 =\{f \mbox{ defined on }E, \forall z \in \{1,\ldots,M_0\}, f(.,z)\in\mathcal{C}^2_K(\mathbb{R}^d)\}$. Notice that $\DomA_0 $ is dense in $\mathcal{C}_0(E)$. The reader may refer to \cite{Yin_Zhu_2010} for more details concerning Markov switching diffusion processes where properties such as recurrence, ergodicity and stability are established. We consider the Euler genuine scheme of this process for every $n \in \mathbb{N}$ and every $t \in [\Gamma_n, \Gamma_{n+1}]$, defined by
\begin{align}
\label{eq:MS_scheme}
\overline{X}_{t}  =& \overline{X}_{\Gamma_n} + (t-\Gamma_{n}) b(\overline{X}_{\Gamma_{n}} , \zeta_{\Gamma_n}) + \sigma(\overline{X}_{\Gamma_{n}} ,  \zeta_{\Gamma_n})(W_t-W_{\Gamma_n}) 
\end{align}
We will also denote $\Delta \overline{X}_{n+1}=\overline{X}_{\Gamma_{n+1}}-\overline{X}_{\Gamma_{n}}$ and
\begin{align}
\label{def:incr_MS}
\Delta \overline{X}^1_{n+1} =   \gamma_{n+1}b(\overline{X}_{\Gamma_{n}} , \zeta_{\Gamma_{n}})  , \quad  \Delta \overline{X}^2_{n+1} =  \sigma (\overline{X}_{\Gamma_{n}}, \zeta_{\Gamma_{n}}) (W_{\Gamma_{n+1}}-W_{\Gamma_n} ),
\end{align}
and $\overline{X}_{\Gamma_{n+1}}^i=\overline{X}_{\Gamma_n}+ \sum_{j=1}^i \Delta \overline{X}^i_{n+1}$. In the sequel we will use the notation $U_{n+1}=\gamma_{n+1}^{-1/2}(W_{\Gamma_{n+1}}-W_{\Gamma_n}) $.
Finally, we consider a Lyapunov function $V: \mathbb{R}^d \times \{1,\ldots,M_0 \} \to [v_{\ast}, \infty)$, $v_{\ast}>0$, which satisfies $\mbox{L}_V$ (see (\ref{hyp:Lyapunov})) with $E=\mathbb{R}^d \times \{1,\ldots,M_0\}$, and
\begin{align}
\label{hyp:Lyapunov_control_MS}
\vert \nabla_x V \vert^2 \leqslant C_V V, \qquad \sup_{(x,z) \in E} \vert D^2_x V(x,z) \vert < + \infty.
\end{align}
Its mean-reverting properties will be defined further depending on the set of `test functions' $f$. We also define
\begin{align}
\label{def:lambda_psi_MS}
\forall x \in \mathbb{R}^d , z \in \{1,\ldots,M_0 \}, \quad \lambda_{\psi}(x,z):=\lambda_{D_x^2V(x,z)+ 2 \nabla_x V(x,z)^{\otimes 2} \psi''\circ V(x,z) \psi'\circ V(x,z)^{-1}}  .
\end{align}
When $\psi(y)=\psi_p(y)=y^{p}$, $p>0$, we will also use the notation $\lambda_p$ instead of $\lambda_{\psi}$.
We suppose that there exists $C >0$ such that $b$ and $\sigma$ satisfy
\begin{align}
\label{hyp:controle_coefficients_MS}
\mathfrak{B}(\phi) \quad \equiv \qquad \forall x \in \mathbb{R}^d, &\forall z \in \{1,\ldots,M_0\}, \nonumber \\
& \vert b(x,z) \vert^2 + \Tr[\sigma \sigma^{\ast}(x,z) ] \leqslant C \phi \circ V (x,z)
\end{align}
\paragraph{Test functions with polynomial growth. \\}
Having in mind Wasserstein convergence, we introduce a weaker assumption on the sequence $(U_n)_{n \in \mathbb{N}^{\ast}} $ than Gaussian distribution . Let $q \in \mathbb{N}^{\ast}$, $p \geqslant 0$. We suppose that $(U_n)_{n \in \mathbb{N}^{\ast}} $ is a sequence of independent and identically distributed random variables such that
\begin{align}
\label{hyp:matching_normal_moment_ordre_q_va_schema_MS}
M_{\mathcal{N},q}(U) \quad \equiv \quad \forall n \in \mathbb{N}^{\ast} ,\forall \tilde{q} \in \{1, \ldots, q\} , \quad \mathbb{E}[(U_n)^{\otimes \tilde{q}}]=\mathbb{E}[(\mathcal{N}(0,I_d))^{\otimes \tilde{q}}]
\end{align}
\begin{align}
\label{hyp:moment_ordre_p_va_schema_MS}
M_p(U) \quad \equiv \qquad \sup_{n \in \mathbb{N}^{\ast}} \mathbb{E}[\vert U_n \vert^{2p} ] < + \infty
\end{align}
We assume that
\begin{align}
\label{hyp:Lyapunov_control_MS_unif} 
\exists c_{V} \geqslant 1, \forall x \in \mathbb{R}^d, \quad \sup_{z \in \{1,\ldots,M_0 \}} V(x,z)\leqslant c_{V} \inf_{z \in \{1,\ldots,M_0 \}} V(x,z).
\end{align}
Let $\alpha>0$ and $\beta \in \mathbb{R}$. We introduce the mean-reverting property of the scheme for the Lyapunov function $V$. We assume that $\liminf\limits_{y \to +\infty} \phi(y)>\beta/\alpha$ and that there exists $\epsilon_0>0$, such that we have
\begin{align}
\label{hyp:recursive_control_param_MS}
\mathcal{R}_p(\alpha,\beta,\phi,V) \quad \equiv &\quad \forall x \in \mathbb{R}^d,  \forall z \in \{1,\ldots,M_0\}, \nonumber \\
&  \langle \nabla V(x,z), b(x,z) \rangle+ \frac{1}{2}   \chi_{p}(x,z) \leqslant \beta - \alpha \phi \circ V (x,z ),
\end{align}
with
\begin{align}
\label{hyp:recursive_control_param_terme_ordre_sup_MS}
\chi_{p}(x,z) = & \Vert \lambda_{p} \Vert_{\infty} 2^{(2p-3)_+} \mbox{Tr}[\sigma \sigma^{\ast}(x,z)]  \nonumber \\
& +V^{1-p}(x,z) \sum_{w=1}^{M_0} (q_{z,w}+\epsilon_0 )V^p(x,w)
\end{align}
\begin{mytheo}
\label{th:cv_was_MS}
Let $p \geqslant 1,a \in (0,1]$, $s \geqslant 1, \rho \in [1,2]$, $\psi_p(y)=y^p$, $\phi(y)=y^a$ and $\epsilon_{\mathcal{I}}(\gamma)=\gamma^{\rho/2}$. Let $\alpha>0$ and $\beta \in \mathbb{R}$. \\

Assume that $(U_n)_{n \in \mathbb{N}^{\ast}}$ satisfies $M_{\mathcal{N},2}(U)$ (see (\ref{hyp:matching_normal_moment_ordre_q_va_schema_MS})) and $M_{p}(U)$ (see (\ref{hyp:moment_ordre_p_va_schema_MS})). \\
Also assume that (\ref{hyp:Lyapunov_control_MS}), $\mathfrak{B}(\phi)$ (see (\ref{hyp:controle_coefficients_MS})), $\mathcal{R}_p(\alpha,\beta,\phi,V)$ (see (\ref{hyp:recursive_control_param_MS})), $\mbox{L}_{V}$ (see (\ref{hyp:Lyapunov})), $\mathcal{S}\mathcal{W}_{\mathcal{I}, \gamma,\eta}(\rho, \epsilon_{\mathcal{I}})$ (see (\ref{hyp:step_weight_I})), $\mathcal{S}\mathcal{W}_{\mathcal{II},\gamma,\eta} $ (see (\ref{hyp:step_weight_II})), (\ref{hyp:accroiss_sw_series_2}) and (\ref{hyp:Lyapunov_control_MS_unif}) hold and that $p\rho/s \leqslant p+a-1$. \\

Then, if $p/s+a-1>0$, $(\nu_n^{\eta})_{n \in \mathbb{N}^{\ast}}$ (built with $(\overline{X}_t)_{t \geqslant 0}$ defined in (\ref{eq:MS_scheme})) is $\mathbb{P}-a.s.$ tight and
\begin{align}
 \label{eq:tightness_MS}
\mathbb{P} \mbox{-a.s.} \quad \sup_{n \in \mathbb{N}^{\ast}} \nu_n^{\eta}( V^{p/s+a-1} ) < + \infty .
\end{align}
Assume also that for every $z \in \{1,\ldots,M_0\}$, $b(.,z)$ and $\sigma(.,z)$ have sublinear growth and $\Tr[\sigma \sigma^{\ast}(x,z)]\leqslant C  V^{p/s+a-1}(x,z)$. Then every weak limiting distribution $\nu$ of $(\nu^{\eta}_n)_{n \in \mathbb{N}^{\ast}}$ is an invariant distribution of $(X_t)_{t \geqslant 0}$ and when $\nu$ is unique, we have
\begin{align}
 \label{eq:cv_was_MS}
\mathbb{P} \mbox{-a.s.} \quad \forall f \in \mathcal{C}_{\tilde{V}_{\psi_p,\phi,s}}(E), \quad \lim\limits_{n \to + \infty} \nu_n^{\eta}(f)=\nu(f),
\end{align}
 with $\mathcal{C}_{\tilde{V}_{\psi_p,\phi,s}}(E)$ defined in (\ref{def:espace_test_function_cv}).
\end{mytheo}
 \paragraph{Test functions with exponential growth. \\} 
We modify the hypothesis concerning the Lyapunov function $V$ in the following way. First, we assume that 
 \begin{align}
\label{hyp:Lyapunov_control_MS_unif_expo} 
\forall z \in \{1,\ldots,M_0\}, \forall x \in \mathbb{R}^d  \quad V(x,z) =V(x,1),
\end{align}
and we will use the notation $V(x):=V(x,1)$. We assume that
\begin{align}
\label{hyp:dom_recurs_MS}
\forall x \in \mathbb{R}^d , &\forall z \in \{1,\ldots,M_0 \},  \nonumber \\
&  \Tr[\sigma \sigma^{\ast} (x,z) ] \vert b(x) \vert \big( \vert\nabla{V}(x)\vert+ \vert b (x,z) \vert \big) \leqslant C V^{1-p}(x) \phi \circ V (x)
\end{align}
Now let $p \leqslant 1$ and let $\alpha>0$ and $\beta \in \mathbb{R}$. We assume that $\liminf\limits_{y \to + \infty} \phi(y)>\beta_+/\alpha$, $\beta_+=0 \vee \beta$, and 
\begin{align}
\label{hyp:recursive_control_param_MS_expo}
\mathcal{R}_{p, \lambda}(\alpha,\beta,\phi,V)& \quad \equiv  \quad \forall x \in \mathbb{R}^d,  \forall z \in \{1,\ldots,M_0\},  \nonumber \\
& \langle \nabla V(x), b(x,z) + \kappa_p(x,z) \rangle+   \frac{1}{2}\chi_{p}(x,z) \leqslant \beta - \alpha \phi \circ V (x),
\end{align}
with
\begin{align*}
\kappa_p(x,z)= \lambda p \frac{ V^{p-1}(x)}{\phi \circ V(x)}\sigma \sigma^{\ast} (x,z) \nabla V (x) 
\end{align*}
and
\begin{align*}
\chi_{p}(x,z) =  - \frac{V^{1-p} (x)}{ \phi \circ V(x) C_{\sigma}(x,z) }  \ln (\det(\Sigma(x,z))) 
\end{align*}
with $\Sigma: \mathbb{R}^d \times  \{1,\ldots,M_0\}, \to \mathcal{S}^d_{+,\ast}$, $\mathcal{S}^d_{+,\ast}$ being the set of a positive definite matrix, defined by $(x,z) \mapsto \Sigma(x,z):=I_d-\Vert D^2 V  \Vert_{\infty}  C_{\sigma}(x,z)V^{p-1}(x)  \sigma^{\ast} \sigma(x,z)$, where $C_{\sigma}: \mathbb{R}^d \times \{1,\ldots,M_0\}\to \mathbb{R}_+^{\ast}$ satisfies
$\inf_{x \in \mathbb{R}^d} \inf_{z \in \{1,\ldots,M_0\}}C_{\sigma}(x,z)>0$.
%
%
%
%
\begin{mytheo}
\label{th:cv_exp_MS}
 Let $p\in [0,1], \lambda \geqslant 0$, $s \geqslant 1$, $\rho \in [1,2]$, let $\phi:[v_{\ast},\infty )\to \mathbb{R}_+$ be a continuous function such that $C_{\phi}:= \sup_{y \in [v_{\ast},+ \infty )}\phi(y)/y< +\infty$ and $\liminf\limits_{y \to +\infty} \phi(y)=+\infty$, let $\psi(y)=\exp ( \lambda y^p )$, $y\in \mathbb{R}_+$ and let $\epsilon_{\mathcal{I}}(\gamma)=\gamma^{\rho/2}$ and $\tilde{\epsilon}_{\mathcal{I}}(\gamma)=\gamma^{\rho(p \wedge 1/2)}$. Let $\alpha>0$ and $\beta \in \mathbb{R}$. \\
 
 Assume that $\rho<s$, (\ref{hyp:Lyapunov_control_MS_unif_expo}), (\ref{hyp:Lyapunov_control_MS}), $\mathfrak{B}(\phi)$ (see (\ref{hyp:controle_coefficients_MS})), $\mathcal{R}_{p, \lambda}(\alpha,\beta,\phi,V)$ (see (\ref{hyp:recursive_control_param_MS_expo})) and $\mbox{L}_{V}$ (see (\ref{hyp:Lyapunov}))  hold.
Also suppose that $\mathcal{S}\mathcal{W}_{\mathcal{I}, \gamma,\eta}(\rho, \epsilon_{\mathcal{I}})$, $\mathcal{S}\mathcal{W}_{\mathcal{I}, \gamma,\eta}(\rho, \tilde{\epsilon}_{\mathcal{I}})$ (see (\ref{hyp:step_weight_I})), $\mathcal{S}\mathcal{W}_{\mathcal{II},\gamma,\eta} $ (see (\ref{hyp:step_weight_II})), (\ref{hyp:accroiss_sw_series_2}) and (\ref{hyp:dom_recurs_MS}) hold.\\
 
  Then $(\nu_n^{\eta})_{n \in \mathbb{N}^{\ast}}$ (built with $(\overline{X}_t)_{t \geqslant 0}$ defined in (\ref{eq:MS_scheme})) is $\mathbb{P}-a.s.$ tight and
\begin{align}
 \label{eq:tightness_MS_expo}
\mathbb{P} \mbox{-a.s.} \quad \sup_{n \in \mathbb{N}^{\ast}} \nu_n^{\eta} \Big( \frac{\phi \circ V }{V} \exp \big(\lambda/s V^{p}) \Big) < + \infty .
\end{align}
Assume also that for every $z \in \{1,\ldots,M_0\}$, $b(.,z)$ and $\sigma(.,z)$ have sub-linear growth. Then, every weak limiting distribution $\nu$ of $(\nu_n^{\eta})_{n \in \mathbb{N}^{\ast}}$ is an invariant distribution of $(X_t)_{t \geqslant 0}$ and if $\nu$ is unique, 
\begin{align}
\label{eq:cv_expo_MS}
\mathbb{P} \mbox{-a.s.} \quad \forall f \in \mathcal{C}_{\tilde{V}_{\psi,\phi,s}}(E), \quad \lim\limits_{n \to + \infty} \nu_n^{\eta}(f)=\nu(f),
\end{align}
 with $\mathcal{C}_{\tilde{V}_{\psi,\phi,s}}(E)$ defined in (\ref{def:espace_test_function_cv}).
\end{mytheo}
\subsubsection{Proof of the recursive mean-reverting control}
\paragraph{Test functions with polynomial growth}
\begin{myprop}
\label{prop:recursive_control_MS}
Let $v_{\ast}>0$, and let $\phi:[v_{\ast},\infty )\to \mathbb{R}_+^{\ast}$ be a continuous function such that $C_{\phi}:= \sup_{y \in [v_{\ast},\infty )}\phi(y)/y< +\infty$.  Now let $p \geqslant 1$ and define $\psi_p(y)=y^p$, $y\in \mathbb{R}_+$.\\

Assume that the sequence $(U_n)_{n \in \mathbb{N}^{\ast}}$ satisfies $M_{\mathcal{N},2}(U)$ (see (\ref{hyp:matching_normal_moment_ordre_q_va_schema_MS})) and $M_{p}(U)$ (see (\ref{hyp:moment_ordre_p_va_schema_MS})).\\
Also suppose that (\ref{hyp:Lyapunov_control_MS}), (\ref{hyp:Lyapunov_control_MS_unif}), $\mathfrak{B}(\phi)$ (see (\ref{hyp:controle_coefficients_MS})), $\mathcal{R}_p(\alpha,\beta,\phi,V)$ (see (\ref{hyp:recursive_control_param_MS})) for some $\alpha>0$ and $\beta \in \mathbb{R}$, are satisfied. \\

Then, for every $\tilde{\alpha}\in (0, \alpha)$, there exists $n_0 \in \mathbb{N}^{\ast}$, such that 
\begin{align}
\label{eq:recursive_control_MS_fonction_pol}
 \forall n   \geqslant n_0, \forall x \in \mathbb{R}^d, &\forall z \in \{1,\ldots,M_0 \}, \nonumber \\
 &\widetilde{A}_{\gamma_n} \psi_p \circ V(x,z)\leqslant  \frac{\psi_p \circ V(x,z)}{V(x,z) }p\big(\beta - \tilde{\alpha} \phi\circ V(x,z)\big). 
\end{align}
Then $\mathcal{RC}_{Q,V}(\psi_p,\phi,p\tilde{\alpha},p\beta)$ (see (\ref{hyp:incr_sg_Lyapunov})) holds for every $\tilde{\alpha}\in (0, \alpha)$ such that $\liminf\limits_{y \to +\infty} \phi(y) > \beta / \tilde{\alpha}$. Moreover, when $\phi=Id$, we have
\begin{align}
\label{eq:mom_pol_MS}
\sup_{n \in \mathbb{N}} \mathbb{E}[\psi_p \circ V (\overline{X}_{\Gamma_n},\zeta_{\Gamma_n})] < + \infty.
\end{align}
\end{myprop}
\begin{proof}  
First we write
\begin{align}
\label{eq:decomp_preuve_RC_pol_MS}
V^p( \overline{X}_{\Gamma_{n+1}}, \zeta_{\Gamma_{n+1}})-V^p(\overline{X}_{\Gamma_{n}}, \zeta_{\Gamma_{n}})=& V^p(\overline{X}_{\Gamma_{n+1}}, \zeta_{\Gamma_{n}})-V^p(\overline{X}_{\Gamma_{n}}, \zeta_{\Gamma_{n}}) \\
& +V^p(\overline{X}_{\Gamma_{n+1}}, \zeta_{\Gamma_{n+1}})- V^p(\overline{X}_{\Gamma_{n+1}}, \zeta_{\Gamma_{n}}) \nonumber
\end{align}
We study the first term of the $r.h.s.$ of the above equality. From the second order Taylor expansion and the definition of $\lambda_{\psi_p}=\lambda_p$ (see (\ref{def:lambda_psi_MS})), we derive
\begin{align}
\label{eq:taylor_preuve_RC_pol_MS}
\psi_{p} \circ V & (\overline{X}_{\Gamma_{n+1}}, \zeta_{\Gamma_{n}})  \nonumber \\
=& \psi_{p} \circ V(\overline{X}_{\Gamma_n},\zeta_{\Gamma_{n}})+ \langle \overline{X}_{\Gamma_{n+1}}-\overline{X}_{\Gamma_n}, \nabla_x V(\overline{X}_{\Gamma_n},\zeta_{\Gamma_{n}}) \rangle \psi_{p}'\circ V(\overline{X}_{\Gamma_n},\zeta_{\Gamma_{n}}) \nonumber \\
&+ \frac{1}{2}  D_x^2 V(\Upsilon_{n+1},\zeta_{\Gamma_{n}} ) \psi_{p}' \circ  V(\Upsilon_{n+1} ,\zeta_{\Gamma_{n}})
 ( \overline{X}_{\Gamma_{n+1}}-\overline{X}_{\Gamma_n} )^{\otimes 2} \nonumber \\
 &+ \frac{1}{2} \nabla_x V (\Upsilon_{n+1},\zeta_{\Gamma_{n}} )^{\otimes 2} \psi_{p}'' \circ V(\Upsilon_{n+1},\zeta_{\Gamma_{n}} ) 
 ( \overline{X}_{\Gamma_{n+1}}-\overline{X}_{\Gamma_n} )^{\otimes 2} \nonumber \\
 \leqslant & \psi_{p} \circ V(\overline{X}_{\Gamma_n},\zeta_{\Gamma_{n}})+ \langle \overline{X}_{\Gamma_{n+1}}-\overline{X}_{\Gamma_n}, \nabla_x V(\overline{X}_{\Gamma_n},\zeta_{\Gamma_{n}}) \rangle \psi_{p}'\circ V(\overline{X}_{\Gamma_n},\zeta_{\Gamma_{n}})  \nonumber \\
&+ \frac{1}{2} \lambda_{p} (\Upsilon_{n+1},\zeta_{\Gamma_{n}} )  \psi_{p}'\circ V(\Upsilon_{n+1} ,\zeta_{\Gamma_{n}}) \vert  \overline{X}_{\Gamma_{n+1}}-\overline{X}_{\Gamma_n}  \vert^{2}. 
\end{align} 
with $\Upsilon_{n+1}  \in (\overline{X}_{\Gamma_n}, \overline{X}_{\Gamma_{n+1}})$. First, from (\ref{hyp:Lyapunov_control_MS}), we have $\sup_{z \in \{1,\ldots,M_0 \}}\sup_{x \in \mathbb{R}^d} \lambda_{p} (x,z)  < + \infty$.
 Now, since  $(U_n)_{n \in \mathbb{N}^{\ast}} $ is $i.i.d.$ and satisfies $M_{\mathcal{N},1}(U)$ (see (\ref{hyp:matching_normal_moment_ordre_q_va_schema_MS})), we compute
 \begin{align*}
& \mathbb{E}[\overline{X}_{\Gamma_{n+1}}-\overline{X}_{\Gamma_n} \vert \overline{X}_{\Gamma_n},\zeta_{\Gamma_{n}} ]= \gamma_{n+1}  b(\overline{X}_{\Gamma_n} ,\zeta_{\Gamma_{n}})\\
& \mathbb{E}[ \vert  \overline{X}_{\Gamma_{n+1}}-\overline{X}_{\Gamma_n} \vert^{2} \vert \overline{X}_{\Gamma_n},\zeta_{\Gamma_{n}} ]= \gamma_{n+1}\mbox{Tr}[\sigma \sigma^{\ast}(\overline{X}_{\Gamma_n},\zeta_{\Gamma_{n}} )]  + \gamma_{n+1}^2 \vert b(\overline{X}_{\Gamma_n},\zeta_{\Gamma_{n}} ) \vert^{2} .
\end{align*}
We focus on the study of the last term of the $r.h.s$ of (\ref{eq:taylor_preuve_RC_pol_MS}), also called the `remainder'.
\paragraph{Case $p=1$. }Assume first that $p=1$. Using $\mathfrak{B}(\phi)$ (see (\ref{hyp:controle_coefficients_MS})), for every $\tilde{\alpha} \in (0, \alpha)$, there exists $n_0(\tilde{\alpha})$ such that, for every $n\geqslant n_0(\tilde{\alpha})$,
\begin{align}
\label{eq:recursive_control_remaider_Id_MS}
\frac{1}{2} \Vert \lambda_{1} \Vert_{\infty}    \gamma_{n+1}^2 \vert b(\overline{X}_{\Gamma_n} ,\zeta_{\Gamma_{n}} ) \vert^2 \leqslant \gamma_{n+1}(\alpha- \tilde{\alpha})\phi \circ V(\overline{X}_{\Gamma_n},,\zeta_{\Gamma_{n}} ).
\end{align}
From assumption $\mathcal{R}_p(\alpha,\beta,\phi,V)$ (see (\ref{hyp:recursive_control_param_MS}) and (\ref{hyp:recursive_control_param_terme_ordre_sup_MS})), we gather all the terms of (\ref{eq:taylor_preuve_RC_pol_MS}) together and we conclude that
\begin{align*}
\gamma_{n+1}^{-1} \mathbb{E}[V (\overline{X}_{\Gamma_{n+1}},\zeta_{\Gamma_{n}})-  V(\overline{X}_{\Gamma_{n}},\zeta_{\Gamma_{n}}) \vert \overline{X}_{\Gamma_{n}},\zeta_{\Gamma_{n}}] +&\sum_{z=1}^{M_0}  (q_{\zeta_{\Gamma_{n}},z}+\epsilon_0 ) V( \overline{X}_{\Gamma_{n}} ,z)  \\
&\quad \leqslant \beta - \tilde{\alpha} \phi \circ V( \overline{X}_{\Gamma_{n}} ,\zeta_{\Gamma_{n}} ).
\end{align*}
\paragraph{Case $p>1$. }Assume now that $p>1$ so that $\psi_{p}'(y)=py^{p-1}$. Since $\vert \nabla V \vert^2 \leqslant C_V V$ (see (\ref{hyp:Lyapunov_control_MS})), then $\sqrt{V}$ is Lipschitz. Now, we use the following inequality: Let $l \in \mathbb{N}^{\ast}$. We have
\begin{align}
\label{eq:puisance_somme_n_terme}
\forall \alpha >0 ,\forall u_i \in \mathbb{R}^d, i=1,\ldots,l, \qquad  \big \vert  \sum_{i=1}^l u_i \big \vert^{\alpha} \leqslant l^{(\alpha-1)_+}  \sum_{i=1}^l \vert u_i  \vert^{\alpha} .
\end{align}
It follows that
\begin{align*}
V^{p-1} (\Upsilon_{n+1} ,\zeta_{\Gamma_{n}} )  \leqslant & \big( \sqrt{V}(\overline{X}_{\Gamma_n},\zeta_{\Gamma_{n}})+[\sqrt{V}]_1 \vert \overline{X}_{\Gamma_{n+1}}-\overline{X}_{\Gamma_n} \vert \big)^{2p-2} \\
\leqslant & 2^{(2p-3)_+} (V^{p-1}(\overline{X}_{\Gamma_n},\zeta_{\Gamma_{n}} ) + [\sqrt{V}]_1^{2p-2} \vert \overline{X}_{\Gamma_{n+1}}-\overline{X}_{\Gamma_n} \vert^{2p-2})
\end{align*}
To study the `remainder' of (\ref{eq:taylor_preuve_RC_pol_MS}), we multiply the above inequality by $\vert \overline{X}_{\Gamma_{n+1}}-\overline{X}_{\Gamma_n} \vert^{2}$. First, we study the second term which appears in the $r.h.s.$ and using $\mathfrak{B}(\phi)$ (see (\ref{hyp:controle_coefficients_MS})), for any $p \geqslant 1$,
\begin{align*}
\vert  \overline{X}_{\Gamma_{n+1}}- \overline{X}_{\Gamma_n} \vert^{2p} \leqslant C \gamma_{n+1}^p \phi \circ V(\overline{X}_{\Gamma_{n}},\zeta_{\Gamma_{n}} )^p(1+\vert U_{n+1} \vert^{2p} ).
\end{align*}
 Let $\hat{\alpha} \in (0,\alpha)$. Therefore, we deduce from $M_{p}(U)$ (see (\ref{hyp:moment_ordre_p_va_schema_MS})) that there exists $n_0(\hat{\alpha}) \in \mathbb{N}$ such that for any $n \geqslant n_0(\hat{\alpha})$, we have
\begin{align*}
 \mathbb{E}[ \vert \overline{X}_{\Gamma_{n+1}}- \overline{X}_{\Gamma_n} \vert^{2p} & \vert \overline{X}_{\Gamma_n}, \zeta_{\Gamma_{n}} ] \\
& \leqslant \gamma_{n+1} \phi \circ V (\overline{X}_{\Gamma_{n}} , \zeta_{\Gamma_{n}})^{p} \frac{\alpha- \hat{\alpha} }{C_{\phi}^{p-1} \Vert \lambda_{p} \Vert_{\infty} 2^{(2p-3)_+} [\sqrt{V}]_1^{2p-2} } .
\end{align*}
To treat the other term of the `remainder' of (\ref{eq:taylor_preuve_RC_pol_MS}), we proceed as in (\ref{eq:recursive_control_remaider_Id_MS}) with $\Vert \lambda_{1} \Vert_{\infty}$ replaced by $\Vert \lambda_{p} \Vert_{\infty} 2^{2p-3} [\sqrt{V}]_1^{2p-2}  $, $\alpha$ replaced by $ \hat{\alpha}$ and $\tilde{\alpha} \in (0, \hat{\alpha})$. We gather all the terms of (\ref{eq:taylor_preuve_RC_pol_MS}) together and using $\mathcal{R}_p(\alpha,\beta,\phi,V)$ (see (\ref{hyp:recursive_control_param_MS}) and (\ref{hyp:recursive_control_param_terme_ordre_sup_MS})), for every $n \geqslant n_0(\tilde{\alpha}) \vee n_0(\hat{\alpha}) $, we obtain
\begin{align*}
\mathbb{E}[V^p (\overline{X}_{\Gamma_{n+1}},\zeta_{\Gamma_{n}}) - &  V^p(\overline{X}_{\Gamma_{n}},\zeta_{\Gamma_{n}}) \vert  \overline{X}_{\Gamma_{n}},\zeta_{\Gamma_{n}}] \\
& +V^{1-p}(\overline{X}_{\Gamma_{n}},\zeta_{\Gamma_{n}}) \sum_{z=1}^{M_0}( q_{\zeta_{\Gamma_{n}},z}+\epsilon )  V^p(\overline{X}_{\Gamma_{n}},z) \\
& \qquad \leqslant   \gamma_{n+1}p V^{p-1}(\overline{X}_{\Gamma_{n}}, \zeta_{\Gamma_{n}})( \beta - \alpha \phi \circ V (\overline{X}_{\Gamma_{n}}, \zeta_{\Gamma_{n}})  ) \\
 &\qquad \qquad+\gamma_{n+1}p V^{p-1}(\overline{X}_{\Gamma_{n}}, \zeta_{\Gamma_{n}}) \Big(  \phi \circ V (\overline{X}_{\Gamma_{n}},\zeta_{\Gamma_{n}}) (\hat{\alpha} -\tilde{\alpha}  ) \\
 & \qquad \qquad  \qquad+ (\alpha-\hat{\alpha})  \frac{ V^{1-p}(\overline{X}_{\Gamma_{n}},\zeta_{\Gamma_{n}})  \phi \circ V (\overline{X}_{\Gamma_{n}}, \zeta_{\Gamma_{n}})^{p} }{C_{\phi}^{p-1}}    \Big) \\
& \qquad\leqslant \gamma_{n+1} V^{p-1}(\overline{X}_{\Gamma_{n}},\zeta_{\Gamma_{n}}) \big( \beta p- \tilde{\alpha} p  \phi \circ V (\overline{X}_{\Gamma_{n}} ,\zeta_{\Gamma_{n}} ) \big ). \\
\end{align*}
 Now, we focus on the second term of the $r.h.s.$ of (\ref{eq:decomp_preuve_RC_pol_MS}). First, since $\zeta$ and $W$ are independent, it follows, with notations (\ref{def:incr_MS}), that
\begin{align*}
\mathbb{E}[V^p(\overline{X}_{\Gamma_{n+1}}, \zeta_{\Gamma_{n+1}})- & V^p(\overline{X}_{\Gamma_{n+1}}, \zeta_{\Gamma_{n}})  \vert  \overline{X}_{\Gamma_{n}}, \zeta_{\Gamma_{n}}, \Delta \overline{X}_{n+1}   ] \\
& = \gamma_{n+1} \sum_{z=1}^{M_0}( q_{\zeta_{\Gamma_{n}}  ,z}+  \underset{n \to + \infty}{o}(\gamma_{n+1} ))V^p(\overline{X}_{\Gamma_{n+1}},z) .
\end{align*}
Now, using the same reasoning as for the first term of the $r.h.s.$ of (\ref{eq:decomp_preuve_RC_pol_MS}) and (\ref{hyp:Lyapunov_control_MS_unif}), since $p \geqslant 1$, we derive, for every $z \in \{1,\ldots,M_0\}$,
\begin{align*}
\vert \mathbb{E} [V^p(\overline{X}_{\Gamma_{n+1}},z)  - & V^p(\overline{X}_{\Gamma_{n}},z) \vert  \overline{X}_{\Gamma_{n}}, \zeta_{\Gamma_{n}} ] \vert \\
 \leqslant &  C(\gamma_{n+1}^{1/2} V^{p-1}(\overline{X}_{\Gamma_{n}} ,z) \phi \circ V( \overline{X}_{\Gamma_{n}} ,\zeta_{\Gamma_{n}})  +\gamma_{n+1}^{p} \phi \circ V( \overline{X}_{\Gamma_{n}} ,\zeta_{\Gamma_{n}}) ^p \\
& +    \gamma_{n+1} V^{p-1/2}(\overline{X}_{\Gamma_{n}} ,z) \sqrt{\phi \circ V( \overline{X}_{\Gamma_{n}} ,\zeta_{\Gamma_{n}}) }  )\\
\leqslant & C\gamma_{n+1}^{1/2} V^{p}(\overline{X}_{\Gamma_{n}} ,\zeta_{\Gamma_{n}}) 
\end{align*}
where $C>0$ is a constant which may change from line to line. We deduce that there exists $\varepsilon: \mathbb{R}_+ \to \mathbb{R}_+$ satisfying $\lim\limits_{\gamma \to 0} \varepsilon(\gamma)=0$, such that we have
\begin{align*}
\mathbb{E}[  V^p(\overline{X}_{\Gamma_{n+1}}, \zeta_{\Gamma_{n+1}})-&  V^p(\overline{X}_{\Gamma_{n+1}}, \zeta_{\Gamma_{n}})  \vert  \overline{X}_{\Gamma_{n}}, \zeta_{\Gamma_{n}}  ]  \\
=& \gamma_{n+1} \sum_{z=1}^{M_0} \big( q_{\zeta_{\Gamma_{n}}  ,z}+  o(\gamma_{n+1} ) \big) \mathbb{E} [V^p(\overline{X}_{\Gamma_{n+1}},z)  \vert \overline{X}_{\Gamma_{n}}, \zeta_{\Gamma_{n}} ]  \\
\leqslant &     \gamma_{n+1} \sum_{z=1}^{M_0} ( q_{\zeta_{\Gamma_{n}}  ,z}+ \varepsilon(\gamma_{n+1}) ) V^p(\overline{X}_{\Gamma_{n}},z).
\end{align*}
%
%
%
This yields (\ref{eq:recursive_control_MS_fonction_pol}) as a direct consequence of $\mathcal{R}_p(\alpha,\beta,\phi,V)$ (see (\ref{hyp:recursive_control_param_MS}) and (\ref{hyp:recursive_control_param_terme_ordre_sup_MS})). The proof of (\ref{eq:mom_pol_MS}) is an immediate application of Lemma \ref{lemme:mom_psi_V} as soon as we notice that the increments of the Euler scheme (for Markov Switching diffusions) have finite polynomial moments which implies (\ref{eq:mom_psi_V}).
%
%
%
%
\end{proof}
\paragraph{Test functions with exponential growth \\}
In this section we do not relax the assumption on the Gaussian structure of the increment as we do in the polynomial case with hypothesis  (\ref{hyp:matching_normal_moment_ordre_q_va_schema_MS}) and (\ref{hyp:moment_ordre_p_va_schema_MS}). In particular, it leads the following result:
\begin{lemme}
\label{lemme:transformee_laplace_prod_gauss}
Let $\Lambda \in \mathbb{R}^{d \times d}$ and $U \sim \mathcal{N}(0,I_d)$. We define $\Sigma \in \mathbb{R}^{d \times d}$ by $\Sigma=I_d-2 \Lambda^{\ast} \Lambda$. Assume that $\Sigma \in \mathcal{S}_{+,\ast}^d$. Then,  for every $h \in (0,1)$,
\begin{align}
\label{eq:transformee_laplace_jointe}
\forall v \in \mathbb{R}^d , \qquad \mathbb{E}\Big[\exp \Big( \sqrt{h}  \langle v , U \rangle    + h \vert \Lambda U \vert^2 \Big) \Big]  \leqslant   \exp\Big( \frac{h}{2(1-h) } \vert v \vert^2    \Big) \det( \Sigma)^{-h/2}.
\end{align}
%
%
\end{lemme}
\begin{proof}
A direct computation yields
\begin{align*}
\mathbb{E}[\exp ( \vert \Lambda U\vert^2  )] = & \int_{\mathbb{R}^d} (2 \pi)^{-d/2}\exp \Big( -\frac{1}{2} \langle -2\Lambda^{\ast} \Lambda u+u,u \rangle \Big) du = \det( \Sigma)^{-1/2}.
\end{align*}
Now, (\ref{eq:transformee_laplace_jointe}) follows from the H\"older inequality since
\begin{align*}
 \mathbb{E}[\exp (  \sqrt{h}  \langle v , U \rangle    + h \vert \Lambda U\vert^2  )]   \leqslant  & \mathbb{E}\Big[\exp \Big( \frac{\sqrt{h}}{1-h}  \langle v , U \rangle \Big)\Big] ^{1-h}   \mathbb{E}[\exp (  \vert \Lambda U\vert^2 )]^h \\
  =& \exp\Big( \frac{h}{2(1-h) } \vert v \vert^2    \Big) \det( \Sigma)^{-h/2}.
\end{align*}
\end{proof}
Using those results, we deduce the recursive control for exponential test functions. 
\begin{myprop}
\label{prop:recursive_control_MS_exp}
Let $v_{\ast}>0$, and let $\phi:[v_{\ast},\infty )\to \mathbb{R}_+$ be a continuous function such that $C_{\phi}:= \sup_{y \in [v_{\ast},\infty )}\phi(y)/y< +\infty$.  Now let $p \in [0,1]$, $\lambda \geqslant 0$ and define $\psi(y)=\exp ( \lambda y^p )$, $y\in \mathbb{R}_+$. \\

Suppose that (\ref{hyp:Lyapunov_control_MS}), (\ref{hyp:Lyapunov_control_MS_unif_expo}), $\mathfrak{B}(\phi)$ (see (\ref{hyp:controle_coefficients_MS})) and $\mathcal{R}_{p, \lambda}(\alpha,\beta,\phi,V)$ (see (\ref{hyp:recursive_control_param_MS_expo})) are satisfied.\\

 Then, for every $\tilde{\alpha} \in (0,\alpha)$, there exists $\tilde{\beta} \in \mathbb{R}_+$ and $n_0 \in \mathbb{N}^{\ast}$, such that
\begin{align}
\label{eq:recursive_control_milstein_fonction_expo}
 \forall n   \geqslant n_0, \forall x \in \mathbb{R}^d,  \quad\widetilde{A}_{\gamma_n} \psi \circ V(x)\leqslant \frac{ \psi \circ V(x)}{V(x)}p\big(\tilde{\beta} - \tilde{\alpha} \phi\circ V(x)\big). 
\end{align}
Then, $\mathcal{RC}_{Q,V}(\psi,\phi,p\tilde{\alpha},p\tilde{\beta})$ (see (\ref{hyp:incr_sg_Lyapunov})) holds as soon as $\liminf\limits_{y \to + \infty} \phi(y) =+\infty$. Moreover, when $\phi=Id$ we have
\begin{align}
\label{eq:mom_exp_MS}
\sup_{n \in \mathbb{N}} \mathbb{E}[\psi \circ V (\overline{X}_{\Gamma_{n}})] < + \infty.
\end{align}
\end{myprop}
\begin{proof}
When $p=0$, the result is straightforward. Since $p \leqslant 1$, the function defined on $\mathbb{R}_+$ by  $y \mapsto y^p$ is concave. Using then the Taylor expansion at order 2 of the function $V$, we have, for every $x,y \in \mathbb{R}^d$,
\begin{align*}
V^p(y) - V^p(x) \leqslant & pV^{p-1}(x) \big(V(y)-V(x) \big) \\
\leqslant & pV^{p-1}(x) \big( \langle \nabla V(x),y-x \rangle +\frac{1}{2} \Vert D^2V \Vert_{\infty} \vert y - x \vert^2 \big).
\end{align*}
Using this inequality with $x=\overline{X}_{\Gamma_{n}}$ and $y=\overline{X}_{\Gamma_{n+1}}=\overline{X}_{\Gamma_{n}}+\Delta \overline{X}^1_{n+1} +\Delta \overline{X}^2_{n+1} $, with notations (\ref{def:incr_MS}), we derive 
\begin{align*}
V^p(\overline{X}_{\Gamma_{n}} &+ \Delta \overline{X}_{n+1} ) - V^p(\overline{X}_{\Gamma_{n}}) \\
 \leqslant & pV^{p-1}(\overline{X}_{\Gamma_{n}} ) \langle \nabla V(\overline{X}_{\Gamma_{n}}),\Delta \overline{X}^1_{n+1} +\Delta \overline{X}^2_{n+1} \rangle \\
 & + \frac{1}{2}  pV^{p-1}(\overline{X}_{\Gamma_{n}} ) \Vert D^2V \Vert_{\infty} ( \vert  \Delta \overline{X}^1_{n+1} \vert^2+\vert  \Delta \overline{X}^2_{n+1} \vert^2 + 2 \langle \Delta \overline{X}^1_{n+1}  , \Delta \overline{X}^2_{n+1}\rangle ) .
\end{align*}
%
%
%
%
It follows that
\begin{align*}
\mathbb{E}[\exp(\lambda V^p(\overline{X}_{\Gamma_{n+1}})) \vert \overline{X}_{\Gamma_{n}},\zeta_{\Gamma_{n}}] \leqslant H_{\gamma_{n+1}}(\overline{X}_{\Gamma_{n}},\zeta_{\Gamma_{n}})L_{\gamma_{n+1}}(\overline{X}_{\Gamma_{n}},\zeta_{\Gamma_{n}})
\end{align*}
with, for every $x \in \mathbb{R}^d$, every $z\in \{1,\ldots,M_0\}$ and every $\gamma \in \mathbb{R}_+^{\ast}$,
\begin{align*}
H_{\gamma}(x,z)=&\exp(\lambda V^p(x)+\gamma \lambda pV^{p-1}(x) \langle \nabla V(x), b(x,z) \rangle \\
& +\gamma^2  \frac{1}{2} \lambda   p \Vert D^2 V \Vert_{\infty} V^{p-1}(x)\vert b (x,z) \vert^2 )
\end{align*}
and
\begin{align*}
L_{\gamma}(x,z)= &\mathbb{E}[\exp( \sqrt{\gamma}  \lambda p V^{p-1}(x) \langle \nabla V(x)+\gamma \Vert D^2 V \Vert_{\infty}  b(x,z) , \sigma(x,z)  U \rangle \\
&+ \frac{1}{2} \gamma  \lambda  p  \Vert D^2 V  \Vert_{\infty}   V^{p-1}(x) \vert \sigma (x,z)  U \vert^2  )]
\end{align*}
where $U\sim \mathcal{N}(0,I_d)$. In order to compute $L_{\gamma}(x,z)$, we use Lemma \ref{lemme:transformee_laplace_prod_gauss} (see (\ref{eq:transformee_laplace_jointe})) with parameters $h=C_{\sigma}(x,z)^{-1} \gamma \lambda p $, $v=\sqrt{C_{\sigma}(x)\lambda p }V^{p-1}(x)\sigma^{\ast}(x,z)(  \nabla V(x)+ \gamma    \Vert D^2 V \Vert_{\infty}   b(x,z) )$ and the matrix 
\begin{align*}
\Sigma(x,z)=I_d- \Vert D^2 V  \Vert_{\infty}  C_{\sigma}(x,z)V^{p-1}(x)  \sigma^{\ast} \sigma(x,z)
\end{align*},
 where $\inf_{x \in \mathbb{R}^d} \inf_{z \in \{1,\ldots,M_0\}}C_{\sigma}(x,z)>0$ and $\Sigma(x,z) \in \in \mathcal{S}_{+, \ast}^d$i. It follows from (\ref{eq:transformee_laplace_jointe}) and $h/(2(1-h))\leqslant h$ for $h \in (0,1/2]$, that for every $\gamma \leqslant \inf_{x \in \mathbb{R}^d} \inf_{z \in \{1,\ldots,M_0\}}C_{\sigma}(x,z)/(2 \lambda p)$,
\begin{align*}
L_\gamma(x,z) \leqslant &\exp \Big(\frac{\gamma \lambda p  C_{\sigma}(x,z)^{-1}}{2(1 -\gamma \lambda p  C_{\sigma}(x,z)^{-1} )} \vert v \vert^2- \frac{1}{2} \gamma \lambda p  C_{\sigma}(x,z)^{-1} \ln ({\det(\Sigma(x,z))})  \Big) \\
\leqslant & \exp \Big(\gamma  \lambda p  C_{\sigma}(x,z)^{-1} \vert v \vert^2- \frac{1}{2} \gamma \lambda p  C_{\sigma}(x,z)^{-1} \ln ({\det(\Sigma(x,z))})  \Big)
\end{align*}
At this point, we focus on the first term inside the exponential. We have
\begin{align*}
\vert v \vert^2 \leqslant & C_{\sigma}(x,z) \lambda p  V^{2p-2}(x)   \big( \langle \sigma \sigma^{\ast} (x,z)\nabla V(x),\nabla V(x)\rangle \\
&+\Tr[\sigma \sigma^{\ast}(x,z)]( \gamma   \Vert D^2 V \Vert_{\infty} 2  \langle \nabla V(x), b(x,z)  \rangle + \gamma^2  \Vert D^2 V \Vert_{\infty}^2 \vert   b(x,z) \vert^2 ) \big)
\end{align*}
Using $\mathfrak{B}(\phi)$ (see (\ref{hyp:controle_coefficients_MS})), (\ref{hyp:dom_recurs_MS}) and  $\mathcal{R}_{p, \lambda}(\alpha,\beta,\phi,V)$ (see (\ref{hyp:recursive_control_param_MS_expo})), it follows that there exists $\overline{C}>0$ such that
\begin{align*}
H_{\gamma}(x,z)L_{\gamma}(x,z) \leqslant & \exp \big(\lambda V^p(x)+\gamma \lambda p  V^{p-1}(x)(\beta - \alpha \phi \circ V(x))+\overline{C} \gamma^2 V^{p-1}(x) \phi \circ V(x) \big) 
\end{align*}
which can be rewritten
\begin{align*}
H_{\gamma}(x,z)L_{\gamma}(x,z) \leqslant & \exp \Big( \Big(1-\gamma p \alpha \frac{\phi \circ V(x)}{V(x)} \Big) \lambda V^{p}(x) \\
& \quad \quad + \gamma p  \alpha \frac{\phi \circ V(x)}{V(x)}  V^p(x) \Big( \frac{ \lambda  \beta }{ \alpha \phi \circ V(x)}+\gamma \overline{C}/(\alpha p) \Big) \Big ) .
\end{align*} 
Using the convexity of the exponential function, we have for every $\gamma p \alpha C_{\phi}<1$,
\begin{align*}
H_{\gamma}(x,z)L_{\gamma}(x,z) \leqslant & \exp \big(\lambda V^p(x)\big) -\gamma p \alpha \frac{\phi \circ V(x)}{V(x)}  \exp\big(\lambda V^p(x)\big)\\
&+ \gamma p \alpha \frac{\phi \circ V(x)}{V(x)} \exp \Big(V^p(x)\Big( \frac{\lambda  \beta }{ \alpha \phi \circ V(x)} +\gamma \overline{C}/(\alpha p) \Big) \Big).
\end{align*} 
It remains to study the last term of the $r.h.s$ of the above inequality. The function defined on $[v_{\ast},+\infty)$ by $y \mapsto \exp(y^p(\frac{\lambda \beta }{\alpha\phi (y)}+\gamma \overline{C}/(\alpha p) ))$ is continuous and locally bounded. Moreover, by $\mathcal{R}_{p, \lambda}(\alpha,\beta,\phi,V)$ (see (\ref{hyp:recursive_control_param_MS_expo})), we have $\liminf\limits_{y \to + \infty} \phi(y)>\beta_+/\alpha$. Hence, there exists $\zeta \in (0,1)$ and $y_{\zeta}\geqslant v_{\ast}$ such that $\phi(y) \geqslant \beta_+/(\alpha \zeta)$ for every $y \geqslant y_{\zeta}$. Consequently, as soon as $\gamma <\zeta \lambda \alpha p/ \overline{C}$,  for every $\tilde{\alpha}\in (0, \alpha)$ there exists $\tilde{\beta} \geqslant 0$ such that
\begin{align*}
\frac{\phi \circ V(x)}{V(x)} \exp \Big(V^p(x)\Big( \frac{\lambda  \beta }{ \alpha \phi \circ V(x)} +\gamma \overline{C}/(\alpha p) \Big) \Big)\leqslant& \frac{\tilde{\beta}}{\alpha} \frac{\exp(\lambda V^p(x))}{V(x)} \\
&+\frac{\alpha -\tilde{\alpha}}{\alpha}  \frac{\phi \circ V(x)}{V(x)} \exp(\lambda V^p(x))
\end{align*}
and the proof of the recursive control (\ref{eq:recursive_control_milstein_fonction_expo}) is completed. Finally (\ref{eq:mom_exp_MS}) follows from (\ref{eq:mom_psi_V}), which follow from the equation above, and Lemma \ref{lemme:mom_psi_V}.
\end{proof}
\subsubsection{Infinitesimal control}
\begin{myprop}
\label{prop:MS_infinitesimal_approx}
Suppose that the sequence $(U_n)_{n \in \mathbb{N}^{\ast}}$ satisfies $M_{\mathcal{N},2}(U)$ (see (\ref{hyp:matching_normal_moment_ordre_q_va_schema_MS})). Also assume that for every $z \in \{1,\ldots,M_0\}$, $b(.,z)$ and $\sigma(.,z)$ have sublinear growth and that $\sup_{n \in \mathbb{N}^{\ast}} \nu_n^{\eta}( \Tr[\sigma\sigma^{\ast}] )< + \infty, \; a.s.$ \\

Then, $\mathcal{E}(\widetilde{A},A,\DomA_0) $ (see (\ref{hyp:erreur_tems_cours_fonction_test_reg})) is fulfilled.
%
%
%
\end{myprop}
\begin{proof}
First we recall that $\DomA_0 =\{f:\mathbb{R}^d \times \{1,\ldots,M_0 \}, \forall z \in \{1,\ldots,M_0\}, f(.,z)\in\mathcal{C}^2_K(\mathbb{R}^d)\}$ and we write, for $f \in \DomA_0$,
\begin{align*}
f ( \overline{X}_{\Gamma_{n+1}}, \zeta_{\Gamma_{n+1}})-f(\overline{X}_{\Gamma_{n}}, \zeta_{\Gamma_{n}})=&f(\overline{X}_{\Gamma_{n+1}}, \zeta_{\Gamma_{n+1}})- f(\overline{X}_{\Gamma_{n+1}}, \zeta_{\Gamma_{n}}) \\
& +f(\overline{X}_{\Gamma_{n+1}}, \zeta_{\Gamma_{n}})-f(\overline{X}_{\Gamma_{n}}, \zeta_{\Gamma_{n}}).
\end{align*}
We study the first term of the $r.h.s.$ of the above equation. Since $U$ and $\zeta$ are independent, we have, with notation (\ref{def:incr_MS}),
\begin{align*}
\mathbb{E}[ f(\overline{X}_{\Gamma_{n+1}}, \zeta_{\Gamma_{n+1}})- &f(\overline{X}_{\Gamma_{n+1}}, \zeta_{\Gamma_{n}}) \vert \overline{X}_{\Gamma_{n}}, \zeta_{\Gamma_{n}}, \Delta \overline{X}_{n+1}   ] \\
=& \gamma_{n+1} \sum_{z=1}^{M_0} \big( q_{\zeta_{\Gamma_{n}}  ,z}+  o(\gamma_{n+1} ) \big) f(\overline{X}_{\Gamma_{n+1}}, z) .
\end{align*}
Using Taylor expansions at order one and two, for every $z \in \{1,\ldots,M_0\}$ and the fact that the sequence $(U_n)_{n \in \mathbb{N}^{\ast}}$ is $i.i.d.$, we obtain
\begin{align*}
\mathbb{E}[&  f(\overline{X}_{\Gamma_{n+1}}, z) - f(\overline{X}_{\Gamma_{n}}, z)  \vert \overline{X}_{\Gamma_{n}}=x, \zeta_{\Gamma_{n}}]  \\
= & \mathbb{E}[  f(\overline{X}_{\Gamma_{n}}+\Delta \overline{X}^1_{n+1}, z) -  f(\overline{X}_{\Gamma_{n}}, z)  \vert \overline{X}_{\Gamma_{n}}=x,\zeta_{\Gamma_{n}}] \\
&+   \mathbb{E}[  f(\overline{X}_{\Gamma_{n+1}},z) - f(\overline{X}_{\Gamma_{n}}+ \Delta \overline{X}^1_{n+1}, z)  \vert \overline{X}_{\Gamma_{n}}=x ,\zeta_{\Gamma_{n}} ] \\
\leqslant&  \int_0^1 \vert  \nabla_xf(x+  \theta b(x,\zeta_{\Gamma_{n}}) \gamma_{n+1} ,z) \vert  \vert b(x,\zeta_{\Gamma_{n}}) \gamma_{n+1}  \vert d \theta  \\
 & +   \int_0^1 \vert D^2_xf(x+   b(x,\zeta_{\Gamma_{n}}) \gamma_{n+1}+ \theta \sigma(x,\zeta_{\Gamma_{n}}) \sqrt{ \gamma_{n+1}} v ,z) \vert  \vert  \sqrt{ \gamma_{n+1}}  \sigma(x,\zeta_{\Gamma_{n}}) u \vert^2 d \theta \tilde{\mathbb{P}}_{U}(du) .
\end{align*}
where $\tilde{\mathbb{P}}_{U}$ denotes the distribution of $U_1$. Combining the two last inequalities, we derive
\begin{align*}
\gamma_{n+1}^{-1} \mathbb{E}[ f & (\overline{X}_{\Gamma_{n+1}}, \zeta_{\Gamma_{n+1}})-   f(\overline{X}_{\Gamma_{n+1}}, \zeta_{\Gamma_{n}}) \vert \overline{X}_{\Gamma_{n}}, \zeta_{\Gamma_{n}}  ] \\
 \leqslant   & \sum_{z=1}^{M_0} q_{\zeta_{\Gamma_{n}}  ,z}f(\overline{X}_{\Gamma_{n}}, z)   + o(\gamma_{n+1})  \Vert f \Vert_{\infty} \\
&+  \sum_{z=1}^{M_0}\big( \vert q_{\zeta_{\Gamma_{n}}  ,z}  \vert+  o(\gamma_{n+1} ) \big) \big(\Lambda_{f,1} (\overline{X}_{\Gamma_{n}}, \zeta_{\Gamma_{n}} ,\gamma_{n+1})    \vert b(\overline{X}_{\Gamma_{n}},\zeta_{\Gamma_{n}})   \vert \\
& \qquad \qquad  \qquad \qquad  \qquad \qquad +\Lambda_{f,2} (\overline{X}_{\Gamma_{n}}, \zeta_{\Gamma_{n}} ,\gamma_{n+1}) \Tr[\sigma \sigma^{\ast}(\overline{X}_{\Gamma_{n}},\zeta_{\Gamma_{n}})   ]\big).
\end{align*}
We study each term in the $r.h.s.$ of the inequality above. First, we have $\Lambda_{f,1}(x,z, \gamma)= \vert b (x,z) \vert  \tilde{\mathbb{E}}[\tilde{\Lambda}_{f,1}(x,z,\gamma)]$ where $\tilde{\Lambda}_{f,1}(x,z,\gamma)=\tilde{\mathcal{R}}_{f,1}(x,z,\gamma, \Theta)$ with $\Theta \sim \mathcal{U}_{[0,1]}$ under $\tilde{\mathbb{P}}$, and
%
%
%
%
\begin{align*}
\begin{array}{crcl}
\tilde{\mathcal{R}}_{f,1} & :  \mathbb{R}^d \times \{1,\ldots,M_0 \} \times \mathbb{R}_+   \times [0,1] & \to & \mathbb{R}_+ \\
 &( x,z, \gamma , \theta ) & \mapsto &   \gamma\sum\limits_{w=1}^{M_0} \vert \nabla_xf(x+  \theta b(x,z) \gamma ,w ) \vert  .
\end{array}
\end{align*}
We are going to prove that $\mathcal{E}(\widetilde{A},A,\DomA_0)$ \ref{hyp:erreur_tems_cours_fonction_test_reg_Lambda_representation_1} (see (\ref{hyp:erreur_temps_cours_fonction_test_reg_Lambda_representation_2_1})) holds.\\

 Since $b$ has sublinear growth w.r.t. its first variable, there exists $C_{b} \geqslant 0$ such that $\vert b(x,z)\vert   \leqslant C_{b}(1+ \vert x \vert ) $ for every $x \in \mathbb{R}^d$ and $z \in \{1,\ldots,M_0 \}$. Therefore, since $f$ has a compact support, it follows that there exists $\gamma_0>0$ and $R>0$ such that we have $\sup_{\vert x \vert >R,z \in \{1,\ldots,M_0 \} } \sup_{\gamma \leqslant \gamma_0} \tilde{\mathcal{R}}_{f,1}(x,z,\gamma, \theta) =0$ for every $\theta \in [0,1]$ which implies $\mathcal{E}(\widetilde{A},A,\DomA_0)$ \ref{hyp:erreur_tems_cours_fonction_test_reg_Lambda_representation_1} (ii).\\
  Since $\nabla_x f$ is bounded, it is immediate that $\mathcal{E}(\widetilde{A},A,\DomA_0)$ \ref{hyp:erreur_tems_cours_fonction_test_reg_Lambda_representation_1} (i) holds. \\
 Finally, $b$ is locally bounded and defining and $g_{1}(x,z)= \mathds{1}_{x \leqslant R} \vert b(x,z) \vert$, the couple $(\tilde{\Lambda}_{f,1},g_1)$ satisfies $\mathcal{E}(\widetilde{A},A,\DomA_0)$ \ref{hyp:erreur_tems_cours_fonction_test_reg_Lambda_representation_1}. \\ 
%
%

Now, we have $\Lambda_{f,2}(x,z, \gamma)= g_2 (x,z)  \tilde{\mathbb{E}}[\tilde{\Lambda}_{f,2}(x,z,\gamma)]$ where $\tilde{\Lambda}_{f,2} ( x,z,\gamma)= \tilde{\mathcal{R}}_{f,2}(x,z,\gamma,U, \Theta)$ with $U \sim \mathbb{P}_U$, $\Theta \sim \mathcal{U}_{[0,1]}$ under $\tilde{\mathbb{P}}$ and $g_{2}(x,z)=  \Tr[ \sigma \sigma^{\ast}(x,z) ]$ and
 \begin{align*}
\begin{array}{crcl}
\tilde{\mathcal{R}}_{f,2} & :  \mathbb{R}^d \times \{1,\ldots,M_0 \} \times \mathbb{R}_+ \times \mathbb{R}^{d}  \times [0,1] & \to & \mathbb{R}_+ \\
 &( x,z, \gamma , u, \theta ) & \mapsto &  \tilde{\mathcal{R}}_{f,2} ( x,z, \gamma , u, \theta )  ,
\end{array}
\end{align*}
with
\begin{align*}
\tilde{\mathcal{R}}_{f,2} ( x,z, \gamma , u, \theta )  =\vert   \sqrt{ \gamma} u \vert^2 \sum\limits_{w=1}^{M_0} \vert  D^2_xf(x+  b(x,z) \gamma + \theta \sigma(x,z) \sqrt{ \gamma} u,w)\vert   .
\end{align*}
We are going to prove that $\mathcal{E}(\widetilde{A},A,\DomA_0)$ \ref{hyp:erreur_tems_cours_fonction_test_reg_Lambda_representation_1} (see (\ref{hyp:erreur_temps_cours_fonction_test_reg_Lambda_representation_2_1})) holds for the couple $(\tilde{\Lambda}_{f,2} ,g_2)$. We fix $u \in \mathbb{R}^N$ and $\theta \in [0,1]$.\\

Since the functions $b$ and $ \sigma$ have sublinear growth, there exists $C_{b, \sigma} \geqslant 0$ such that $\vert b(x,z)\vert +  \vert \sigma(x,z) \vert \leqslant C_{b, \sigma}(1+ \vert x \vert ) $ for every $x \in \mathbb{R}^d$ and $z \in \{1,\ldots,M_0 \}$. Therefore, since $f$ has compact support, there exists $\gamma_0(u,\theta)>0$ and $R>0$ such that 
\begin{align*}
\sup_{\vert x \vert >R, z \in \{1,\ldots,M_0 \} } \sup_{\gamma \leqslant \gamma_0(u,\theta)} \vert\tilde{\mathcal{R}}_{f,2}(x,z,\gamma,u,\theta) \vert=0.
\end{align*}
 It follows that $\mathcal{E}(\widetilde{A},A,\DomA_0)$ \ref{hyp:erreur_tems_cours_fonction_test_reg_Lambda_representation_1} (ii) holds.\\
  Moreover since $D^2_x f$ is bounded,  it is immediate that $\mathcal{E}(\widetilde{A},A,\DomA_0)$ \ref{hyp:erreur_tems_cours_fonction_test_reg_Lambda_representation_1} (i) is also satisfied. \\
  Finally, we recall that $\sup_{n \in \mathbb{N}^{\ast}} \nu_n^{\eta}( \Tr[\sigma\sigma^{\ast}] )< + \infty, \; a.s. $ and $U$ is bounded in $\mbox{L}^2$ and then $\mathcal{E}(\widetilde{A},A,\DomA_0)$ \ref{hyp:erreur_tems_cours_fonction_test_reg_Lambda_representation_1} holds for $(\tilde{\Lambda}_{f,2} ,g_2)$.\\

Moreover, it is immediate to show that $\mathcal{E}(\widetilde{A},A,\DomA_0)$ \ref{hyp:erreur_temps_cours_fonction_test_reg_Lambda_representation_2} (see \ref{hyp:erreur_temps_cours_fonction_test_reg_Lambda_representation_2_2})) holds for every couple of functions with form $(\underset{n \to + \infty}{o}(\gamma_{n+1})  \Vert f \Vert_{\infty} ,1)$ which concludes the study of the first term.\\

It remains to study $\mathbb{E}[ f(\overline{X}_{\Gamma_{n+1}}, \zeta_{\Gamma_{n}})- f(\overline{X}_{\Gamma_{n}}, \zeta_{\Gamma_{n}}) \vert \overline{X}_{\Gamma_{n}}, \zeta_{\Gamma_{n}}  ]$. Using once again Taylor expansions at order one and two, we derive
\begin{align*}
\gamma_{n+1}^{-1} \big( \mathbb{E}[  f(\overline{X}_{\Gamma_{n+1}}, \zeta_{\Gamma_{n}}) - & f(\overline{X}_{\Gamma_{n}}, \zeta_{\Gamma_{n}})  \vert \overline{X}_{\Gamma_{n}}=x, \zeta_{\Gamma_{n}}=z] \\
-& \langle \nabla_xf(x ,z ),b(x,z) \rangle -  \frac{1}{2}\sum_{i,j=1}^d (\sigma \sigma^{\ast})_{i,j}(x, z) \frac{\partial^2f}{\partial x_i \partial x_j}(x, z)   \big)  \\
\leqslant&  \int_0^1 \vert  \nabla_xf(x+  \theta b(x,z) \gamma_{n+1} ,z) - \nabla_xf(x) \vert  \vert b(x,z)   \vert d \theta  \\
& +   \int_0^1 \vert D^2_xf(x+   b(x,z) \gamma_{n+1}+ \theta \sigma(x,z) \sqrt{ \gamma_{n+1}} u ,z) \\
& \qquad  \qquad \qquad \qquad-D^2_xf(x)\vert  \vert  \sigma(x,z)  v \vert^2 d \theta p_{U}(du) .
\end{align*}
 Using a similar reasoning as before, one can show that $\mathcal{E}(\widetilde{A},A,\DomA_0)$ \ref{hyp:erreur_tems_cours_fonction_test_reg_Lambda_representation_1} holds for $(\tilde{\Lambda}_{f,3} ,g_1)$ and $(\tilde{\Lambda}_{f,4} ,g_2)$ where $\tilde{\Lambda}_{f,3}(x,z,\gamma)=\tilde{\mathcal{R}}_{f,3}(x,z,\gamma, \Theta)$ and $\tilde{\Lambda}_{f,4}(x,z,\gamma)=\tilde{\mathcal{R}}_{f,4}(x,z,\gamma, U,\Theta)$  with $U \sim p_U$ and $\Theta \sim \mathcal{U}_{[0,1]}$ under $\tilde{\mathbb{P}}$,
\begin{align*}
\begin{array}{crcl}
\tilde{\mathcal{R}}_{f,3} & :  \mathbb{R}^d \times \{1,\ldots,M_0 \} \times \mathbb{R}_+   \times [0,1] & \to & \mathbb{R}_+ \\
 &( x,z, \gamma , \theta ) & \mapsto &   \vert \nabla_xf(x+  \theta b(x,z) \gamma ,z ) -\nabla_xf(x ,z ) \vert,
\end{array}
\end{align*}
and
 \begin{align*}
\begin{array}{crcl}
\tilde{\mathcal{R}}_{f,4} & :  \mathbb{R}^d \times \{1,\ldots,M_0 \} \times \mathbb{R}_+ \times \mathbb{R}^{d }  \times [0,1] & \to & \mathbb{R}_+ \\
 &( x,z, \gamma , u, \theta ) & \mapsto &  \tilde{\mathcal{R}}_{f,4}( x,z, \gamma , u, \theta )  ,
\end{array}
\end{align*}
with
\begin{align*}
\tilde{\mathcal{R}}_{f,4}( x,z, \gamma , u, \theta ) = \vert D^2_xf(x+  b(x,z) \gamma + \theta \sigma(x,z) \sqrt{ \gamma} u ,z) - D^2_xf(x) \vert\vert  u \vert^2 .
\end{align*}
We gather all the terms together noticing that $\tilde{\Lambda}_{f,q}= \tilde{\Lambda}_{-f,q}$, $q \in \{1,\ldots , 4\}$, and the proof is completed.
\end{proof}
\subsubsection{Proof of Growth control and Step Weight assumptions}
\paragraph{Test functions with polynomial growth. }
\begin{lemme}
\label{lemme:incr_lyapunov_X_MS}
 Let $p \geqslant 1,a \in (0,1]$, $\rho \in [1,2]$, $s \geqslant 1$ and let $\psi_p(y)=y^p$ and $\phi(y)=y^a$ . We suppose that the sequence $(U_n)_{n \in \mathbb{N}^{\ast}}$ satisfies $M_{(\rho/2) \vee  (p\rho /s) }(U)$ (see (\ref{hyp:moment_ordre_p_va_schema_MS})). Then, for every $n \in \mathbb{N}$, we have
\begin{align}
\label{lemme:incr_lyapunov_X_MS_f_DomA}
 \forall f \in \DomA_0 ,\quad  \mathbb{E}[  \vert f(\overline{X}_{\Gamma_{n+1}}, \zeta_{\Gamma_{n+1}})- & f( \overline{X}^1_{\Gamma_{n}}, \zeta_{\Gamma_{n}} ) \vert^{\rho}\vert \overline{X}_{\Gamma_{n}}, \zeta_{\Gamma_{n}} ] \nonumber \\
 &  \leqslant   C_f \gamma_{n+1}^{\rho/2}  1 \vee\Tr[ \sigma \sigma^{\ast} (\overline{X}_{\Gamma_n} , \zeta_{\Gamma_{n}}) ]^{\rho/2}  .
\end{align}
with notations (\ref{def:incr_MS}). In other words, we have $\mathcal{GC}_{Q}(\DomA_0,1 \vee \Tr[ \sigma \sigma^{\ast} ]^{\rho/2} ,\rho,\epsilon_{\mathcal{I}}) $ (see (\ref{hyp:incr_X_Lyapunov})) with $\epsilon_{\mathcal{I}}(\gamma)=\gamma^{\rho/2}$ for every $\gamma \in \mathbb{R}_+$. \\
%
%
%
Moreover, if (\ref{hyp:Lyapunov_control_MS}), (\ref{hyp:Lyapunov_control_MS_unif}) and $\mathfrak{B}(\phi)$ (see (\ref{hyp:controle_coefficients_MS})) hold and $p\rho/s \leqslant p+a-1$, then, for every $n \in \mathbb{N}$, we have
\begin{align}
\label{lemme:incr_lyapunov_X_MS_f_tens}
 \mathbb{E}[\vert  V^{p/s}(\overline{X}_{\Gamma_{n+1}}, \zeta_{\Gamma_{n+1}})- V^{p/s}(\overline{X}_{\Gamma_{n}}, \zeta_{\Gamma_{n}}) \vert^{\rho}\vert &\overline{X}_{\Gamma_{n}}, \zeta_{\Gamma_{n}} ]  \nonumber \\
& \leqslant  C \gamma_{n+1}^{\rho/2} V^{p+a-1}(\overline{X}_{\Gamma_{n}}, \zeta_{\Gamma_{n}}),
\end{align}
In other words, we have $\mathcal{GC}_{Q}(V^{p/s},V^{p+a-1},\rho,\epsilon_{\mathcal{I}}) $ (see (\ref{hyp:incr_X_Lyapunov})) with $\epsilon_{\mathcal{I}}(\gamma)=\gamma^{\rho/2}$ for every $\gamma \in \mathbb{R}_+$. \\
\end{lemme}
\begin{proof}
We begin by noticing that, with notations (\ref{def:incr_MS}),
\begin{align*}
 \vert \overline{X}_{\Gamma_{n+1}} -\overline{X}^1_{\Gamma_{n+1}} \vert\leqslant C  \gamma_{n+1}^{1/2}\Tr[ \sigma \sigma^{\ast} (\overline{X}_{\Gamma_n} , \zeta_{\Gamma_{n}}) ]^{1/2} \vert U_{n+1} \vert 
\end{align*} 
Let $f \in \DomA_0$. We employ this estimation and since for $f \in \DomA_0$ then $f(.,z)$ is uniformly Lipschitz in $z \in \{1,\ldots,M_0\}$, it follows that  
\begin{align*}
\mathbb{E}\big[  \vert f(\overline{X}_{\Gamma_{n+1}}, \zeta_{\Gamma_{n}})- f( \overline{X}^1_{\Gamma_{n}}, \zeta_{\Gamma_{n}} ) \vert^{\rho}\vert \overline{X}_{\Gamma_{n}}, \zeta_{\Gamma_{n}} \big] \leqslant  C  \gamma_{n+1}^{\rho/2}\vert \sigma \sigma^{\ast} (\overline{X}_{\Gamma_n} , \zeta_{\Gamma_{n}}) \vert^{\rho/2} .
\end{align*}
Moreover, 
\begin{align*}
\mathbb{E}[  \vert  f( &\overline{X}_{\Gamma_{n+1}}, \zeta_{\Gamma_{n+1}})-  f(\overline{X}_{\Gamma_{n+1}}, \zeta_{\Gamma_{n}}) \vert^{\rho}  \vert  \overline{X}_{\Gamma_{n}}, \zeta_{\Gamma_{n}}  ] \\
 =&   \gamma_{n+1} \sum_{z=1}^{M_0}( q_{\zeta_{\Gamma_{n}}  ,z}+  \underset{n \to + \infty}{o}(\gamma_{n+1} )) \mathbb{E} [\vert f(\overline{X}_{\Gamma_{n+1}},z) -f(\overline{X}_{\Gamma_{n+1}}, \zeta_{\Gamma_{n}}) \vert^{\rho} \vert \overline{X}_{\Gamma_{n}}, \zeta_{\Gamma_{n}} ]  \\
 \leqslant& C \gamma_{n+1} \Vert f \Vert_{\infty}^{\rho}.
\end{align*}
Gathering both terms concludes the study for $f \in \DomA_0$.\\
 We focus now on the case $f=V^{p/s}$. We notice that $\mathfrak{B}(\phi)$ (see (\ref{hyp:controle_coefficients_MS})) implies that for any $n \in \mathbb{N}$,
\begin{align*}
 \vert \overline{X}_{\Gamma_{n+1}} -\overline{X}_{\Gamma_n} \vert\leqslant C  \gamma_{n+1}^{1/2}\sqrt{\phi \circ V (\overline{X}_{\Gamma_n},\zeta_{\Gamma_{n}} )} (1+\vert U_{n+1} \vert).
\end{align*}  
We rewrite the term that we study as follows 
\begin{align*}
V^{p/s}( \overline{X}_{\Gamma_{n+1}}, \zeta_{\Gamma_{n+1}} )-V^{p/s}(\overline{X}_{\Gamma_n},\zeta_{\Gamma_{n}} ) = & V^{p/s}( \overline{X}_{\Gamma_{n+1}}, \zeta_{\Gamma_{n}} )-V^{p/s}(\overline{X}_{\Gamma_n},\zeta_{\Gamma_{n}} ) \\
& + V^{p/s}( \overline{X}_{\Gamma_{n+1}}, \zeta_{\Gamma_{n+1}} )-V^{p/s}(\overline{X}_{\Gamma_{n+1}},\zeta_{\Gamma_{n}} ).
\end{align*}
We study the first term of the $r.h.s.$ of the equality above. Using  the following inequality
 \begin{align}
 \label{eq:puisssance_sup_1}
\forall u,v \in \mathbb{R}_+,\forall \alpha \geqslant 1, \qquad \vert u^{\alpha} -v^{\alpha} \vert \leqslant &  \alpha 2^{\alpha-1} ( v^{\alpha-1} \vert u -v \vert + \vert u -v \vert ^{\alpha} ),
 \end{align}
with $\alpha=2p/s$, it follows from (\ref{hyp:Lyapunov_control_MS}) that $\sqrt{V(.,z)}$ is Lipschitz uniformly in $z \in \{1,\ldots,M_0 \}$ and
\begin{align*}
\big\vert V^{p/s}( \overline{X}_{\Gamma_{n+1}},z )-& V^{p/s}(\overline{X}_{\Gamma_n},z ) \big\vert \\
\leqslant &  2^{2p/s}p/s \big( V^{p/s-1/2}(\overline{X}_{\Gamma_n},z ) \big\vert \sqrt{V}( \overline{X}_{\Gamma_{n+1}} ,z)  - \sqrt{V} (\overline{X}_{\Gamma_n} ,z) \big\vert  \\
& + \vert  \sqrt{V}( \overline{X}_{\Gamma_{n+1}} ,z)  - \sqrt{V}(\overline{X}_{\Gamma_n} , z) \vert^{2p/s}  \big) \\
\leqslant &  2^{2p/s}p/s \big(  [ \sqrt{V}]_1 V^{p/s-1/2}(\overline{X}_{\Gamma_n},z ) \vert  \overline{X}_{\Gamma_{n+1}}- \overline{X}_{\Gamma_n} \vert \\
 & + [ \sqrt{V}]_1^{2p/s}  \vert   \overline{X}_{\Gamma_{n+1}} -\overline{X}_{\Gamma_n}\vert^{2p/s}   \big).
\end{align*}
We use the assumption $p\rho/s \leqslant p+a-1$, $a\in (0,1]$, $p\geqslant1$ and it follows from $\mathfrak{B}(\phi)$ (see (\ref{hyp:controle_coefficients_MS})) and (\ref{hyp:Lyapunov_control_MS_unif}) when $z  \neq \zeta_{\Gamma_n}$, that
\begin{align*}
\mathbb{E}[ \vert V^{p/s}( \overline{X}_{\Gamma_{n+1}},z )-V^{p/s}(\overline{X}_{\Gamma_n},z)  \vert^{\rho} \vert \overline{X}_{\Gamma_n}, \zeta_{\Gamma_{n}}]
\leqslant C \gamma_{n+1}^{\rho/2} V^{p+a-1}(\overline{X}_{\Gamma_n},z) .
\end{align*}
In order to treat the first term, we put $z =\zeta_{\Gamma_n}$ in this estimation. It remains to study the second term. We notice that since $p\rho/s \leqslant p+a-1$, it is immediate from the previous inequality that for every $z\in \{1,\ldots,M_0\}$, we have
\begin{align*}
\mathbb{E}\big[ V^{p \rho/s}( \overline{X}_{\Gamma_{n+1}},z ) \vert \overline{X}_{\Gamma_n}, z \big] \leqslant C V^{p+a-1}(\overline{X}_{\Gamma_n},z) .
\end{align*}.
 We focus on the term to estimate and using this inequality, we obtain
\begin{align*}
\mathbb{E}[  \vert V^{p/s}(\overline{X}_{\Gamma_{n+1}}, & \zeta_{\Gamma_{n+1}})- V^{p/s}(\overline{X}_{\Gamma_{n+1}}, \zeta_{\Gamma_{n}}) \vert^{\rho}  \vert  \overline{X}_{\Gamma_{n}}, \zeta_{\Gamma_{n}}  ] \\
 =&   \gamma_{n+1} \sum_{z=1}^{M_0} \big( q_{\zeta_{\Gamma_{n}}  ,z}+  o(\gamma_{n+1} ) \big) \\
 & \qquad \qquad \qquad \times \mathbb{E} [\vert V^{p/s}(\overline{X}_{\Gamma_{n+1}},z) - V^{p/s}(\overline{X}_{\Gamma_{n+1}}, \zeta_{\Gamma_{n}}) \vert^{\rho} \vert \overline{X}_{\Gamma_{n}}, \zeta_{\Gamma_{n}} ]  \\
 \leqslant & C \gamma_{n+1} \sum_{z=1}^{M_0} \big( \vert  q_{\zeta_{\Gamma_{n}}  ,z} \vert+ \gamma_{n+1} \big) \big( V^{p+a-1}(\overline{X}_{\Gamma_n},z) +V^{p+a-1}(\overline{X}_{\Gamma_n}, \zeta_{\Gamma_{n}}) \big)\\
 \leqslant& C \gamma_{n+1} V^{p+a-1}(\overline{X}_{\Gamma_{n}}, {\zeta_{\Gamma_{n}}  }),
\end{align*}
where the last inequality follows from (\ref{hyp:Lyapunov_control_MS_unif}). We rearrange the terms and the proof of (\ref{lemme:incr_lyapunov_X_MS_f_tens}) is completed.
\end{proof}
\paragraph{Test functions with exponential growth. }
\begin{lemme}
\label{lemme:incr_lyapunov_X_MS_expo}
 Let $p\in [0,1], \lambda \geqslant 0$, $s \geqslant 1$, $\rho \in [1,2]$ and let $\phi:[v_{\ast},\infty )\to \mathbb{R}_+$ be a continuous function such that $C_{\phi}:= \sup_{y \in [v_{\ast},\infty )}\phi(y)/y< +\infty$ and let $\psi(y)=\exp(\lambda y^p)$. We assume that $ \rho < s$, (\ref{hyp:Lyapunov_control_MS}), (\ref{hyp:Lyapunov_control_MS_unif_expo}) and $\mathfrak{B}(\phi)$ (see (\ref{hyp:controle_coefficients_MS})) hold, and that
\begin{align}
\label{hyp:incr_lyapunov_X_MS_expo}
 \forall \tilde{\lambda} \leqslant \lambda, \exists C\geqslant 0, & \forall n \in \mathbb{N}, \nonumber\\
 & \mathbb{E}[  \exp( \tilde{\lambda} V^p(\overline{X}_{\Gamma_{n+1}}))  \vert \overline{X}_{\Gamma_n}, \zeta_{\Gamma_n}] \leqslant  C \exp(\tilde{\lambda} V^p(\overline{X}_{\Gamma_{n}}))   .
\end{align}
Then, for every $n \in \mathbb{N}$, we have
\begin{align}
\label{eq:incr_lyapunov_X_MS_f_tens_expo}
\mathbb{E}[\vert \exp(  \lambda/s V^p(\overline{X}_{\Gamma_{n+1}}))- & \exp(\lambda/s V^p(\overline{X}_{\Gamma_{n}}))  \vert^{\rho} \vert \overline{X}_{\Gamma_{n}}, \zeta_{\Gamma_n}] \nonumber  \\
\leqslant &  C \gamma_{n+1}^{\rho(p \wedge 1/2)} \frac{\phi \circ V (\overline{X}_{\Gamma_{n}}) }{V(\overline{X}_{\Gamma_{n}}) }\exp(\lambda V^p(\overline{X}_{\Gamma_{n}})) .
\end{align}
In other words, we have $\mathcal{GC}_{Q}(\exp(\lambda/s  V^{p}),V^{-1}  .\phi \circ V .\exp(\lambda  V^{p}),\rho,\epsilon_{\mathcal{I}}) $ (see (\ref{hyp:incr_X_Lyapunov})) and $\epsilon_{\mathcal{I}}(\gamma)=\gamma^{\rho(p \wedge 1/2)}$ for every $\gamma \in \mathbb{R}_+$. 
\end{lemme}
%
%
%
%
%
%
%
%
%
%
%
\begin{proof}
When $p=0$ the result is straightforward. We begin by noticing that $\mathfrak{B}(\phi)$ (see (\ref{hyp:controle_coefficients_MS})) implies that for every $n \in \mathbb{N}$,
\begin{align*}
 \vert \overline{X}_{\Gamma_{n+1}} -\overline{X}_{\Gamma_n} \vert\leqslant C  \gamma_n^{1/2}  \sqrt{\phi \circ V (\overline{X}_{\Gamma_n})}(1+\vert U_{n+1} \vert^2 ).
\end{align*} 
Let $x,y \in \mathbb{R}^d$. From Taylor expansion at order one, we derive,
\begin{align}
\label{eq:preuve_growth_control_expo_decomp_MS}
\big\vert\exp(\lambda/s & V^p(y))-\exp(\lambda/s V^p(x)) \big\vert \nonumber \\
&\leqslant  \frac{\lambda}{s}\big(\exp(\lambda/s V^p(y))+\exp(\lambda/s V^p(x))\big) \big\vert  V^p(y) -V^p(x) \big\vert .
\end{align}
First, let $p \in[1/2,1]$ we use (\ref{eq:puisssance_sup_1}) with $\alpha=2p$ and since $\sqrt{V}$ is Lipschitz, we obtain
\begin{align*}
\vert V^{p}( y )-V^{p}(x)\vert \leqslant & 2^{2p}p(V^{p-1/2}(x)   \vert \sqrt{V}( y )-\sqrt{V}(x)\vert+\vert \sqrt{V}( y )-\sqrt{V}(x)\vert^{2p})\\
 \leqslant & 2^{2p}p(V^{p-1/2}(x)  [\sqrt{V}]_1 \vert y-x\vert+[\sqrt{V}]_1^{2p}\vert y-x\vert^{2p}).
\end{align*}
When $p \in [0,1/2 ]$. We notice that from (\ref{hyp:Lyapunov_control_MS}), the function $V^p$ is $\alpha$-H\"older for every $\alpha \in [2p,1]$ (see Lemma 3. in \cite{Panloup_2008}) and then $V^{p}$ is $2p$-H\"older that is
\begin{align*}
\vert V^{p}( y )-V^{p}(x)\vert \leqslant & [\sqrt{V}]_{2p}   \vert y-x \vert^{2p}.
\end{align*}
We focus on the case $p \in[1/2,1]$. When $p\leqslant 1/2$ the proof is similar and left to the reader. Using (\ref{eq:preuve_growth_control_expo_decomp_MS}), we derive from the H\"older inequality that
\begin{align*}
\mathbb{E}\big[\vert \exp( & \lambda/s V^p(\overline{X}_{\Gamma_{n+1}}))- \exp(\lambda/s V^p(\overline{X}_{\Gamma_{n}}))  \vert^{\rho} \vert \overline{X}_{\Gamma_n}, \zeta_{\Gamma_n}\big] \\
\leqslant & C \exp(\lambda \rho /s V^p(\overline{X}_{\Gamma_{n}})) \Big( V^{p \rho-\rho/2}(\overline{X}_{\Gamma_{n}}) \mathbb{E}\big[ \vert \overline{X}_{\Gamma_{n+1}} -\overline{X}_{\Gamma_n} \vert^{\rho} \vert \overline{X}_{\Gamma_n}, \zeta_{\Gamma_n}\big] \\
& \qquad \qquad \qquad \qquad \qquad \qquad \qquad+ \mathbb{E}\big[ \vert \overline{X}_{\Gamma_{n+1}} -\overline{X}_{\Gamma_n} \vert^{2p\rho} \vert \overline{X}_{\Gamma_n}, \zeta_{\Gamma_n}\big] \Big)\\
&+ C \mathbb{E}\Big[  \exp(\lambda \rho /s V^p(\overline{X}_{\Gamma_{n+1}})) \big( V^{p \rho-\rho/2}(\overline{X}_{\Gamma_{n}})  \vert \overline{X}_{\Gamma_{n+1}} -\overline{X}_{\Gamma_n} \vert^{\rho} \\
& \qquad \qquad \qquad \qquad  \qquad \qquad \qquad + \vert \overline{X}_{\Gamma_{n+1}} -\overline{X}_{\Gamma_n} \vert^{2p\rho}\big) \Big\vert \overline{X}_{\Gamma_n}, \zeta_{\Gamma_n}\Big] \\
\leqslant &  C \exp(\lambda \rho /s V^p(\overline{X}_{\Gamma_{n}})) \Big( V^{p \rho-\rho/2}(\overline{X}_{\Gamma_{n}}) \mathbb{E}\big[ \vert \overline{X}_{\Gamma_{n+1}} -\overline{X}_{\Gamma_n} \vert^{\rho} \vert \overline{X}_{\Gamma_n}, \zeta_{\Gamma_n}\big]  \\
& \qquad \qquad \qquad \qquad \qquad \qquad \qquad + \mathbb{E}\big[ \vert \overline{X}_{\Gamma_{n+1}} -\overline{X}_{\Gamma_n} \vert^{2p\rho} \vert \overline{X}_{\Gamma_n}, \zeta_{\Gamma_n}\big] \Big)\\
&+ CV^{p \rho-\rho/2}(\overline{X}_{\Gamma_{n}}) \mathbb{E}\big[  \exp(\lambda \rho \theta /s V^p(\overline{X}_{\Gamma_{n+1}}))  \vert \overline{X}_{\Gamma_n}, \zeta_{\Gamma_n}]^{1/\theta} \\
& \qquad \qquad \qquad \qquad \qquad \times \mathbb{E}[ \vert \overline{X}_{\Gamma_{n+1}} -\overline{X}_{\Gamma_n} \vert^{\rho \theta/(\theta-1)} \vert \overline{X}_{\Gamma_n}, \zeta_{\Gamma_n}\big]^{(\theta-1)/\theta} \\
&+ C  \mathbb{E}\big[  \exp(\lambda \rho \theta /s V^p(\overline{X}_{\Gamma_{n+1}}))  \vert \overline{X}_{\Gamma_n}, \zeta_{\Gamma_n}]^{1/\theta}  \\
& \qquad \qquad \qquad \qquad \qquad \times  \mathbb{E}[ \vert \overline{X}_{\Gamma_{n+1}} -\overline{X}_{\Gamma_n} \vert^{2p \rho \theta/(\theta-1)} \vert \overline{X}_{\Gamma_n}, \zeta_{\Gamma_n}\big]^{(\theta-1)/\theta},
\end{align*}
for every $\theta>1$. From (\ref{hyp:incr_lyapunov_X_MS_expo}) and since $\rho<s$, we take $\theta \in (1,\rho/s]$ and we get
\begin{align*}
\mathbb{E}\big[  \exp(\lambda \rho \theta /s V^p(\overline{X}_{\Gamma_{n+1}})  \vert \overline{X}_{\Gamma_n}, \zeta_{\Gamma_{n}}\big] \leqslant & C \exp( \lambda \theta \rho  /s V^p(\overline{X}_{\Gamma_{n}}, \zeta_{\Gamma_n}))   .
\end{align*}
Rearranging the terms and  since $\rho < s$, we conclude from $\mathfrak{B}(\phi)$ (see (\ref{hyp:controle_coefficients_MS})) that
\begin{align*}
\mathbb{E}[\vert \exp(   \lambda/s V^p(\overline{X}_{\Gamma_{n+1}}))- & \exp(\lambda/s V^p(\overline{X}_{\Gamma_{n+1}}))  \vert^{\rho} \vert \overline{X}_{\Gamma_{n+1}}, \zeta_{\Gamma_n}]   \\
\leqslant & C   \exp(\lambda \rho /s V^p(\overline{X}_{\Gamma_{n}}))   \big( \gamma_n^{\rho/2} V^{p \rho-\rho/2}(\overline{X}_{\Gamma_{n}}) \vert \phi \circ V  (\overline{X}_{\Gamma_{n}} ) \vert^{\rho/2} \\
& \qquad \qquad \qquad \qquad \qquad \qquad \qquad + \gamma_n^{p\rho} \vert \phi \circ V  (\overline{X}_{\Gamma_{n}} ) \vert^{p \rho} \big)\\
 \leqslant & C  \gamma_n^{\rho/2} \frac{ \phi \circ V  (\overline{X}_{\Gamma_{n}} )}{V(\overline{X}_{\Gamma_{n}})}  \exp(\lambda V^p(\overline{X}_{\Gamma_{n}}))   ,
\end{align*}
and the proof of (\ref{eq:preuve_growth_control_expo_decomp_MS}) is completed.
\end{proof}
\subsubsection{Proof of Theorem \ref{th:cv_was_MS}}

This result follows from Theorem \ref{th:tightness} and Theorem \ref{th:identification_limit}. The proof consists in showing that the assumptions from those theorems are satisfied.

\paragraph{Step 1. Mean reverting recursive control}
First, we show that $\mathcal{RC}_{Q,V}(\psi_p,\phi,p\tilde{\alpha},p\beta)$ (see (\ref{hyp:incr_sg_Lyapunov})) is satisfied for every $\tilde{\alpha} \in (0,\alpha)$.\\

 Since (\ref{hyp:Lyapunov_control_MS}), $\mathfrak{B}(\phi)$ (see (\ref{hyp:controle_coefficients_MS})) and $\mathcal{R}_p(\alpha,\beta,\phi,V)$ (see (\ref{hyp:recursive_control_param_MS})) hold, it follows from Proposition \ref{prop:recursive_control_MS} that $\mathcal{RC}_{Q,V}(\psi_p,\phi,p\tilde{\alpha},p\beta)$ (see (\ref{hyp:incr_sg_Lyapunov})) is satisfied for every $\tilde{\alpha} \in (0,\alpha)$ since $\liminf_{y \to + \infty} \phi(y) > \beta / \tilde{\alpha}$

\paragraph{Step 2. Step weight assumption} 
Now, we show that $\mathcal{S}\mathcal{W}_{\mathcal{I}, \gamma,\eta}(V^{p+a-1} ,\rho,\epsilon_{\mathcal{I}}) $ (see (\ref{hyp:step_weight_I_gen_chow})) and $\mathcal{S}\mathcal{W}_{\mathcal{II},\gamma,\eta}(V^{p+a-1}) $ (see (\ref{hyp:step_weight_I_gen_tens})) hold. \\

First we recall that $\mathcal{RC}_{Q,V}(\psi_p,\phi,p\tilde{\alpha},p\beta)$ (see (\ref{hyp:incr_sg_Lyapunov})) is satisfied for every $\tilde{\alpha} \in (0,\alpha)$. Then, using $\mathcal{S}\mathcal{W}_{\mathcal{I}, \gamma,\eta}(\rho, \epsilon_{\mathcal{I}})$ (see (\ref{hyp:step_weight_I})) with Lemma \ref{lemme:mom_V} gives $\mathcal{S}\mathcal{W}_{\mathcal{I}, \gamma,\eta}( V^{p+a-1},\rho,\epsilon_{\mathcal{I}}) $ (see (\ref{hyp:step_weight_I_gen_chow})). Similarly, $\mathcal{S}\mathcal{W}_{\mathcal{II},\gamma,\eta}(V^{p+a-1}) $ (see (\ref{hyp:step_weight_I_gen_tens}) follows from $\mathcal{S}\mathcal{W}_{\mathcal{II},\gamma,\eta} $ (see (\ref{hyp:step_weight_II})) and  Lemma \ref{lemme:mom_V}. 

\paragraph{Step 3. Growth control assumption}
Now, we prove $\mathcal{GC}_{Q}(F,V^{a +p-1},\rho,\epsilon_{\mathcal{I}}) $ (see (\ref{hyp:incr_X_Lyapunov})) for $F= \DomA_0$ and $F=\{V^{p/s}\}$ .\\

This is a consequence of Lemma \ref{lemme:incr_lyapunov_X_MS}. We notice that $\rho\leqslant 2p$ and $ \rho/s \leqslant 1$. Consequently $M_{(\rho/2)\vee(p\rho /s) }(U)$ (see (\ref{hyp:moment_ordre_p_va_schema_MS})) follows from $M_{p}(U)$.  Now, we notice that Lemma \ref{lemme:incr_lyapunov_X_MS} and the fact that under $\mathfrak{B}(\phi)$ (see (\ref{hyp:controle_coefficients_MS})) and $p \geqslant 1$, we have $\Tr[ \sigma \sigma^{\ast} ] \leqslant C V^{p+a-1} $, imply that for $F= \DomA_0$ and $F=\{V^{p/s}\}$, then $\mathcal{GC}_{Q}(F,V^{a +p-1},\rho,\epsilon_{\mathcal{I}}) $ (see (\ref{hyp:incr_X_Lyapunov})) holds

\paragraph{Step 4. Conclusion}
\begin{enumerate}[label=\textbf{\roman*.}]
\item
The first part of Theorem \ref{th:cv_was_MS} (see (\ref{eq:tightness_MS})) is a consequence of Theorem \ref{th:tightness}. Let us observe that assumptions from Theorem \ref{th:tightness} indeed hold. \\

On the one hand, we observe that from Step 2. and Step 3. the assumptions $\mathcal{GC}_{Q}(V^{p/s},V^{a +p-1},\rho,\epsilon_{\mathcal{I}}) $ (see (\ref{hyp:incr_X_Lyapunov})), $\mathcal{S}\mathcal{W}_{\mathcal{I}, \gamma,\eta}( V^{p+a-1},\rho,\epsilon_{\mathcal{I}}) $ (see (\ref{hyp:step_weight_I_gen_chow})) and $\mathcal{S}\mathcal{W}_{\mathcal{II},\gamma,\eta}(V^{p+a-1}) $ (see (\ref{hyp:step_weight_I_gen_tens})) hold which are the hypothesis from Theorem \ref{th:tightness} point \ref{th:tightness_point_A} with $g=V^{p+a-1}$.\\

On the other hand, form Step 1. the assumption$\mathcal{RC}_{Q,V}(\psi_p,\phi,p\tilde{\alpha},p\beta)$ (see (\ref{hyp:incr_sg_Lyapunov})) is satisfied for every $\tilde{\alpha} \in (0,\alpha)$. Moreover, since $\mbox{L}_{V}$ (see (\ref{hyp:Lyapunov})) holds and  that $p/s+a-1>0$, then the hypothesis  from Theorem \ref{th:tightness} point \ref{th:tightness_point_B} are satisfied. \\

We thus conclude from Theorem \ref{th:tightness} that $(\nu_n^{\eta})_{n \in \mathbb{N}^{\ast}}$ (built with $(\overline{X}_t)_{t \geqslant 0}$ defined in (\ref{eq:MS_scheme})) is $\mathbb{P}-a.s.$ tight and (\ref{eq:tightness_MS}) holds which concludes the proof of the first part of Theorem \ref{th:cv_was_MS}. \\

\item Let us now prove the second part of Theorem \ref{th:cv_was_MS} (see (\ref{eq:cv_was_MS})) which is a consequence of Theorem \ref{th:identification_limit}.\\

On the one hand,we observe that from Step 2. and Step 3. the assumptions $\mathcal{GC}_{Q}(\DomA_0,V^{a +p-1},\rho,\epsilon_{\mathcal{I}}) $ (see (\ref{hyp:incr_X_Lyapunov})) and $\mathcal{S}\mathcal{W}_{\mathcal{I}, \gamma,\eta}( V^{p+a-1},\rho,\epsilon_{\mathcal{I}}) $ (see (\ref{hyp:step_weight_I_gen_chow})) hold which are the hypothesis from Theorem \ref{th:identification_limit} point \ref{th:identification_limit_A} with $g=V^{p+a-1}$.\\

On the other hand, since $z \in \{1,\ldots,M_0\}$, $b(.,z)$ and $\sigma(.,z)$ have sublinear growth and $\Tr[\sigma \sigma^{\ast}]\leqslant C  V^{p/s+a-1}$, so that $\mathbb{P} \mbox{-a.s.} \;\sup_{n \in \mathbb{N}^{\ast}} \nu_n^{\eta}( \Tr[\sigma \sigma^{\ast}] ) < + \infty $,  it follows from Proposition \ref{prop:MS_infinitesimal_approx} that $\mathcal{E}(\widetilde{A},A,\DomA_0) $ (see (\ref{hyp:erreur_tems_cours_fonction_test_reg})) is satisfied. Then, the hypothesis from Theorem \ref{th:identification_limit} point \ref{th:identification_limit_B} hold and (\ref{eq:cv_was_MS}) follows from (\ref{eq:test_function_gen_cv}).

\end{enumerate}

\subsubsection{Proof of Theorem \ref{th:cv_exp_MS}}

This result follows from Theorem \ref{th:tightness} and Theorem \ref{th:identification_limit}. The proof consists in showing that the assumptions from those theorems are satisfied.

\paragraph{Step 1. Mean reverting recursive control}

First, we show that for every $\tilde{\alpha} \in (0,\alpha)$, there exists $\tilde{\beta} \in \mathbb{R}_+$ such that $\mathcal{RC}_{Q,V}(\tilde{\psi},\phi,p\tilde{\alpha},p\tilde{\beta})$ (see (\ref{hyp:incr_sg_Lyapunov})) is satisfied for every function $\tilde{\psi}: [v_{\ast},\infty) \to \mathbb{R}_+$ such that $\tilde{\psi}(y)=  \exp( \tilde{\lambda} V^p)$ with $\tilde{\lambda} \leqslant \lambda$. Notice that this property and the fact that $\phi$ has sublinear growth imply (\ref{hyp:incr_lyapunov_X_MS_expo}).\\

We begin by noticing that $\mathcal{R}_{p, \lambda}(\alpha,\beta,\phi,V)$ (see (\ref{hyp:recursive_control_param_MS_expo})) implies $\mathcal{R}_{p,\tilde{ \lambda}}(\alpha,\beta,\phi,V)$ for every $\tilde{\lambda} \leqslant \lambda$. Since (\ref{hyp:Lyapunov_control_MS}), $\mathfrak{B}(\phi)$ (see (\ref{hyp:controle_coefficients_MS})), $\mathcal{R}_{p, \lambda}(\alpha,\beta,\phi,V)$  (see (\ref{hyp:recursive_control_param_MS_expo})) and (\ref{hyp:dom_recurs_MS}) hold, it follows from Proposition \ref{prop:recursive_control_MS_exp} with $\lim_{y \to +\infty} \phi(y)=+\infty$, that that for every $\tilde{\alpha} \in (0,\alpha)$, there exists $\tilde{\beta} \in \mathbb{R}_+$ such that $\mathcal{RC}_{Q,V}(\tilde{\psi},\phi,p\tilde{\alpha},p\tilde{\beta})$ (see (\ref{hyp:incr_sg_Lyapunov})) is satisfied for every function $\tilde{\psi}: [v_{\ast},\infty) \to \mathbb{R}_+$ such that $\tilde{\psi}(y)=  \exp( \tilde{\lambda} V^p)$ with $\tilde{\lambda} \leqslant \lambda$. 

\paragraph{Step 2. Step weight assumption} 
Now, we show that $\mathcal{S}\mathcal{W}_{\mathcal{I}, \gamma,\eta}(V^{-1}. \phi \circ V  .\exp(\lambda  V^{p}) ,\rho,\tilde{\epsilon}_{\mathcal{I}}) $, $\mathcal{S}\mathcal{W}_{\mathcal{I}, \gamma,\eta}(V^{-1}. \phi \circ V  .\exp(\lambda  V^{p}) ,\rho,\epsilon_{\mathcal{I}}) $ (see (\ref{hyp:step_weight_I_gen_chow})) and $\mathcal{S}\mathcal{W}_{\mathcal{II},\gamma,\eta}(\exp(\lambda /s V^{p})) $ (see (\ref{hyp:step_weight_I_gen_tens})) hold. \\

First we recall that that there exists $\tilde{\alpha} \in (0,\alpha)$ and $\tilde{\beta} \in \mathbb{R}_+$ such that $\mathcal{RC}_{Q,V}(\psi,\phi,\tilde{\alpha},\tilde{\beta})$ (see (\ref{hyp:incr_sg_Lyapunov})) is satisfied. Then, using $\mathcal{S}\mathcal{W}_{\mathcal{I}, \gamma,\eta}(\rho, \tilde{\epsilon}_{\mathcal{I}})$ and $\mathcal{S}\mathcal{W}_{\mathcal{I}, \gamma,\eta}(\rho, \epsilon_{\mathcal{I}})$ (see (\ref{hyp:step_weight_I})) with Lemma \ref{lemme:mom_V} gives $\mathcal{S}\mathcal{W}_{\mathcal{I}, \gamma,\eta}( V^{-1}. \phi \circ V  .\exp(\lambda  V^{p}),\rho,\tilde{\epsilon}_{\mathcal{I}}) $ and $\mathcal{S}\mathcal{W}_{\mathcal{I}, \gamma,\eta}( V^{-1}. \phi \circ V  .\exp(\lambda  V^{p}),\rho,\epsilon_{\mathcal{I}}) $ (see (\ref{hyp:step_weight_I_gen_chow})). Similarly, $\mathcal{S}\mathcal{W}_{\mathcal{II},\gamma,\eta}(V^{-1}. \phi \circ V  .\exp(\lambda  V^{p})) $ (see (\ref{hyp:step_weight_I_gen_tens}) follows from $\mathcal{S}\mathcal{W}_{\mathcal{II},\gamma,\eta} $ (see (\ref{hyp:step_weight_II})) and  Lemma \ref{lemme:mom_V}. 

\paragraph{Step 3. Growth control assumption}
Now, we prove $\mathcal{GC}_{Q}(F,V^{-1}. \phi \circ V  .\exp(\lambda  V^{p}),\rho,\epsilon_{\mathcal{I}}) $ (see (\ref{hyp:incr_X_Lyapunov})) for $F= \DomA_0$ and $F=\{\exp(\lambda /s V^{p})\}$ .\\

This is a consequence of Lemma \ref{lemme:incr_lyapunov_X_MS} and Lemma \ref{lemme:incr_lyapunov_X_MS_expo}. We notice indeed that $\mathfrak{B}(\phi)$ (see (\ref{hyp:controle_coefficients_MS})) gives $\Tr[ \sigma \sigma^{\ast} ]^{\rho/2}  \leqslant (\phi \circ V)^{ \rho} $. Moreover, we have already shown that (\ref{hyp:incr_lyapunov_X_MS_expo}) is satisfied in Step 1. These observations combined with (\ref{eq:incr_lyapunov_X_MS_f_tens_expo}) imply that $\mathcal{GC}_{Q}(\DomA_0,V^{-1}  \phi \circ V \exp(\lambda V^p) ,\rho,\epsilon_{\mathcal{I}}) $  and $\mathcal{GC}_{Q}(\exp(\lambda /s V^{p}),V^{-1}.  \phi \circ V . \exp(\lambda V^p) ,\rho,\tilde{\epsilon}_{\mathcal{I}}) $ (see (\ref{hyp:incr_X_Lyapunov})) hold.

\paragraph{Step 4. Conclusion}
\begin{enumerate}[label=\textbf{\roman*.}]
\item
The first part of Theorem \ref{th:cv_exp_MS} (see (\ref{eq:tightness_MS_expo})) is a consequence of Theorem \ref{th:tightness}. Let us observe that assumptions from Theorem \ref{th:tightness} indeed hold. \\

On the one hand, we observe that from Step 2. and Step 3. the assumptions $\mathcal{GC}_{Q}(\exp(\lambda /s V^{p}),V^{-1}  \phi \circ V \exp(\lambda V^p) ,\rho,\tilde{\epsilon}_{\mathcal{I}}) $ (see (\ref{hyp:incr_X_Lyapunov})), $\mathcal{S}\mathcal{W}_{\mathcal{I}, \gamma,\eta}( V^{-1}  \phi \circ V \exp(\lambda V^p) ,\rho,\tilde{\epsilon}_{\mathcal{I}}) $ (see (\ref{hyp:step_weight_I_gen_chow})) and $\mathcal{S}\mathcal{W}_{\mathcal{II},\gamma,\eta}( V^{-1}  \phi \circ V \exp(\lambda V^p)) $ (see (\ref{hyp:step_weight_I_gen_tens})) hold which are the hypothesis from Theorem \ref{th:tightness} point \ref{th:tightness_point_A} with $g=V^{-1}  \phi \circ V \exp(\lambda V^p)$.\\

On the other hand, form Step 1. for every $\tilde{\alpha} \in (0,\alpha)$, there exists $\tilde{\beta} \in \mathbb{R}_+$ such that $\mathcal{RC}_{Q,V}(\psi,\phi,p\tilde{\alpha},p\tilde{\beta})$ (see (\ref{hyp:incr_sg_Lyapunov})) is satisfied. Moreover, since $\mbox{L}_{V}$ (see (\ref{hyp:Lyapunov})) holds, then the hypothesis  from Theorem \ref{th:tightness} point \ref{th:tightness_point_B} are satisfied. \\

We thus conclude from Theorem \ref{th:tightness} that $(\nu_n^{\eta})_{n \in \mathbb{N}^{\ast}}$ (built with $(\overline{X}_t)_{t \geqslant 0}$ defined in (\ref{eq:MS_scheme})) is $\mathbb{P}-a.s.$ tight and (\ref{eq:tightness_MS_expo}) holds
which concludes the proof of the first part of Theorem \ref{th:cv_exp_MS}. \\

\item Let us now prove the second part of Theorem \ref{th:cv_exp_MS} (see (\ref{eq:cv_expo_MS})) which is a consequence of Theorem \ref{th:identification_limit}.\\

On the one hand,we observe that from Step 2. and Step 3. the assumptions $\mathcal{GC}_{Q}(\DomA_0,V^{-1}  \phi \circ V \exp(\lambda V^p),\rho,\epsilon_{\mathcal{I}}) $ (see (\ref{hyp:incr_X_Lyapunov})) and $\mathcal{S}\mathcal{W}_{\mathcal{I}, \gamma,\eta}( V^{-1}  \phi \circ V \exp(\lambda V^p),\rho,\epsilon_{\mathcal{I}}) $ (see (\ref{hyp:step_weight_I_gen_chow})) hold which are the hypothesis from Theorem \ref{th:identification_limit} point \ref{th:identification_limit_A} with $g=V^{-1}  \phi \circ V \exp(\lambda V^p)$.\\

On the other hand, since $z \in \{1,\ldots,M_0\}$, $b(.,z)$ and $\sigma(.,z)$ have sublinear growth and $\Tr[\sigma \sigma^{\ast}]\leqslant C V^{-1}  \phi \circ V \exp(\lambda/s V^p)$, so that $\mathbb{P} \mbox{-a.s.} \;\sup_{n \in \mathbb{N}^{\ast}} \nu_n^{\eta}( \Tr[\sigma \sigma^{\ast}] ) < + \infty $,  it follows from Proposition \ref{prop:MS_infinitesimal_approx} that $\mathcal{E}(\widetilde{A},A,\DomA_0) $ (see (\ref{hyp:erreur_tems_cours_fonction_test_reg})) is satisfied. Then, the hypothesis from Theorem \ref{th:identification_limit} point \ref{th:identification_limit_B} hold and (\ref{eq:cv_expo_MS}) follows from (\ref{eq:test_function_gen_cv}).

\end{enumerate}

\bibliography{Biblio_these}
\bibliographystyle{plain}

\end{document}